\documentclass[11pt]{article}
\usepackage{relsize}
\pdfoutput=1
\usepackage{color}
\usepackage{amsmath}
\usepackage{comment}
\usepackage{amsmath,amsthm,amssymb}
\usepackage{amsfonts}
\usepackage{breqn}
\usepackage{lmodern,amssymb}
\usepackage{enumerate}
\usepackage{appendix}
\usepackage{lmodern,amssymb}
\usepackage{hyperref}

\usepackage{graphicx}
\usepackage{cancel}

\newtheorem{thm}{Theorem}[section]

\newtheorem{theorem}[thm]{Theorem}
\newtheorem{remark}[thm]{Remark}
\newtheorem{question}[thm]{Question}

\newtheorem{lemma}[thm]{Lemma}
\newtheorem{proposition}[thm]{Proposition}
\newtheorem{example}[thm]{Example}
\theoremstyle{definition}

\numberwithin{equation}{section}

\def\p{\partial}

\def\R{\mathbb{R}}

 \allowdisplaybreaks
\begin{document}

	\title{Sharp Asymptotic Stability of Blasius Profile in the Steady Prandtl Equation}
	
	\author{Hao Jia \footnotemark[1]\ \footnotemark[3] \and Zhen Lei \footnotemark[2]\ \footnotemark[4]\and Cheng Yuan  \footnotemark[2]\ \footnotemark[5]}
	\renewcommand{\thefootnote}{\fnsymbol{footnote}}
	\footnotetext[1]{University of Minnesota.}
	\footnotetext[2]{School of Mathematical Sciences; LMNS and Shanghai Key Laboratory for Contemporary Applied Mathematics, Fudan University, Shanghai 200433, P. R.China.}
	\footnotetext[3]{Email: \url{jia@umn.edu}}
	 \footnotetext[4]{Email: \url{zlei@fudan.edu.cn}}
	\footnotetext[5]{Email: \url{cyuan22@m.fudan.edu.cn}}
	
	\date{\today}
	
	\maketitle
\begin{abstract}
This work presents an asymptotic stability result concerning the self-similar Blasius profiles $[\bar{u}, \bar{v}]$ of the stationary Prandtl boundary layer equation. Initially demonstrated by Serrin \cite{MR0282585}, the profiles $[\bar{u}, \bar{v}]$ were shown to act as a self-similar attractor of solutions $[u, v]$ to the Prandtl equation through the use of von Mises transform and maximal principle techniques. Specifically, as $x \to \infty$, $\|u - \bar{u}\|_{L^{\infty}_{y}} \to 0$. Iyer \cite{MR4097332} employed refined energy methods to derive an explicit convergence rate for initial data close to Blasius. Wang and Zhang \cite{MR4657422} utilized barrier function methods, removing smallness assumptions but imposing stronger asymptotic conditions on the initial data. It was suggested that the optimal convergence rate should be $\|u-\bar{u}\|_{L^{\infty}_{y}}\lesssim (x+1)^{-\frac{1}{2}}$, treating the stationary Prandtl equation as a 1-D parabolic equation in the entire space.

In this study, we establish that $\|u - \bar{u}\|_{L^{\infty}_{y}} \lesssim (x+1)^{-1}$. Our proof relies on discovering nearly conserved low-frequency quantities and inherent degenerate structures at the boundary, which enhance the convergence rate through iteration techniques. Notably, the convergence rate we have demonstrated is optimal. We can find special solutions of Prandtl's equation such that the convergence between the solutions and the Blasius profile is exact, represented as $ (x+1)^{-1} $.
\end{abstract}

 \tableofcontents
\linespread{1.5}

\section{Introduction}
In this article, we investigate the stationary Prandtl boundary layer equation on the quadrant:
\begin{equation}\label{generalprandtl}
\begin{cases}
uu_{x}+vu_{y}-u_{yy}+p'(x)=0,\quad (x,y)\in\R^{+}\times\R^{+};\\
u_{x}+v_{y}=0;\\
u(0,y)=u_{0}(y),\quad [u,v]|_{y=0}=0,\quad u|_{y\uparrow \infty}=U(x).
\end{cases}
\end{equation}
Here, $[u,v]$ denotes the velocity field, $p(x)$ represents the pressure, and $U(x)$ is the outer streaming speed. The functions $p(x)$ and $U(x)$ adhere to Bernoulli's law:
\begin{equation*}
2p(x)+U^{2}(x)=\text{constant}.
\end{equation*}
For simplicity, we set $U\equiv 1$, leading to $p(x)\equiv \text{constant}$. Thus, \eqref{generalprandtl} simplifies to the following concise form:
\begin{equation}\label{prandtl}
\begin{cases}
uu_{x}+vu_{y}-u_{yy}=0,\quad (x,y)\in\R^{+}\times\R^{+};\\
u_{x}+v_{y}=0;\\
u(0,y)=u_{0}(y),\quad [u,v]|_{y=0}=0,\quad u|_{y\uparrow \infty}=1.
\end{cases}
\end{equation}
\subsection{Physical Backgroud}
Prandtl's boundary layer equation holds significant importance in mathematical fluid mechanics. It serves to characterize the behavior of a parallel outer flow $\vec{U}=[U(x),0]$ passing a solid body through a fluid with small viscosity, representing a fundamental problem in fluid mechanics. In our specific scenario, $\vec{U}=[1,0]$, a solution derived from the Euler equations for an ideal fluid. In the context of fluids with low viscosity, the absence of sliding at the solid body's surface primarily influences the flow in a thin layer adjacent to the body, known as the \emph{boundary layer}.

The groundwork for boundary layer theory was established in 1904 by Prandtl \cite{Pr1904}. He identified two distinct regions within flows: the main flow region, where viscosity can be disregarded, corresponding to the inviscid solution, and the thin boundary layer near the body's surface, where viscosity plays a crucial role. The mathematical transition from the boundary layer to the outer flow is encapsulated by Prandtl's equation \eqref{generalprandtl}.

In our specific case, with $\vec{U}=[1,0]$, a special stable solution to the Euler equations for an ideal fluid with constant pressure, the general Prandtl's equation \eqref{generalprandtl} simplifies to the more concise form given by equation \eqref{prandtl}.
\subsection{Solutions of Prandtl's Equations}
Let us revisit the von Mises transform, a method that converts equation \eqref{prandtl} into a degenerate parabolic equation. We introduce new variables as follows:
\begin{equation}\label{von}
X=x;\quad \psi=\psi(x,y;u)=\int_{0}^{y}u(x,s)ds.
\end{equation}
When $ u $ remains positive except at $ y = 0 $, the von Mises transform becomes invertible. The inverse transformation is represented by
\begin{equation}\label{invervon}
x=X;\quad y=y(X,\psi;u).
\end{equation}
We define a new unknown function:
\begin{equation}\label{unknown}
w(X,\psi;u)=u^{2}\left(x,y(X,\psi;u)\right).
\end{equation}
Through direct calculations, Prandtl's equation \eqref{prandtl} is converted under the von Mises transform into the following quasilinear parabolic equation:
\begin{equation}\label{vonprandtl}
\begin{cases}
w_{X}-\sqrt{w}w_{\psi\psi}=0;\\
w(X,0)=0,\quad w(X,\psi)|_{\psi\uparrow\infty}=1;\\
w(0,\psi)=w_{0}(\psi),\quad \displaystyle w_{0}\left(\int_{0}^{y}u_{0}(s)ds\right)=\left(u_{0}(y)\right)^{2}.
\end{cases}
\end{equation}

Building upon the von Mises transform, Oleinik established the existence and uniqueness of classical solutions for Prandtl's equation \eqref{prandtl} (refer to \cite{Oleinik99}).

\begin{theorem}[\bf Oleinik]\label{oleinik}
Assume that the initial data $ u_{0}(y) $ satisfies for some $\alpha\in(0,1)$:
\begin{equation*}
\begin{aligned}
&u_{0}(y)\in C^{2,\alpha}([0,\infty)),\quad u_{0}(0)=0,\quad u_{0}(y)>0\ \text{for}\ y>0;\\
&u_{0}'(0)>0,\quad u_{0}''(y)=O(y^{2}),\quad \lim_{y\to\infty}u_{0}(y)=1.
\end{aligned}
\end{equation*}
Then the Prandtl equation \eqref{prandtl} admits a unique global classical solution $ [u,v] $ satisfying
\begin{equation*}
u(x,y)>0\ \text{for}\ y>0,\quad u\in C^{1}(\R^{+}\times \R^{+}),\quad v, u_{yy}\in C^{0}(\R^{+}\times \R^{+}),
\end{equation*}
and for any $ X_{0}>0 $:
\begin{itemize}
\item[(1)] $ u, u_{y}, u_{yy} $ are bounded in $ [0,X_{0}]\times \R^{+}$;
\item[(2)] $ u_{x}, v, v_{y} $ are locally bounded in $ [0,X_{0}]\times \R^{+}$.
\end{itemize}
\end{theorem}
In a recent study by Wang and Zhang \cite{MR4327905}, utilizing maximal principle techniques, the authors demonstrated that the solutions proposed by Oleinik in Theorem \ref{oleinik} exhibit smoothness up to the boundary $y=0$ for any $x>0$, showcasing what is termed as \textit{instant smoothness}. Furthermore, Guo and Iyer's work \cite{MR4232771} established the higher regularity of Oleinik's solutions through energy methods.
Oleinik \cite{Oleinik99} also established that under favorable pressure conditions (i.e. $p'(x)\leq 0$), global solutions for \eqref{generalprandtl} exist. In the studies referenced in \cite{MR4327905} and \cite{MR4232771}, the authors also assumed the favorable pressure condition. Conversely, for adverse pressure scenarios, only local solutions exist. The discontinuation of solutions in cases of adverse pressure is known as \textit{boundary layer separation}, a phenomenon justified mathematically in \cite{MR4028516} and \cite{SWZ21}.

In \cite{Blas1908}, Blasius (1908) introduced a notable solution to Prandtl's equations known as \textit{Blasius's self-similar solution} or the \textit{Blasius profile}. Consider a function $f(\cdot)$ that satisfies the ordinary differential equation:
\begin{equation}\label{blasiusode}
\begin{cases}
f''' + \frac{1}{2}ff'' = 0, \quad \text{in} \ \mathbb{R}^{+};\\
f(0) = f'(0) = 0, \quad f'(\infty) = 1.
\end{cases}
\end{equation}
Define the vector field on the domain $\mathbb{R}^{+} \times \mathbb{R}^{+}$ as:
\begin{equation}\label{blasius}
[\bar{u}, \bar{v}] = \left[f'(\eta), \frac{1}{2\sqrt{x+1}}\left(\eta f'(\eta) - f(\eta)\right)\right], \quad \eta = \frac{y}{\sqrt{x+1}}.
\end{equation}
Then, the pair $[\bar{u},\bar{v}]$ satisfies Prandtl's equation \eqref{prandtl} along with the specified boundary conditions.

By solving the ODE \eqref{blasiusode} rather than a PDE, we can derive insights into the Blasius profile \eqref{blasius}. This leads to a fundamental question:
\begin{question}\label{question}
Do the solutions of Prandtl's equations \eqref{prandtl} converge to the Blasius profile \eqref{blasius}, and if they do, what is the convergence rate?
\end{question}
It is crucial to note that this question plays a significant role in mathematically validating the applicability of Prandtl's boundary layer expansion for viscous flow; see \cite{MR4642817}, \cite{MR4462474}, \cite{iyer2021globalinx} for further insights.
\subsection{Results in Existing Literature}
In the work of Serrin \cite{MR0282585}, he proved that the Blasius profile \eqref{blasius} is indeed a self-similar attractor of Prandtl's system \eqref{prandtl}, which answered the first part of Question \ref{question}. More precisely, he proved that
\begin{theorem}[\bf Serrin]\label{serrin}
Let $ u(x,y) $ be the solution of \eqref{prandtl}, then the following asymptotic behaviour holds
\begin{equation}
\lim_{x\to\infty}\|u-\bar{u}\|_{L^{\infty}_{y}}=0.
\end{equation}
\end{theorem}
The proof of Serrin is based on the von Mises transform mentioned above. Let us explain the idea briefly. Define
\begin{equation}\label{phii}
\phi(X,\psi)=u^{2}\left(X,y(X,\psi;u)\right)-\bar{u}^{2}\left(X,y(X,\psi;\bar{u})\right).
\end{equation}
We can regard $ \phi $ as the \textit{twisted subtraction} between $ u $ and $ \bar{u} $, since in the expression of $ \phi $,  $ u $ and $ \bar{u} $ are evaluated at the different vertical axis in physical coordinates.
\begin{remark}In our article, we adopt the convention that $ u , \bar{u} $ are viewed simultaneously as functions of $ (x,y) $ and $ (X,\psi) $, through the von Mises transformation. Although there is a slight abuse of notation (more rigorously different symbols should be used for the same function in different coordinates), the great advantage is simplicity in notations.
\end{remark}
Substitute \eqref{phii} into Prandtl's equation \eqref{prandtl}, the twisted subtraction $ \phi $ satisfies the following quasilinear parabolic equation:
\begin{equation}\label{eqphi}
\begin{cases}
\phi_{X}-u\phi_{\psi\psi}+A\phi=0,\quad \text{in}\ \R^{+}\times\R^{+};\\
\phi(0,\psi)=\phi_{0}(\psi),\quad \phi(X,0)=\phi(X,\infty)=0\\
A(X,\psi)=\displaystyle\frac{-2\bar{u}_{yy}}{\bar{u}(u+\bar{u})}\bigg|_{(X,\psi)}.
\end{cases}
\end{equation}

Under this transformation, Serrin employed the maximal principle technique for the equation \eqref{eqphi} to obtain the result (Theorem \ref{serrin}).

Recently, Iyer \cite{MR4097332} proved an explicit decay rate of the twisted subtraction $ \phi $ and its derivations in von Mises coordinates, when the initial $ \phi_{0}(\psi) $ is small in suitable weighted Sobolev spaces with high regularities. More precisely, in \cite{MR4097332}, Iyer proved that
\begin{equation}
\begin{aligned}
&\sum_{j\leq K_{0}-1}\Bigg(\|\p_{X}^{j}\phi\|_{L^{2}_{\psi}}(X+1)^{j+\frac{1}{4}-\delta}+\|\p_{X}^{j}\phi_{\psi}\|_{L^{2}_{\psi}}(X+1)^{j+\frac{3}{4}-\delta}\\
&+\|\p_{X}^{j}\phi\|_{L^{\infty}_{\psi}}(X+1)^{j+\frac{1}{2}-\delta}+\|\p_{X}^{j}\phi_{\psi}\|_{L^{\infty}_{\psi}}(X+1)^{j+1-\delta}\Bigg)\lesssim 1,\quad \text{for}\ \delta=\frac{1}{100},
\end{aligned}
\end{equation}
provided
\begin{equation}\label{iyerassume}
\sum_{l\leq 2K_{0}}|\p_{\psi}^{l}\phi_{0}(\psi)(\psi+1)^{10}|\leq \kappa
\end{equation}
holds for $0<\kappa \ll 1$ and $K_{0}\gg 1$ relative to universal constants. In particular,
\begin{equation*}
|\rho(X,\psi)|(X+1)^{\frac{1}{2}-\delta}\lesssim 1,\quad \rho(X,\psi):=u(X,\psi)-\bar{u}(X,\psi),
\end{equation*}
and
\begin{equation*}
|u(x,y)-\bar{u}(x,y)|(x+1)^{\frac{1}{2}-\delta}\lesssim 1.
\end{equation*}

The key ingredient in \cite{MR4097332} is \emph{a three-tiered energy estimate}: basic $L^{2}$ bounds, intermediate derivatives bounds, and high derivatives bounds. Therefore, to apply this technique, the initial data $\phi_{0}(\psi)$ needs to be quite regular in \cite{MR4097332}. 

However, when we turn only to prove the $L^{\infty}$ decay rate of $u$ and $\bar{u}$, one may hope to work with a lower regularity scheme. Before stating our main result Theorem \ref{main}, we first present the following auxiliary theorem, which can be viewed as a low regularity version of \cite{MR4097332}.
\begin{theorem}\label{iyer}
Under the assumptions in Theorem \ref{oleinik}, there exists a positive constant $ \kappa \ll 1 $ such that if the initial data $ u_{0}(y) $ in \eqref{prandtl} satisfies
\begin{equation}\label{smallnorm}
\sum_{j\leq 2}\left\|(1+y)\partial_{y}^{j}(u_{0}-\bar{u}_{0})(y)\right\|_{L^{2}(\R^{+})}\leq \kappa,
\end{equation}
then the twisted subtraction $ \phi $ satisfies the following decay estimates. For $\delta=\frac{1}{100}$:
\begin{equation}\label{iyerdecay}
\begin{cases}
\|\phi(X,\psi)\|_{L^{2}_{\psi}}\lesssim (X+1)^{-(\frac{1}{4}-\delta)},\|\phi(X,\psi)\|_{L^{\infty}_{\psi}}\lesssim (X+1)^{-(\frac{1}{2}-\delta)};\\
\|\phi_{\psi}(X,\psi)\|_{L^{2}_{\psi}}\lesssim (X+1)^{-(\frac{3}{4}-\delta)},\|\phi_{\psi}(X,\psi)\|_{L^{\infty}_{\psi}}\lesssim (X+1)^{-(1-\delta)};\\
\|\phi_{X}(X,\psi)\|_{L^{2}_{\psi}}\lesssim (X+1)^{-(\frac{5}{4}-\delta)},\left\|\phi_{X\psi}(X,\psi)(X+1)^{(\frac{5}{4}-\delta)}\right\|_{L^{2}_{X,\psi}}\lesssim 1.
\end{cases}
\end{equation}
\end{theorem}
\begin{remark}
It is worth noting that the initial data assumptions in Theorem \ref{iyer} are less regular compared to those in \emph{\cite{MR4097332}}, as we only need to utilize \eqref{iyerdecay} in our current study. In addition, notice that the assumptions \eqref{iyerassume} in \emph{\cite{MR4097332}} are established in the von Mises coordinates. Since Prandtl's equation is originally formulated in physical coordinates, it is more natural to impose assumptions on the initial data within the physical coordinate system. Since the assumptions on the initial data in Theorem \ref{iyer} are weaker than those in \emph{\cite{MR4097332}}, the conclusions in Theorem \ref{iyer} cannot be derived from the results in \emph{\cite{MR4097332}}. Similar to \emph{\cite{MR4097332}}, the proof of Theorem \ref{iyer} is based on weighted energy estimate techniques and also more careful calculations. For completeness, we present a self-contained proof of Theorem \ref{iyer} in Appendix \ref{pfiyer}. The key point is that the energy estimates crucial for the proofs both here and in Iyer's work can be closed even in relatively low regularity function spaces.
\end{remark}
In the recent work \cite{MR4657422}, Wang and Zhang proved the decay rate\begin{equation}
\|u-\bar{u}\|_{L^{\infty}_{y}}\lesssim (x+1)^{-(\frac{1}{2}-)}\footnote{This is, for any $\gamma\in(0,\frac{1}{2})$,
\begin{equation*}
\|u-\bar{u}\|_{L^{\infty}_{y}}\leq C(\gamma)(x+1)^{-\gamma},\quad C(\gamma)|_{\gamma\uparrow\frac{1}{2}}=\infty.
\end{equation*}}
,
\end{equation}
without smallness assumptions but with a stronger localization condition. Specifically,
\begin{equation}\label{zhangcondition}
|u_{0}(y)-1 |\lesssim e^{-Cy^{2}}
\end{equation}
for some large $ C>0 $.  By the results on Serrin \cite{MR0282585}, one can expect that the perturbations become sufficiently small at large times, provided that suitable localization estimates can be established.

Indeed, in \cite{MR4657422}, the key estimate to construct barrier function is
\begin{equation}\label{zhangkey}
|\phi_{0}(\psi)|\lesssim \p_{\psi}\bar{w}_{0}(\psi),\quad \bar{w}=\bar{u}^{2},
\end{equation}
hence we need $C>0$ is large in \eqref{zhangcondition} due to the exponential decay of $\p_{\psi}\bar{w}_{0}$.

In \cite{MR4097332} and \cite{MR4657422}, it was suggested that the convergence rate $ (x+1)^{-(\frac{1}{2}-)} $ is nearly optimal, based on comparing with the one-dimensional heat equation. However, the following heuristic analysis suggests that the convergence rate may reach $ (x+1)^{-1} $.\\

{\bf Heuristic Analysis.} {\small Recall the equation \eqref{vonprandtl}. If $ u $ satisfies Prandtl's equation \eqref{prandtl}, then in von Mises coordinate, $ w=u^{2} $ satisfies}
\begin{equation}\label{prandtlvon}
 w_{X}-\sqrt{w}w_{\psi\psi}=0.
\end{equation}
{\small We introduce the following transformation:}
\begin{equation}\label{transform}
\begin{cases}
s = \log(X+1);\\
Y = \displaystyle\frac{\psi}{\sqrt{X+1}};\\
W(s,Y) = w(X,\psi).
\end{cases}
\end{equation}
{\small Under this transformation, \eqref{prandtlvon} transforms to}
\begin{equation}\label{prandtlnew}
\partial_{s}W - \frac{Y}{2}\partial_{Y}W - \sqrt{W}W_{YY} = 0.
\end{equation}
{\small By the definition of the von Mises transform \eqref{von}, $\bar{u}^{2}$ still exhibits a self-similar structure in the von Mises coordinates. Therefore, under the transform \eqref{transform}, $\bar{u}^{2} = \overline{W}(Y)$ for some profile $\overline{W}(\cdot)$, and $\overline{W}$ satisfies the following ordinary differential equation for $Y>0$,}
\begin{equation}\label{Weq}
\frac{Y}{2}\partial_{Y}\overline{W}(Y) + \sqrt{\overline{W}}\partial_{Y}^{2}\overline{W}(Y) = 0.
\end{equation}
{\small Consider the operator $W \rightarrow \Phi(W) := \frac{Y}{2}\partial_{Y}W + \sqrt{W}W_{YY}$. Heuristically,}
\begin{equation*}
\partial_{s}(W - \overline{W}) = \Phi(W) - \Phi(\overline{W}) \approx \Phi'(\overline{W})(W - \overline{W}).
\end{equation*}
{\small Thus, to determine the exact convergence rate of $W(s,Y)$ towards $\overline{W}(Y)$, it is important to study the linearized operator:}
\begin{equation}\label{linearized}
\mathcal{L} = -\Phi'(\overline{W}) = -\frac{Y}{2}\partial_{Y} - \frac{1}{2}\frac{\overline{W}''}{\sqrt{\overline{W}}} - \sqrt{\overline{W}}\partial_{Y}^{2}.
\end{equation}
{\small By judiciously defining the domain of $\mathcal{L}$ within suitably weighted Sobolev spaces, the convergence rate depends on the principal eigenvalue of $\mathcal{L}$. Notably,}
\begin{equation*}
\mathcal{L}(\varphi) = \varphi, \quad \text{with}\ \varphi(Y) = Y\overline{W}'(Y),
\end{equation*}
{\small which can be easily verified through \eqref{Weq}. Given the intrinsic properties of the Blasius profile, the positivity of $\varphi(Y)$ in $(0,\infty)$ shows that $\lambda = 1$ is the principal eigenvalue of $\mathcal{L}$. This can be rigorously established using arguments similar to those in Section 6.5 of \cite{MR1625845}. 

Consequently, the spectral analysis of $\mathcal{L}$ implies an exponential convergence rate of $e^{-s}$ within the $(s,Y)$ coordinates, corresponding to $(X+1)^{-1}$ within the von Mises coordinates.} \\

In this paper, we confirm the above heuristic analysis and establish the expected sharp rate given by $ (x+1)^{-1} $ as $ x \to \infty $, using additional structures the Prandtl equation possesses. We identify these special structures and use nearly conserved weighted quantities (which capture the localization property of the solutions), to establish the sharp rate of decay $ (x+1)^{-1} $. Our main result will be presented in the next subsection \ref{result}, and our strategies will be outlined in subsection \ref{idea}. The optimality of the decay rate we obtained will be explained in Example \ref{optimal}.

\subsection{Main Results}\label{result}
Our first main result is the following theorem:
\begin{theorem}\label{main}
Suppose that the initial data $ u_{0}(y) $ satisfies Oleinik's conditions in Theorem \ref{oleinik}, and
\begin{equation}\label{decayassume}
(1+y)(u_{0}-\bar{u}_{0})\in L^{1}_{y\in\R^{+}},\quad \sum_{j\leq 2}\left\|(1+y)\partial_{y}^{j}(u_{0}-\bar{u}_{0})(y)\right\|_{L^{2}(\R^{+})}\leq \kappa \ll 1.
\end{equation}
Then,  in both  physical coordinates and  von Mises coordinates, the following asymptotic stability results hold for $ X,x>0 $:
\begin{equation}\label{mainresult}
\|u(X,\psi)-\bar{u}(X,\psi)\|_{L^{\infty}_{\psi\in\R^{+}}}\lesssim (X+1)^{-1},\quad \|u(x,y)-\bar{u}(x,y)\|_{L^{\infty}_{y\in\R^{+}}}\lesssim (x+1)^{-1}.
\end{equation}
\end{theorem}
\begin{remark}\label{initialdata}
It is important to highlight that \eqref{decayassume} can be deduced from the assumptions outlined in Iyer's work \emph{\cite{MR4097332}}. The condition \eqref{decayassume} is natural since it is stated in the physical coordinate system, without the use of von Mises transformations.
\end{remark}
\begin{remark}
Clearly, \eqref{decayassume} can be verified provided the following localization condition holds for some $\kappa_{1}>0$ depending on $ \kappa $,
\begin{equation}\label{localization}
\left | (1+y)^{r_{1}} (u_{0}-\bar{u}_{0})(y) \right|\lesssim 1,\ \left|(1+y)^{r_{2}}\partial_{y}^{2}(u_{0}-\bar{u}_{0})(y)\right| \leq \kappa_{1} \ll 1, \  \emph{for}\  r_{1}>2, r_{2}> \frac{3}{2}.
\end{equation}
Furthermore, since $ u_{0}(y)\in C^{2,\alpha} $ \emph{( Oleinik's condition )} , by Sobolev interpolation inequality, \eqref{decayassume} still holds, provided that
\begin{equation}\label{localiza}
\left|(1+y)^{\Theta}(u_{0}-\bar{u}_{0})(y)\right|\leq \kappa_{2}\ll 1
\end{equation}
holds for some $ \Theta = \Theta(\alpha) $ depends on $ \alpha $.
\end{remark}

We emphasize that the convergence rate achieved in our main result (Theorem \ref{main}) is \emph{optimal}, which can be seen from the following example.
\begin{example}\label{optimal} For any $ A>0 $, one can check that
\begin{equation*}
u_{A}(x,y)=f'\left(\frac{y}{\sqrt{x+A}}\right)
\end{equation*}
is the solution of Prandtl equation \eqref{prandtl}, where $ f(\cdot) $ is the solution of Blasius equation \eqref{blasius}. Clearly, by the basic properties of the Blasius profile \emph{ (see Lemma \ref{Blasius})}, when $ A $ is close to $ 1 $, the initial data $ u_{A}(0,y) $ satisfies the assumption \eqref{decayassume}. However, when $ A\neq 1 $, and $ \frac{1}{2}\leq A\leq 2 $:
\begin{equation*}
\begin{aligned}
\left|\bar{u}(x,\sqrt{x+1})-u_{A}(x,\sqrt{x+1})\right|&=\left|f'(1)-f'\left(\frac{\sqrt{x+1}}{\sqrt{x+A}}\right)\right|\\
&\geq f''(2)\left|1-\frac{\sqrt{x+1}}{\sqrt{x+A}}\right|\geq \frac{f''(2)}{x+1}.
\end{aligned}
\end{equation*}
Hence, the convergence rate $ (x+1)^{-1} $ between the solution of the Prandtl equation and the Blasius profile can not be improved in general.

\end{example}

In our second main result, we prove the sharp convergence rate for \emph{general initial data}, without the assumption that the initial data be sufficiently close to $\bar{u}_{0}(y)$ in the sense of Theorem \ref{main}.
\begin{theorem}\label{main2}
Assume the initial data $u_{0}(y)$ of Prandtl's equation \eqref{prandtl} satisfies the assumptions in Theorem \ref{oleinik}. In addition,
\begin{equation}\label{additional}
\begin{cases}
0\leq u_{0}(y)\leq 1,\\
|u_{0}(y)-1|\lesssim \displaystyle\frac{1}{1+y^{4}},
\end{cases}
\quad \emph{for}\ y\in[0,\infty).
\end{equation}
Then the sharp convergence rates \eqref{mainresult} still hold, i.e. for $ X,x>0 $:
\begin{equation*}
\|u(X,\psi)-\bar{u}(X,\psi)\|_{L^{\infty}_{\psi\in\R^{+}}}\lesssim (X+1)^{-1},\quad \|u(x,y)-\bar{u}(x,y)\|_{L^{\infty}_{y\in\R^{+}}}\lesssim (x+1)^{-1}.
\end{equation*}
\begin{remark}Comparing with \emph{\cite{MR4657422}}, we only need the convergence of the initial data $u_{0}(y)$ to the outer flow at infinity at a rate of the reciprocal of a polynomial. Furthermore, we obtain the sharp decay rate $(x+1)^{-1}$, which is stronger than the decay rate $(x+1)^{-(\frac{1}{2}-)}$ obtained in \emph{\cite{MR4657422}}.

The additional condition $0\leq u_{0}(y)\leq 1$ ensures that we can construct suitable barrier functions. This condition covers the most physically relevant case when the flow is generated by the top outer flow and is slowed down by the viscosity near the boundary.
\end{remark}
\end{theorem}

\subsection{Main Ideas}\label{idea}
Now let us introduce the main ideas in this paper.

\textit{(a) Low frequency “conserved” quantity.} 

Inspired by \cite{MR4097332}, our approach is based on deriving the decay rate through energy estimates. Typically, we can establish the following inequality:
\begin{equation}\label{formalenergy}
\frac{d}{dX}\|\phi\|_{L^{2}_{\psi}}^{2}+\|\phi_{\psi}\|_{L^{2}_{\psi}}^{2}\leq \mathcal{E}(X),
\end{equation}
where the remainder term $ \mathcal{E}(X) $  decays in $ X $ at a certain rate.   The decay estimate in $ L^{2} $ for \eqref{eqphi} stems from \eqref{formalenergy}, although this derivation may not be straightforward. Assuming the existence of a low-frequency quantity $ \|\phi\|_{Y} $ that we can control, i.e. $ \|\phi\|_{Y} \lesssim1$ for $X\ge0$, and the validity of the following interpolation inequality:
\begin{equation}\label{interpolation}
\|\phi\|_{L^{2}_{\psi}}\lesssim \|\phi\|_{Y}^{1-s}\|\phi_{\psi}\|_{L^{2}_{\psi}}^{s},\quad 0<s<1.
\end{equation}
Substitute \eqref{interpolation} into \eqref{formalenergy}, we can deduce the differential inequality, under the condition that $ \|\phi\|_{Y} $ remains uniformly bounded over time:
\begin{equation}\label{de}
\frac{d}{dX}\|\phi\|_{L^{2}_{\psi}}^{2}+\|\phi\|_{L^{2}_{\psi}}^{\frac{2}{s}}\leq \mathcal{E}(X).
\end{equation}
Clearly, \eqref{de} indicates the decay rate of $ \|\phi\|_{L^{2}_{\psi}} $, marking the initial step towards deriving the decay rate of higher-order derivatives of $ \phi $.

The origins of this idea can be traced back at least to Nash's seminal work \cite{Nash58} on deriving the decay rate of fundamental solutions for parabolic equations. An important observation in \cite{Nash58} is the inequality (formulated here specifically to the one-dimensional case)
\begin{equation}\label{nash}
\|\Gamma\|_{L^{2}}\lesssim \|\Gamma\|_{L^{1}}^{\frac{2}{3}}\|\partial\Gamma\|_{L^{2}}^{\frac{1}{3}}.
\end{equation}
By combining the $ L^{1} $ conservation principle for the fundamental solution $ \Gamma $, \eqref{nash}, and the standard parabolic energy estimate, Nash deduced the decay rate of $ \Gamma $ in $ L^{2} $. It is worthwhile to note that the low-frequency quantity identified in \cite{Nash58} and \cite{MR4097332} only suggests a decay rate of $ (X+1)^{-\frac{1}{4}} $ for the $ L^{2} $ norm of $\phi$, resulting in a decay rate of $ (X+1)^{-\frac{1}{2}} $ for the $ L^{\infty} $ norm of $\phi$, which falls short of the anticipated decay rate $(X+1)^{-1}$.

%However, in equation \eqref{eqphi}, it is not immediately apparent how to identify a conserved or dissipated quantity with a frequency lower than $ L^{1} $. Drawing inspiration from Example \ref{optimal}, 
To improve the convergence rate further, we need to obtain stronger localization properties than just $L^1$ bounds. It is natural to consider the quantity
\begin{equation*}
I(X)=\int_{0}^{\infty}\psi^{\alpha}|\phi(X,\psi)|d\psi,
\end{equation*}
where $ \alpha $ is a positive parameter to be determined later. As $ \alpha>0 $, the quantity $ I(X) $ exhibits a lower frequency than the $ L^{1} $ norm, suggesting that the boundedness of $ I(X) $ will result in a more rapid decay rate in $ L^{2} $ than $ (X+1)^{-\frac{1}{4}} $. There are two technical obstacles:
\begin{itemize}
\item[(i)] The non-smooth nature of $ |\phi| $ prevents a direct differentiation of $ I(X) $;
\item[(ii)] It is not immediately clear whether $ I(X) $ remains finite for all $ X $.
\end{itemize}
To overcome them, we introduce a small parameter $ \epsilon>0 $ and a large parameter $ N>1 $, and consider
\begin{equation*}
I_{\epsilon,N}(X)=\int_{0}^{\infty}\psi^{\alpha}\left(\phi^{2}+\epsilon^{2}\right)^{\frac{1}{2}}\chi_{N}(\psi)d\psi,
\end{equation*}
where $ \chi_{N} $ is a standard cut-off function associated with parameter $ N $. Hence the uniform bound on $ I_{\epsilon, N}(X) $ implies the boundedness of $ I(X) $. To establish the bound of $ I_{\epsilon, N}(X) $, it is sufficient to demonstrate that for suitable $s_1,s_2>0$, and some $ f_{1},f_{2}\in L^{1}(\R^{+}) $,
\begin{equation}
I_{\epsilon,N}'(X)\lesssim o(\epsilon^{s_{1}})+o(N^{-s_{2}})+f_{1}(X)I(X)+ f_{2}(X).
\end{equation}
See Section \ref{L1} for more details.

\emph{(b) Improve the decay rate by iteration techniques.} 

Our objective is to establish the $ L^{2} $ decay rate for $ \phi $ based on the differential inequality \eqref{de}. However, if the remainder term $ \mathcal{E}(X) $ does not decay rapidly, we may not achieve the desired decay rate for $ \|\phi\|_{L^{2}_{\psi}} $. Specifically, if we expect $ \|\phi\|_{L^{2}_{\psi}}\lesssim (X+1)^{-Q} $ for $ X>0 $, then it is necessary to show that
\begin{equation}\label{decayfast}
\mathcal{E}(X)\lesssim (X+1)^{-\frac{2Q}{s}}.
\end{equation}

In the first step of our energy estimates, the evaluation of the remainder term $ \mathcal{E}(X) $ relies only on the bounds in Theorem \ref{iyer}. Consequently, initially, $ \mathcal{E}(X) $ does not satisfy the decay rate specified in \eqref{decayfast}.

To tackle this issue, it is important to recognize that $ \mathcal{E}(X) $ can be formally expressed as terms in the form of
\begin{equation}\label{polynomial}
\int_{0}^{\infty} \mathcal{Q}(\phi,\phi_{X})d\psi,
\end{equation}
involving polynomials $ \mathcal{Q} $ of at least cubic order. While it is not possible to directly achieve the desired decay rate, we can attain a decay rate faster than the initial one indicated in Theorem \ref{iyer}.

Upon establishing a decay rate faster than that stipulated in Theorem \ref{iyer}, the decay rate of $ \mathcal{E}(X) $ accelerates further due to the nonlinear structure outlined in \eqref{polynomial}. Reintegrating this improved decay rate of $ \mathcal{E}(X) $ back into \eqref{de}, we can achieve a faster decay rate for $ \|\phi\|_{L^{2}} $ compared to the previous iteration, resulting in an enhanced decay rate for $ \mathcal{E}(X)$.

By iteratively applying this process for a finite number of steps, we can ultimately attain the targeted decay rate for $ \|\phi\|_{L^{2}} $ and consequently the anticipated decay rate for the high-order energy of $ \phi $ through successive energy estimates. We refer to Section \ref{enhanced} and Section \ref{proofmain} for the detailed proof.

\emph{(c) Remove the smallness assumption.}

As in Wang and Zhang \cite{MR4657422}, the proof is based on the comparison principle of the (transformed) Prandtl equation. In Theorem \ref{main2}, we only assume the initial data $u_{0}(y)$ converges to the outer flow in a polynomial rate, hence we can not construct barrier functions as in Wang and Zhang's work \cite{MR4657422}. The main feature of our approach is to combine the scaling invariance and the result in Theorem \ref{main} to construct suitable barrier functions, which allows for much weaker assumptions on the asymptotic behavior of the initial data. We refer to subsection \ref{pfmain2} for more details.
\subsection{Structure of the Paper}

The remainder of the paper is structured as follows: In Section \ref{Preliminary}, we will introduce several auxiliary lemmas that will be frequently referenced in our proof. In Section \ref{L1}, we will establish the weighted $ L^{1} $ estimate, which is a key ingredient in improving decay rates. In Section \ref{enhanced} and Section \ref{proofmain}, we give a detailed proof of the main theorem. We prove Theorem \ref{main} in subsection \ref{pfmain1}, and we prove Theorem \ref{main2} in subsection \ref{pfmain2}. The proof is divided into several steps. Lastly, For the sake of completeness, we will give a self-contained proof of Theorem \ref{iyer} in Appendix \ref{pfiyer}.\\

\textit{Notations.} In this article, we adopt the following notation conventions:
\begin{itemize}
\item[(1)] We use $ A\lesssim B $ means $ A\leq c B $ for some absolute constant $ c $. The notation $ A\sim B $ means $ A\lesssim B $ and $ B\lesssim A $. Additionally, $ A \lesssim_{\star} B $ implies $ A\leq c B $ with $ c $ dependent on a specific quantity $ \star $, i.e., $ c=c(\star) $. 

\item[(2)] The notation $ \p_{l}\phi $ denotes the partial derivative of $ \phi $ with respect to the variable $ l $, and $ \phi_{l}=\p_{l}\phi $. 
\end{itemize}

\section{Preliminary}\label{Preliminary}
In this section, we present several preliminary results that will be used frequently.

The first lemma addresses the basic properties of the Blasius profile, some of which can be found in the existing literature. For the sake of completeness, we summarize them here.

\begin{lemma}\label{Blasius}
Assume $ f(\eta) $ is the solution of the Blasius equation \eqref{blasiusode}. Then:
\begin{itemize}
\item[(1)]
\begin{equation*}
f(\eta), f'(\eta), f''(\eta)\geq 0, \quad f'''(\eta)\leq 0.
\end{equation*}

\item[(2)] For $ 0\leq\eta\leq 1$:
\begin{equation*}
f(\eta)\sim \eta^{2},\ f'(\eta)\sim \eta,\ f''(\eta)\sim 1,\  -f'''(\eta)\sim \eta^{2}.
\end{equation*}

\item[(3)] For $ \eta\geq 1 $, there exists $ c>0 $ such that:
\begin{equation*}
f(\eta)\sim\eta,\ f'(\eta)\sim 1,\ 0\leq f''(\eta)\lesssim e^{-c\eta^{2}},\ -e^{-c\eta^{2}}\lesssim f'''(\eta)\leq 0.
\end{equation*}

\item[(4)] Particularly, the Blasius profile $ \bar{u} $ satisfies:
\begin{equation*}
-\frac{\bar{u}^{2}(x,y)}{x+1}\lesssim \bar{u}_{yy}(x,y)\leq 0.
\end{equation*}
\end{itemize}
\end{lemma}

\begin{proof}
To begin, as $ f'(\infty)=1 $ and $ f'(0)=0 $, $ f'' $ must be positive at some $ \eta_{0} $. Considering $ f'''+\frac{1}{2}ff''=0 $, we have:
\begin{equation}\label{twice}
f''(\eta)=f''(\eta_{0})\exp\left(-\frac{1}{2}\int_{\eta_{0}}^{\eta}f(\tau)d\tau\right)>0\quad {\rm for\ all}\  \eta \geq 0.
\end{equation}
Consequently:
\begin{equation*}
f'(\eta)=\int_{0}^{\eta}f''(\tau)d\tau> 0,\quad f(\eta)=\int_{0}^{\eta}f'(\tau)d\tau> 0\quad {\rm for\ all}\  \eta > 0.
\end{equation*}
Thus, $ f'''=-\frac{1}{2}ff''\leq 0 $.

By Taylor's formula:
\begin{equation*}
f(\eta)\sim\eta^{2},\quad f'(\eta)\sim \eta,\quad \text{for}\ \eta\leq 1. 
\end{equation*}
Since $ f(0)=0 $, one has $ f'''(0)=-\frac{1}{2}f(0)f''(0)=0 $. Moreover,
\begin{equation*}
\begin{aligned}
&f^{(4)}(0)=-\frac{1}{2}\big[f(0)f'''(0)+f'(0)f''(0)\big]=0,\\
&f^{(5)}(0)=-\frac{1}{2}\big[f(0)f^{(4)}(0)+2f'(0)f'''(0)+(f''(0))^{2}\big]<0.
\end{aligned}
\end{equation*}
Utilizing Taylor's formula again, one has
\begin{equation*}
-f'''(\eta)\sim\eta^{2},\quad \text{for}\ \eta\leq 1.
\end{equation*}

For the third statement, since $ f'(\infty)=1 $,
\begin{equation*}
\lim_{\eta\to \infty}\frac{f(\eta)}{\eta}=1,
\end{equation*}
 there exist $ C_{1},C_{2}>0 $ such that:
\begin{equation*}
C_{1}\eta\leq f(\eta) \leq C_{2}\eta,\quad \text{for}\ \eta\geq 1.
\end{equation*}
Hence, when $\eta\geq 1$, by \eqref{twice}
\begin{equation*}
f''(\eta)\leq f''(1)\exp\left(-\frac{1}{2}\int_{1}^{\eta}C_{1}\eta' d\eta'\right)\lesssim e^{-\frac{C_{1}}{2}\eta^{2}},
\end{equation*}
and the same exponential decay bound holds for $ f'''(\eta) $ due to Blasius equation.
\end{proof}

In the second lemma, we state an elementary but useful property that establishes a comparison principle between $u$ and $\bar{u}$ in the von Mises coordinate system. It is worth noting that a similar property was previously proved in \cite{MR4657422}, albeit with the involvement of abstract constants. The constants here are precise and hence independent of data.
\begin{lemma}\label{comparison}
Under the assumptions in Theorem \ref{main}, we have
\begin{equation*}
\frac{1}{2}\bar{u}(X,\psi)\leq u(X,\psi)\leq 2\bar{u}(X,\psi)\quad \text{in}\ \R^{+}\times\R^{+}.
\end{equation*}
Moreover,
\begin{equation*}
\frac{1}{2}y(X,\psi;\bar{u})\leq y(X,\psi;u)\leq 2y(X,\psi;\bar{u}).
\end{equation*}
\end{lemma}
\begin{proof}For convenience, let's denote $ \psi(y) := \psi(x,y;u)|_{x=0} $. Consequently, 
\begin{equation*}
\psi(y)=\int_{0}^{y}\bar{u}_{0}(s)ds+\int_{0}^{y}\left[u_{0}(s)-\bar{u}_{0}(s)\right]ds=f(y)+\int_{0}^{y}\left[u_{0}(s)-\bar{u}_{0}(s)\right]ds.
\end{equation*}
By our assumptions in Theorem \ref{main}, Sobolev embedding, and the boundary condition $ u_{0}(0)=\bar{u}_{0}(0)=0 $, we obtain
\begin{equation}\label{Lm2}
f(y)-\kappa c(y)\leq \psi(y)\leq f(y)+\kappa c(y),\quad c(y)=
\begin{cases}
y^{2},\quad y\leq 1;\\
1,\quad y\geq 1.
\end{cases}
\end{equation}
Since
\begin{equation*}
\lim_{y\to\infty}u_{0}(y)=\lim_{y\to\infty}\bar{u}_{0}(y)=1.
\end{equation*}
Hence, for $ \kappa\leq \frac{1}{20} $, there exists an absolute constant $ y_\ast \geq 1 $, dependent solely on $ \bar{u}_{0} $, such that
\begin{equation*}
\frac{9}{10} \leq u_{0}(y), \bar{u}_{0}(y) \leq \frac{11}{10},\quad {\rm for}\ y\geq y_\ast.
\end{equation*}
Hence, based on \eqref{Lm2}, if $ \psi(y)\geq f(y_\ast) +1 $, it implies $y > y_\ast$ and thus
\begin{equation*}
\frac{9}{10} \leq u(0,\psi),\bar{u}(0,\psi) \leq \frac{11}{10},\quad \psi\geq  f(y_\ast) +1.
\end{equation*}
According to the characteristics of the Blasius profile (Lemma \ref{Blasius}), the aforementioned estimate \eqref{Lm2} indicates the existence of a constant $ A $ such that
\begin{equation*}
f\big((1-A\kappa)y\big)\leq \psi(y) \leq f\big((1+A\kappa)y\big).
\end{equation*}
Thus, following the definition of the von Mises transform, there exists $ y_{\psi} $ such that
\begin{equation*}
u(0,\psi)=u_{0}(y_{\psi}),\quad \text{with}\  \frac{1}{1+A\kappa}f^{-1}(\psi)\leq y_{\psi} \leq \frac{1}{1-A\kappa}f^{-1}(\psi).
\end{equation*}
Therefore, if $ \psi \leq f(y^{*})+1 $, by selecting a small $ \kappa $ and utilizing the mean-value theorem
\begin{equation*}
\begin{aligned}
|u(0,\psi)-\bar{u}(0,\psi)|&=|u_{0}(y_{\psi})-\bar{u}_{0}\big(f^{-1}(\psi)\big)|\\
&\leq |u_{0}(y_{\psi})-\bar{u}_{0}(y_{\psi})|+|\bar{u}_{0}(y_{\psi})-\bar{u}_{0}\big(f^{-1}(\psi)\big)|\\
&\leq \|u_{0}'-\bar{u}_{0}'\|_{L^{\infty}_{y}}\times y_{\psi}+\|u_{0}'\|_{L^{\infty}_{y}}\times |y_{\psi}-f^{-1}(\psi)|\\
&\lesssim \kappa f^{-1}(\psi)\leq\frac{1}{2}\bar{u}_{0}\big(f^{-1}(\psi)\big)=\frac{1}{2}\bar{u}(0,\psi).
\end{aligned}
\end{equation*}
Therefore, when a small value of $ \kappa $ is selected, we obtain
\begin{equation*}
\frac{1}{4}\bar{u}^{2}(0,\psi)\leq u^{2}(0,\psi)\leq 4 \bar{u}^{2}(0,\psi).
\end{equation*}
Let
\begin{equation*}
 V(X,\psi)=u^{2}(X,\psi)-4\bar{u}^{2}(X,\psi) = u^{2}\left(X,y(X,\psi;u)\right)-4\bar{u}^{2}\left(X,y(X,\psi;\bar{u})\right).
 \end{equation*}
By direct calculation,
\begin{equation*}
\p_{X}V(X,\psi)-u(X,\psi)\p_{\psi}^{2}V(X,\psi)=\frac{8\bar{u}_{yy}(u-\bar{u})}{\bar{u}}\leq \frac{8\bar{u}_{yy}(u-2\bar{u})}{\bar{u}}=\frac{8\bar{u}_{yy}V}{\bar{u}(u+2\bar{u})}.
\end{equation*}
By the previous arguments, 
\begin{equation*}
V(0,\psi)\leq 0;\quad V(X,0)=0,\quad V(X,\infty)<0.
\end{equation*}
Therefore, by employing the standard parabolic maximal principle, we deduce that $ V(X,\psi)\leq 0 $, which implies $ u(X,\psi)\leq 2\bar{u}(X,\psi) $. Similarly, $ \bar{u}(X,\psi)\leq 2u(X,\psi) $, thereby concluding the proof.
\end{proof}
\begin{lemma}\label{Lm3}
Under the assumptions in Theorem \ref{main}, there exists $ 0<c<1<C $ such that for $v\in\{u,\bar{u}\}$:
\begin{equation*}
\begin{aligned}
&c\frac{\psi^{1/2}}{(X+1)^{1/4}} \leq v(X,\psi) \leq C\frac{\psi^{1/2}}{(X+1)^{1/4}},\quad \psi\leq \sqrt{X+1};\\
&c< v(X,\psi) <C,\quad \psi\geq \sqrt{X+1}.
\end{aligned}
\end{equation*}
\end{lemma}
\begin{proof}
Since
\begin{equation*}
\psi(x,y;\bar{u})=\int_{0}^{y}f\left(\frac{s}{\sqrt{x+1}}\right)ds=\sqrt{x+1}f\left(\frac{y}{\sqrt{x+1}}\right).
\end{equation*}
Therefore,
\begin{equation*}
\bar{u}(X,\psi)=f'\left[f^{-1}\left(\frac{\psi}{\sqrt{X+1}}\right)\right].
\end{equation*}
Therefore, Lemma \ref{Lm3} can be derived from the characteristics of the Blasius solution (Lemma \ref{Blasius}) and the comparability of $ u $ and $ \bar{u} $ in the von Mises coordinate system (Lemma \ref{comparison}).
\end{proof}
At last, we recall some standard inequalities.
\begin{lemma}\label{sobolev}Assume $ f $ is absolutely continuous in every bounded interval of $ \R $, then:
\begin{itemize}
    \item[(1)] Sobolev embedding $ \dot{W}^{1,p}\rightarrow \dot{C}^{\alpha} $, for $ t,s\in \R $:
    \begin{equation*}
    |f(t)-f(s)| \lesssim \|f'\|_{L^{p}}|t-s|^{\alpha},\quad \alpha=1-\frac{1}{p}.
    \end{equation*}
    \item[(2)] Agmon's inequality, for $ t\in \mathbb{R} $:
    \begin{equation*}
    |f(t)|^{2} \lesssim \|f\|_{L^{2}}\|f'\|_{L^{2}}
    \end{equation*}
    \item[(3)] Hardy's inequality,
    \begin{equation*}
    \int_{\R}\frac{|f(t)|^{2}}{|t|^{2\alpha}}dt \lesssim_{\alpha} \int_{\R}|f'(t)|^{2}|t|^{2-2\alpha}dt,\quad 0\leq \alpha<\frac{1}{2}.
    \end{equation*}
    In addition, if $ f(0)=0 $, then for $ \frac{1}{2}<\alpha\leq 1 $,
    \begin{equation*}
     \int_{\R}\frac{|f(t)|^{2}}{|t|^{2\alpha}}dx \lesssim_{\alpha} \int_{\R}|f'(t)|^{2}|t|^{2-2\alpha}dx.
     \end{equation*}
\end{itemize}
\end{lemma}
\section{Weighted $L^{1}$ Estimate}\label{L1}
In this section, we establish the following weighted $L^{1}$ estimate, which controls the \emph{low-frequency part} of $ \phi $ in \eqref{eqphi}.
\begin{lemma}\label{weight}
Under the assumptions in Theorem \ref{main}, for $ \alpha=\frac{19}{20} $ and any $ X\in\R^{+} $,  we have
\begin{equation*}
\int_{0}^{\infty}\frac{1}{u}|\phi(X,\psi)|\psi^{\alpha}d\psi \lesssim 1.
\end{equation*}

\end{lemma}
\begin{proof}First, for large $ N>1 $, we introduce the smooth cutoff function $ \chi_{N}(\psi) $ defined as
\begin{equation*}
\chi(\psi)=\chi_{N}(\psi)=
\begin{cases} 1, & 0\leq\psi\leq N,\\
0, & \psi\geq 2N,
\end{cases}
\quad |\chi'(\psi)|\lesssim N^{-1},\ |\chi''(\psi)|\lesssim N^{-2}.
\end{equation*}
For $ \epsilon>0 $ $ N>1 $ and $ \alpha\in(0,1) $, we define
\begin{equation*}
I(X)=I_{\epsilon,N}(X)=\int_{0}^{\infty}\frac{1}{u}\psi^{\alpha}\left(\phi^{2}(X,\psi)+\epsilon^{2}\right)^{\frac{1}{2}}\chi(\psi)d\psi,
\end{equation*}
and we choose $\epsilon N^{2} \leq 1$ to ensure $ I(X)<\infty $ for all $ X\in\R^{+} $.

We first explain the idea briefly. If $ I(X) $ satisfies the following differential inequality(for $\delta=\frac{1}{100}$):
\begin{equation}\label{diff}
\begin{aligned}
I'(X) \lesssim& \epsilon^{\alpha} + N^{\alpha-3/2} +\epsilon\left(N^{\frac{1}{2}+\alpha}(X+1)^{\frac{1}{2}}+(X+1)^{\frac{3}{4}+\frac{\alpha}{2}}\right)\left(1+\|\phi_{X\psi}\|_{L^{2}_{\psi}}\right)\\
&+(X+1)^{-(\frac{3}{2}-\frac{\alpha}{2}-2\delta)}+(X+1)^{(\frac{1}{4}+\frac{\alpha}{2}+\delta)}\|\phi_{X\psi}\|_{L^{2}_{\psi}}\\
&+\left[(X+1)^{-(\frac{3}{2}-\delta)}+(X+1)^{-(\frac{5}{8}-\frac{\alpha}{2})}\|\phi_{X\psi}\|_{L^{2}_{\psi}}^{\frac{1}{2}}\right]I(X).
\end{aligned}
\end{equation}
By Theorem \ref{iyer}, the terms $ (X+1)^{(\frac{1}{4}+\frac{\alpha}{2}+\delta)}\|\phi_{X\psi}\|_{L^{2}_{\psi}} $ and  $(X+1)^{-(\frac{5}{8}-\frac{\alpha}{2})}\|\phi_{X\psi}\|_{L^{2}_{\psi}}^{\frac{1}{2}}$ are integrable over $\R^{+}$.

Applying Gronwall's inequality, we obtain that
\begin{equation*}
\begin{aligned}
I(X)\lesssim& I(0)+1+\left(\epsilon^{\alpha}+N^{\alpha-\frac{3}{2}}\right)X+\epsilon\left(N^{\frac{1}{2}+\alpha}(X+1)^{\frac{3}{2}}+(X+1)^{\frac{7}{4}+\frac{\alpha}{2}}\right)\\
\leq &\int_{0}^{\infty}\frac{1}{u_{0}(\psi)}\psi^{\alpha}|\phi_{0}(\psi)|\chi(\psi)d\psi+ \epsilon\int_{0}^{\infty}\frac{1}{u_{0}(\psi)}\psi^{\alpha}\chi(\psi)d\psi\\
&\quad +1+\left(\epsilon^{\alpha}+N^{\alpha-\frac{3}{2}}\right)X+\epsilon\left(N^{\frac{1}{2}+\alpha}(X+1)^{\frac{3}{2}}+(X+1)^{\frac{7}{4}+\frac{\alpha}{2}}\right)\\
\lesssim& \|\phi_{0}\|_{L^{\infty}_{\psi}}+\|\psi^{\alpha}\phi_{0}\|_{L^{1}_{\psi}} +1+\left(\epsilon^{\alpha}+N^{\alpha-\frac{3}{2}}\right)X+\epsilon\left(N^{\frac{1}{2}+\alpha}(X+1)^{\frac{3}{2}}+(X+1)^{\frac{7}{4}+\frac{\alpha}{2}}\right)\\
\lesssim&  1+\left(\epsilon^{\alpha}+N^{\alpha-\frac{3}{2}}\right)X+\epsilon\left(N^{\frac{1}{2}+\alpha}(X+1)^{\frac{3}{2}}+(X+1)^{\frac{7}{4}+\frac{\alpha}{2}}\right).
\end{aligned}
\end{equation*}
Then Lemma \ref{weight} is directly obtained since by Fatou's lemma,
\begin{equation}\label{limit}
\int_{0}^{\infty}\frac{1}{u}|\phi(X,\psi)|\psi^{\alpha}d\psi\leq \liminf_{\epsilon\to 0,N\to\infty}I(X)\lesssim 1.
\end{equation}

In the remaining part of the proof, we focus on obtaining the desired differential inequality \eqref{diff}  for $ I(X) $. First,  by applying equation \eqref{eqphi}, we get that
\begin{equation*}
\begin{aligned}
I'(X)&=\int_{0}^{\infty}\phi\phi_{X}(\phi^{2}+\epsilon^{2})^{-\frac{1}{2}}\frac{1}{u}\psi^{\alpha}\chi(\psi)d\psi-\int_{0}^{\infty}(\phi^{2}+\epsilon^{2})^{\frac{1}{2}}\frac{u_{X}}{u^{2}}\psi^{\alpha}\chi(\psi)d\psi\\
&=\int_{0}^{\infty}\phi\phi_{\psi\psi}(\phi^{2}+\epsilon^{2})^{-\frac{1}{2}}\psi^{\alpha}\chi(\psi)d\psi\\
&\quad-\bigg(\int_{0}^{\infty}A(X,\psi)\phi^{2}(\phi^{2}+\epsilon^{2})^{-\frac{1}{2}}\frac{1}{u}\psi^{\alpha}\chi(\psi)d\psi+\int_{0}^{\infty}(\phi^{2}+\epsilon^{2})^{\frac{1}{2}}\frac{u_{X}}{u^{2}}\psi^{\alpha}\chi(\psi)d\psi\bigg)\\
&=:J+K.
\end{aligned}
\end{equation*}
\textit{Estimate of $J$.}
Integrating by parts, we obtain that
\begin{equation*}
\begin{aligned}
J=&-\int_{0}^{\infty}\phi_{\psi}\Big(\phi(\phi^{2}+\epsilon^{2})^{-\frac{1}{2}}\Big)_{\psi}\psi^{\alpha}\chi(\psi)d\psi\\
&-\alpha \int_{0}^{\infty}\phi_{\psi}\phi(\phi^{2}+\epsilon^{2})^{-\frac{1}{2}}\psi^{\alpha-1}\chi(\psi)d\psi\\
&-\int_{0}^{\infty}\phi_{\psi}\phi(\phi^{2}+\epsilon^{2})^{-\frac{1}{2}}\psi^{\alpha}\chi'(\psi)d\psi\\
=&J_{1}+J_{2}+J_{3}.
\end{aligned}
\end{equation*}	
Direct computation reveals that $ J_{1} $ is non-positive, since
\begin{equation*}
\Big(\phi(\phi^{2}+\epsilon^{2})^{-\frac{1}{2}}\Big)_{\psi}=\frac{\epsilon^{2}\phi_{\psi}}{(\phi^{2}+\epsilon^{2})^{\frac{3}{2}}}.
\end{equation*}
For $J_{2}$, since $\psi^{\alpha-1}$ has a singularity at $\psi=0$, we decompose $J_{2}$ into two parts (with a parameter $\iota>0$ to be determined below) :
\begin{equation*}
\begin{aligned}
J_{2} &= -\alpha\int_{0}^{\iota}\phi_{\psi}\phi(\phi^{2}+\epsilon^{2})^{-\frac{1}{2}}\psi^{\alpha-1}\chi(\psi)d\psi\\
&\quad -\alpha\int_{\iota}^{\infty}\phi_{\psi}\phi(\phi^{2}+\epsilon^{2})^{-\frac{1}{2}}\psi^{\alpha-1}\chi(\psi)d\psi\\
&= J_{2}^{(1)} + J_{2}^{(2)}.
\end{aligned}
\end{equation*}
$J_{2}^{(1)}$ can be estimated directly:
\begin{equation*}
J_{2}^{(1)} \lesssim \|\phi_{\psi}\|_{L^{\infty}_{\psi}}\|\psi^{\alpha-1}\|_{L^{1}(0,\iota)} \lesssim \iota^{\alpha}.
\end{equation*}
For $J_{2}^{(2)}$, integrating by parts, we get that
\begin{equation*}
\begin{aligned}
J_{2}^{(2)} &\sim \int_{\iota}^{\infty}\psi^{\alpha-1}\chi(\psi)d\Big(-(\phi^{2}+\epsilon^{2})^{\frac{1}{2}}\Big)\\
&= \iota^{\alpha-1}\left(\phi^{2}(\iota)+\epsilon^{2}\right)^{\frac{1}{2}} + \int_{\iota}^{\infty}(\phi^{2}+\epsilon^{2})^{\frac{1}{2}}\Big[\chi'(\psi)\psi^{\alpha-1}+(\alpha-1)\chi(\psi)\psi^{\alpha-2}\Big]d\psi\\
&\lesssim \iota^{\alpha} + \iota^{\alpha-1}\epsilon.
\end{aligned}
\end{equation*}
We use the facts that $\chi$ is non-increasing and $\alpha<1$. By choosing $\iota=\epsilon$, we obtain that
\begin{equation*}
J_{2} \lesssim \epsilon^{\alpha}.
\end{equation*}

The term $J_{3}$ can be estimated similarly to $J_{2}$:
\begin{equation*}
\begin{aligned}
J_{3} &\sim \int_{0}^{\infty}(\psi+1)^{\alpha}\chi'(\psi)d\Big(-(\phi^{2}+\epsilon^{2})^{\frac{1}{2}}\Big)\\
&= \int_{0}^{\infty}(\phi^{2}+\epsilon^{2})^{\frac{1}{2}}\Big[\alpha\psi^{\alpha-1}\chi'(\psi)+\psi^{\alpha}\chi''(\psi)\Big]d\psi\\
&\leq \int_{0}^{\infty}(\phi^{2}+\epsilon^{2})^{\frac{1}{2}}\psi^{\alpha}|\chi''(\psi)|d\psi,
\end{aligned}
\end{equation*}
where we have used the fact that $\chi$ is non-increasing. Since $|\chi''(\psi)|\lesssim N^{-2}$, we have,
\begin{equation*}
\begin{aligned}
J_{3} &\lesssim \int_{0}^{\infty}(|\phi|+\epsilon)\psi^{\alpha}|\chi''(\psi)|d\psi\\
&\lesssim \epsilon N^{\alpha-1}+\|\phi\|_{L^{2}_{\psi}}\|\chi''(\psi)\psi^{\alpha}\|_{L^{2}_{\psi}}\\
&\lesssim \epsilon N^{\alpha-1}+N^{\alpha-\frac{3}{2}}.
\end{aligned}
\end{equation*}
Combining the above estimates, we obtain that
\begin{equation}\label{j}
J \lesssim \epsilon^{\alpha}+N^{\alpha-\frac{3}{2}}.
\end{equation}

\textit{Estimate of K.}
Recall that
\begin{equation*}
\begin{aligned}
K=&-\bigg(\int_{0}^{\infty}A(X,\psi)\phi^{2}(\phi^{2}+\epsilon^{2})^{-\frac{1}{2}}\frac{1}{u}\psi^{\alpha}\chi(\psi)d\psi+\int_{0}^{\infty}(\phi^{2}+\epsilon^{2})^{\frac{1}{2}}\frac{u_{X}}{u^{2}}\psi^{\alpha}\chi(\psi)d\psi\bigg)\\
=&-\int_{0}^{\infty}\phi^{2}(\phi^{2}+\epsilon^{2})^{-\frac{1}{2}}\psi^{\alpha}\chi(\psi)\Big(\frac{A}{u}+\frac{u_{X}}{u^{2}}\Big)d\psi\\
&+\int_{0}^{\infty}\Big[\phi^{2}(\phi^{2}+\epsilon^{2})^{-\frac{1}{2}}-(\phi^{2}+\epsilon^{2})^{\frac{1}{2}}\Big]\frac{u_{X}}{u^{2}}\psi^{\alpha}\chi(\psi)d\psi\\
=&:K_{1}+K_{2}.
\end{aligned}
\end{equation*}
First, we estimate the term $K_{2}$. Since
\begin{equation}\label{phi}
u-\bar{u}=\rho(X,\psi),\quad u^{2}-\bar{u}^{2}=\phi(X,\psi),
\end{equation}
we have $\phi=\rho(u+\bar{u})$. Taking derivatives with respect to $ X $, we get
\begin{equation*}
\phi_{X}=\rho_{X}(u+\bar{u})+\rho(2\bar{u}_{X}+ \rho_{X}).
\end{equation*}
Therefore,
\begin{equation}\label{rho}
\rho_{X}=\frac{\phi_{X}-2\rho \bar{u}_{X}}{2u}.
\end{equation}
By applying Cauchy-Schwarz, and Hardy's inequalities, and the fact that $ |\bar{u}_{X}|\lesssim u $, we have the following series of inequalities:
\begin{equation}\label{K2}
\begin{aligned}
&K_{2}\leq \epsilon \left(\int_{0}^{2N} \frac{|\bar{u}_{X}|}{u^{2}}\psi^{\alpha}+\int_{0}^{2N} \frac{|\rho_{X}|}{u^{2}}\psi^{\alpha}\right)\\
&\lesssim \epsilon\int_{0}^{2N}\left[\frac{1}{u}+\frac{|\phi_{X}|}{u^{3}}+\frac{|\phi|}{u^{3}}\right]\psi^{\alpha}d\psi\\
&\leq \epsilon \left[\int_{0}^{2N}\frac{1}{u}d\psi +\left(\int_{0}^{2N}\frac{|\phi_{X}|^{2}+|\phi|^{2}}{u^{4}}d\psi\right)^{\frac{1}{2}}\left(\int_{0}^{2N}\frac{\psi^{2\alpha}}{u^{2}}d\psi\right)^{\frac{1}{2}}\right]\\
&\lesssim \epsilon N+\epsilon\left(N^{\frac{1}{2}+\alpha}+(X+1)^{\frac{\alpha}{2}+\frac{1}{4}}\right)\left[\|\phi\|_{L^{2}_{\psi}}+\|\phi_{X}\|_{L^{2}_{\psi}}+(X+1)^{\frac{1}{2}}(\|\phi_{\psi}\|_{L^{2}_{\psi}}+\|\phi_{X\psi}\|_{L^{2}_{\psi}})\right]\\
&\lesssim \epsilon\left(N^{\frac{1}{2}+\alpha}(X+1)^{\frac{1}{2}}+(X+1)^{\frac{3}{4}+\frac{\alpha}{2}}\right)\left(1+\|\phi_{X\psi}\|_{L^{2}_{\psi}}\right).
\end{aligned}
\end{equation}
Finally, we estimate the term $K_{1}$. Following a similar approach as in \cite{MR4097332}, we decompose the term as
\begin{equation*}
\begin{aligned}
\frac{A}{u}+\frac{u_{X}}{u^{2}} &= \frac{1}{u^{3}}(u^{2}A+uu_{X}) \\
&= \frac{1}{u^{3}}\bigg(-2\bar{u}^{2}\frac{\bar{u}_{yy}}{\bar{u}(2\bar{u}+\rho)}+\phi A+(\bar{u}+ \rho)(\bar{u}_{X}+ \rho_{X})\bigg) \\
&= \frac{1}{u^{3}}\big(-\bar{u}_{yy}+\bar{u}\bar{u}_{X}+\mathcal{R}\big).
\end{aligned}
\end{equation*}
Here, the remainder term $\mathcal{R}$ can be expressed as
\begin{equation}\label{R}
\mathcal{R}= A\phi+\frac{\rho \bar{u}_{yy}}{2\bar{u}+\rho}+(\rho \bar{u}_{X}+u\rho_{X})=:\mathcal{R}_{1}+\mathcal{R}_{2}+\mathcal{R}_{3}.
\end{equation}
It is important to note that in the von Mises coordinate system,
\begin{equation*}
\bar{u}_{X}(X,\psi)=\bar{u}_{x}(X,y(X,\psi;\bar{u}))+\frac{\bar{v}}{\bar{u}}\bar{u}_{y}(X,y(X,\psi;\bar{u})).
\end{equation*}
Thus, in the von Mises coordinates, $-\bar{u}_{yy}+\bar{u}\bar{u}_{X}=0$ as $\bar{u}$ satisfies the Prandtl equation \eqref{prandtl}. Consequently,
\begin{equation*}
K_{1}=-\sum_{i=1}^{3}\int_{0}^{\infty}\phi^{2}(\phi^{2}+\epsilon^{2})^{-\frac{1}{2}}\psi^{\alpha}\chi(\psi)\frac{\mathcal{R}_{i}}{u^{3}}d\psi=K_{11}+K_{12}+K_{13}.
\end{equation*}
First, we estimate $K_{11}$. Indeed,
\begin{equation*}
\begin{aligned}
K_{11} &\lesssim (X+1)^{-1}\int_{0}^{\infty}\frac{1}{u^{3}}\phi^{2}\psi^{\alpha}d\psi \\
&\lesssim (X+1)^{-1}\bigg(\int_{0}^{\sqrt{X+1}}\frac{(X+1)^{\frac{3}{4}}}{\psi^{\frac{3}{2}}}\phi^{2}\psi^{\alpha}d\psi + \int_{\sqrt{X+1}}^{\infty}\phi^{2}\psi^{\alpha}\chi(\psi)d\psi\bigg) \\
&=:K_{11}^{(1)} + K_{11}^{(2)}.
\end{aligned}
\end{equation*}
For the term $K_{11}^{(1)}$, note that $\phi(x,0)=0$, thus
\begin{equation*}
\frac{\phi^{2}}{\psi^{\frac{3}{2}}} = \frac{\phi}{\psi}\frac{\phi}{\psi^{\frac{1}{2}}} \leq \|\phi_{\psi}\|_{L^{\infty}_{\psi}}\|\phi\|_{\dot{C}^{\frac{1}{2}}_{\psi}}.
\end{equation*}
By applying Sobolev embedding $\dot{H}^{1}\rightarrow \dot{C}^{1/2}$ in $\mathbb{R}^{1}$ and Theorem \ref{iyer}, we obtain
\begin{equation}\label{K111}
\begin{aligned}
K_{11}^{(1)} &\lesssim (X+1)^{\frac{1}{4}+\frac{\alpha}{2}}\|\phi_{\psi}\|_{L^{\infty}_{\psi}}\|\phi_{\psi}\|_{L^{2}_{\psi}} \\
&\lesssim (X+1)^{-(\frac{3}{2}-\frac{\alpha}{2}-2\delta)}.
\end{aligned}
\end{equation}
For $K_{11}^{(2)}$, utilizing Theorem \ref{iyer} once more, we get that
\begin{equation}\label{K112}
K_{11}^{(2)}\lesssim (X+1)^{-1}\|\phi\|_{L^{\infty}_{\psi}}I(X)\lesssim (X+1)^{-(\frac{3}{2}-\delta)}I(X).
\end{equation}
The term $K_{12}$ can be handled similarly to $K_{11}$ since
\begin{equation*}
|\mathcal{R}_{2}| \lesssim \frac{|\rho \bar{u}_{yy}|}{\bar{u}} \lesssim |A\phi| = |\mathcal{R}_{1}|.
\end{equation*}
Therefore,
\begin{equation}\label{K12}
K_{12} \lesssim (X+1)^{-(\frac{3}{2}-\frac{\alpha}{2}-2\delta)}+(X+1)^{-(\frac{3}{2}-\delta)}I(X).
\end{equation}

For the final term $K_{13}$, recall that
\begin{equation*}
\begin{aligned}
K_{13} &\leq \int_{0}^{\infty}|\phi|\psi^{\alpha}\frac{|\rho \bar{u}_{X}+u\rho_{X}|}{u^{3}}d\psi \\
&= \int_{0}^{\sqrt{X+1}}|\phi|\psi^{\alpha}\frac{|\rho \bar{u}_{X}+ u\rho_{X}|}{u^{3}}d\psi \\
&\quad + \int_{\sqrt{X+1}}^{\infty}|\phi|\psi^{\alpha}\frac{|\rho \bar{u}_{X}+ u\rho_{X}|}{u^{3}}d\psi \\
&= K_{13}^{(1)} + K_{13}^{(2)}.
\end{aligned}
\end{equation*}
Applying \eqref{phi} and \eqref{rho}, we have
\begin{equation}\label{K131}
\begin{aligned}
K_{13}^{(1)} &\lesssim \int_{0}^{\sqrt{X+1}}\left(\frac{|\phi\phi_{X}|\psi^{\alpha}}{u^{3}} + \frac{|\bar{u}_{X}|\phi^{2}\psi^{\alpha}}{u^{4}}\right)d\psi.
\end{aligned}
\end{equation}
To estimate the first term of \eqref{K131}, we apply H\"older's inequality, Hardy's inequality, and Theorem \ref{iyer}:
\begin{equation*}
\begin{aligned}
\int_{0}^{\sqrt{X+1}}\frac{|\phi\phi_{X}|\psi^{\alpha}}{u^{3}}d\psi &\lesssim (X+1)^{\frac{3}{4}}\int_{0}^{\sqrt{X+1}}\frac{|\phi\phi_{X}|\psi^{\alpha}}{\psi^{\frac{3}{2}}} \\
&\lesssim (X+1)^{\frac{3}{4}+\frac{\alpha}{2}}\left(\int_{0}^{\sqrt{X+1}}\frac{|\phi_{X}|^{2}}{\psi^{2}}\right)^{\frac{1}{2}}\left(\int_{0}^{\sqrt{X+1}}\phi^{2}\right)^{\frac{1}{4}}\left(\int_{0}^{\sqrt{X+1}}\frac{|\phi|^{2}}{\psi^{2}}\right)^{\frac{1}{4}} \\
&\lesssim (X+1)^{\frac{3}{4}+\frac{\alpha}{2}}\|\phi_{X\psi}\|_{L^{2}_{\psi}}\|\phi\|_{L^{2}_{\psi}}^{\frac{1}{2}}\|\phi_{\psi}\|_{L^{2}_{\psi}}^{\frac{1}{2}} \\
&\lesssim (X+1)^{(\frac{1}{4}+\frac{\alpha}{2}+\delta)}\|\phi_{X\psi}\|_{L^{2}_{\psi}}.
\end{aligned}
\end{equation*}
For the second term of \eqref{K131}, by applying Hardy's inequality, Theorem \ref{iyer}, and Lemma \ref{Blasius}, we get:
\begin{equation*}
\begin{aligned}
\int_{0}^{\sqrt{X+1}}\frac{|\bar{u}_{X}|\phi^{2} \psi^{\alpha}}{u^{4}}d\psi &\lesssim (X+1)^{\frac{\alpha}{2}-\frac{1}{4}}\int_{0}^{\sqrt{X+1}}\frac{\phi^{2}}{\psi^{\frac{3}{2}}}d\psi \\
&\lesssim (X+1)^{\frac{\alpha}{2}}\int_{0}^{\sqrt{X+1}}\frac{\phi^{2}}{\psi^{2}}d\psi \\
&\lesssim (X+1)^{\frac{\alpha}{2}}\|\phi_{\psi}\|_{L^{2}_{\psi}}^{2} \\
&\lesssim (X+1)^{-(\frac{3}{2}-\frac{\alpha}{2}-2\delta)}.
\end{aligned}
\end{equation*}
Therefore, 
\begin{equation}\label{k131}
 K_{13}^{(1)}\lesssim (X+1)^{(\frac{1}{4}+\frac{\alpha}{2}+\delta)}\|\phi_{X\psi}\|_{L^{2}_{\psi}}+(X+1)^{-(\frac{3}{2}-\frac{\alpha}{2}-2\delta)}.
\end{equation}
Now we estimate the far-field term $K_{13}^{(2)}$ by utilizing H\"older's inequality and Theorem \ref{iyer}:
\begin{equation}\label{K132}
\begin{aligned}
K_{13}^{(2)} &\lesssim \int_{\sqrt{X+1}}^{\infty}\left(|\phi\phi_{X}| \psi^{\alpha} + |\bar{u}_{X}|\phi^{2}\psi^{\alpha}\right)d\psi \\
&\lesssim  \|\phi_{X}\|_{L^{\infty}_{\psi}}I(X)+ \|\bar{u}_{X}\|_{L^{\infty}_{\psi}}\|\phi\|_{L^{\infty}_{\psi}}I(X)\\
&\lesssim (X+1)^{-(\frac{5}{8}-\frac{\delta}{2})}\|\phi_{X\psi}\|_{L^{2}_{\psi}}^{\frac{1}{2}}I(X)+(X+1)^{-(\frac{3}{2}-\delta)}I(X).
\end{aligned}
\end{equation}

Combining \eqref{K2}, \eqref{K111}, \eqref{K112}, \eqref{K12}, \eqref{k131}, and \eqref{K132}, we get:
\begin{equation*}
\begin{aligned}
K \lesssim &\epsilon\left(N^{\frac{1}{2}+\alpha}(X+1)^{\frac{1}{2}}+(X+1)^{\frac{3}{4}+\frac{\alpha}{2}}\right)\left(1+\|\phi_{X\psi}\|_{L^{2}_{\psi}}\right)+(X+1)^{-(\frac{3}{2}-\frac{\alpha}{2}-2\delta)}\\
&+(X+1)^{(\frac{1}{4}+\frac{\alpha}{2}+\delta)}\|\phi_{X\psi}\|_{L^{2}_{\psi}}+\left[(X+1)^{-(\frac{3}{2}-\delta)}+(X+1)^{-(\frac{5}{8}-\frac{\alpha}{2})}\|\phi_{X\psi}\|_{L^{2}_{\psi}}^{\frac{1}{2}}\right]I(X).
\end{aligned}
\end{equation*}

Referring back to \eqref{j}, we finally derive:
\begin{equation*}
\begin{aligned}
I'(X) \lesssim& \epsilon^{\alpha} + N^{\alpha-3/2} +\epsilon\left(N^{\frac{1}{2}+\alpha}(X+1)^{\frac{1}{2}}+(X+1)^{\frac{3}{4}+\frac{\alpha}{2}}\right)\left(1+\|\phi_{X\psi}\|_{L^{2}_{\psi}}\right)\\
&+(X+1)^{-(\frac{3}{2}-\frac{\alpha}{2}-2\delta)}+(X+1)^{(\frac{1}{4}+\frac{\alpha}{2}+\delta)}\|\phi_{X\psi}\|_{L^{2}_{\psi}}\\
&+\left[(X+1)^{-(\frac{3}{2}-\delta)}+(X+1)^{-(\frac{5}{8}-\frac{\alpha}{2})}\|\phi_{X\psi}\|_{L^{2}_{\psi}}^{\frac{1}{2}}\right]I(X).
\end{aligned}
\end{equation*}
which is the desired differential inequality \eqref{diff}. As explained at the beginning of the proof presentation, this concludes the proof of Lemma \ref{weight}.
\end{proof}

\section{Enhanced $ H^{1} $ Decay}\label{enhanced}
In this section, we investigate the decay rates of $\| \phi \|_{L^{2}_{\psi}}$ and $\| \phi_{\psi} \|_{L^{2}_{\psi}}$, which are demonstrated to be faster than those established in \cite{MR4097332}. The enhanced decay rates are derived through weighted energy estimates and iteration techniques.

We first establish the $L^{2}$ decay rate of $\phi$.
\begin{lemma}\label{L2decay} For $ \beta=\alpha+\frac{1}{2}=\frac{29}{20} $ and $ X\geq 0 $, we have
\begin{equation}\label{L2}
\int_{0}^{\infty}\frac{\phi^{2}}{u}d\psi \lesssim (X+1)^{-\beta}.
\end{equation}
\end{lemma}
\begin{proof}
By testing Equation \eqref{eqphi} with $\frac{\phi}{u}$, we obtain
\begin{equation*}
\int_{0}^{\infty}\phi_{X}\phi\frac{1}{u}d\psi - \int_{0}^{\infty}\phi_{\psi\psi}\phi d\psi + \int_{0}^{\infty} A\phi^{2}\frac{1}{u}d\psi = 0.
\end{equation*}
Integrating by parts, we have
\begin{equation}\label{test}
\frac{\partial_{X}}{2}\int_{0}^{\infty}\phi^{2}\frac{1}{u} + \int_{0}^{\infty}|\phi_{\psi}|^{2} + \int_{0}^{\infty}\phi^{2}\frac{2u^{2}A+uu_{X}}{2u^{3}} = 0.
\end{equation}
Since $ A\geq 0 $, the third term can be handled similarly to the proof of Lemma \ref{weight}:
\begin{equation}\label{remainder}
\int_{0}^{\infty}\phi^{2}\frac{2u^{2}A+uu_{X}}{2u^{3}} \geq \int_{0}^{\infty}\phi^{2}\frac{\mathcal{R}_{1}+\mathcal{R}_{2}+\mathcal{R}_{3}}{u^{3}}d\psi =: \mathcal{N}(\phi),
\end{equation}
where $\mathcal{R}_{i}$ is defined in \eqref{R}. For the first term of \eqref{remainder}, using Theorem \ref{iyer} and Lemma \ref{sobolev}, we obtain
\begin{equation*}
\begin{aligned}
\int_{0}^{\infty}\phi^{2}\frac{|\mathcal{R}_{1}|}{u^{3}}d\psi &\leq \int_{0}^{\infty}|A \phi^{3}|\frac{1}{u^{3}}d\psi \\
&\lesssim (X+1)^{-1}\left((X+1)^{\frac{3}{4}}\int_{0}^{\sqrt{X+1}}\frac{|\phi|^{3}}{\psi^{\frac{3}{2}}}d\psi + \int_{\sqrt{X+1}}^{\infty}|\phi|^{3}d\psi\right) \\
&\lesssim (X+1)^{\frac{1}{4}}\|\phi\|_{\dot{C}^{1/2}_{\psi}}^{3} + (X+1)^{-1}\|\phi\|_{L^{\infty}_{\psi}}\|\phi\|_{L^{2}_{\psi}}^{2} \\
&\lesssim (X+1)^{\frac{1}{4}}\|\phi_{\psi}\|_{L^{2}_{\psi}}^{3} + (X+1)^{-1}\|\phi\|_{L^{\infty}_{\psi}}\|\phi\|_{L^{2}_{\psi}}^{2} \\
&\lesssim (X+1)^{-(2-3\delta)}.
\end{aligned}
\end{equation*}
The remaining terms of \eqref{remainder} can be handled similarly and we get $|\mathcal{N}(\phi)| \lesssim (X+1)^{-(2-3\delta)} $ Thus we obtain the following energy inequality:
\begin{equation}\label{energy}
\frac{\partial_{X}}{2}\left(\int_{0}^{\infty}\phi^{2}\frac{1}{u}\right) + \int_{0}^{\infty}|\phi_{\psi}|^{2} \lesssim (X+1)^{-(2-3\delta)}.
\end{equation}
Next, we need to control the (weighted) $L^{2}$ norm by interpolating between the norm $\|\phi_{\psi}\|_{L^{2}_{\psi}}$ (high-frequency) and the weighted $L^{1}$ norm $\|\psi^{\alpha}\phi\|_{L^{1}_{\psi}}$ (low-frequency) for $\alpha=\frac{19}{20}$. We have
\begin{equation}\label{zero}
\begin{aligned}
\int \phi^{2}\frac{1}{u} &\lesssim \int_{0}^{\sqrt{X+1}}\phi^{2}\frac{(X+1)^{\frac{1}{4}}}{\psi^{\frac{1}{2}}}d\psi + \int_{\sqrt{X+1}}^{\infty}\phi^{2}d\psi \\
&\lesssim \left\{ \int_{0}^{\infty}\phi^{2}\frac{(X+1)^{\frac{1}{4}}}{\psi^{\frac{1}{2}}}d\psi, \int_{0}^{\infty}\phi^{2}d\psi \right\}.
\end{aligned}
\end{equation}
By Hardy's inequality and Lemma \ref{weight}, we have
\begin{equation}\label{onee}
\begin{aligned}
\int_{0}^{\infty}\frac{\phi^{2}}{\psi^{\frac{1}{2}}}d\psi&=\int_{0}^{r}\frac{\phi^{2}}{\psi^{\frac{1}{2}}}d\psi+\int_{r}^{\infty}\frac{\phi^{2}}{\psi^{\frac{1}{2}}}d\psi\\
&\leq r^{\frac{3}{2}}\int_{0}^{r}\frac{\phi^{2}}{\psi^{2}}d\psi+r^{-(\alpha+\frac{1}{2})}\|\phi\|_{L^{\infty}_{\psi}}\|\psi^{\alpha} \phi\|_{L^{1}_{\psi}}\\
&\lesssim r^{\frac{3}{2}}\|\phi_{\psi}\|_{L^{2}_{\psi}}^{2}+r^{-(\alpha+\frac{1}{2})}\|\phi\|_{L^{2}_{\psi}}^{\frac{1}{2}}\|\phi_{\psi}\|_{L^{2}_{\psi}}^{\frac{1}{2}}\|\psi^{\alpha} \phi\|_{L^{1}_{\psi}}\\
&\lesssim r^{\frac{3}{2}}\|\phi_{\psi}\|_{L^{2}_{\psi}}^{2}+r^{-(\alpha+\frac{1}{2})}\|\phi\|_{L^{2}_{\psi}}^{\frac{1}{2}}\|\phi_{\psi}\|_{L^{2}_{\psi}}^{\frac{1}{2}}\\
&\lesssim \|\phi\|_{L^{2}_{\psi}}^{\frac{3}{8+4\alpha}}\|\phi_{\psi}\|_{L^{2}_{\psi}}^{\frac{7+8\alpha}{8+4\alpha}},
\end{aligned}
\end{equation}
where we choose $ r $ such that $ r^{\frac{3}{2}}\|\phi_{\psi}\|_{L^{2}_{\psi}}^{2}=r^{-(\alpha+\frac{1}{2})}\|\phi\|_{L^{2}_{\psi}}^{\frac{1}{2}}\|\phi_{\psi}\|_{L^{2}_{\psi}}^{\frac{1}{2}} $. Furthermore, by applying Hardy's inequality and Lemma \ref{weight} once more, we obtain that 
\begin{equation}\label{one}
\begin{aligned}
\|\phi\|_{L^{2}_{\psi}}^{2}&=\int_{0}^{r}\phi^{2}d\psi+\int_{r}^{\infty}\phi^{2}d\psi\\
&\leq r^{2}\int_{0}^{r}\frac{\phi^{2}}{\psi^{2}}d\psi+r^{-\alpha}\|\phi\|_{L^{\infty}_{\psi}}\|\psi^{\alpha} \phi\|_{L^{1}_{\psi}}\\
&\lesssim r^{2}\|\phi_{\psi}\|_{L^{2}_{\psi}}^{2}+r^{-\alpha}\|\phi\|_{L^{2}_{\psi}}^{\frac{1}{2}}\|\phi_{\psi}\|_{L^{2}_{\psi}}^{\frac{1}{2}}\|\psi^{\alpha} \phi\|_{L^{1}_{\psi}}\\
&\lesssim r^{2}\|\phi_{\psi}\|_{L^{2}_{\psi}}^{2}+r^{-\alpha}\|\phi\|_{L^{2}_{\psi}}^{\frac{1}{2}}\|\phi_{\psi}\|_{L^{2}_{\psi}}^{\frac{1}{2}}\\
&\lesssim \|\phi\|_{L^{2}_{\psi}}^{\frac{1}{2+\alpha}}\|\phi_{\psi}\|_{L^{2}_{\psi}}^{\frac{1+2\alpha}{2+\alpha}},
\end{aligned}
\end{equation}
where we choose $ r $ such that $ r^{2}\|\phi_{\psi}\|_{L^{2}_{\psi}}^{2} = r^{-\alpha}\|\phi\|_{L^{2}_{\psi}}^{\frac{1}{2}}\|\phi_{\psi}\|_{L^{2}_{\psi}}^{\frac{1}{2}}$. Hence,
\begin{equation}\label{two}
\|\phi\|_{L^{2}_{\psi}}\lesssim \|\phi_{\psi}\|_{L^{2}_{\psi}}^{\frac{1+2\alpha}{3+2\alpha}}.
\end{equation}
Substitute \eqref{two} into \eqref{onee}, we immediately derive
\begin{equation}\label{three}
\int_{0}^{\infty}\frac{\phi^{2}}{\psi^{\frac{1}{2}}}d\psi\lesssim \|\phi_{\psi}\|_{L^{2}_{\psi}}^{\frac{3+4\alpha}{3+2\alpha}}.
\end{equation}
Combining \eqref{zero},\eqref{two},\eqref{three}, we obtain
\begin{equation}\label{four}
\int_{0}^{\infty}\phi^{2}\frac{1}{u}d\psi\lesssim \max\left\{(X+1)^{\frac{1}{4}}\|\phi_{\psi}\|_{L^{2}_{\psi}}^{\frac{3+4\alpha}{3+2\alpha}}, \|\phi_{\psi}\|_{L^{2}_{\psi}}^{\frac{2+4\alpha}{3+2\alpha}}\right\}.
\end{equation}
Define
\begin{equation*}
\mathcal{A}(X):=\int_{0}^{\infty}|\phi(X,\psi)|^{2}\frac{1}{u}d\psi.
\end{equation*}
To achieve the decay rate \eqref{L2} for $\mathcal{A}(X)$, we require two iterations. We already know that $\mathcal{A}(X)\lesssim (X+1)^{-(\frac{1}{2}-2\delta)}$ from Theorem \ref{iyer}. To illustrate our iteration approach clearly, recalling \eqref{test}, \eqref{remainder}, we present the following chart:
{\small \begin{equation*}
\begin{aligned}
&{\boxed { \mathcal{A}(X) \lesssim (X+1)^{-(\frac{1}{2}-2\delta)}} }\longrightarrow {\boxed{ |\mathcal{N}(\phi)|\lesssim (X+1)^{-(2-3\delta)}}}\xrightarrow{\text{First Iteration}}{\boxed {\mathcal{A}(X) \lesssim (X+1)^{-\frac{11}{10}}}}\rightarrow\cdots\\
&{\boxed {\mathcal{A}(X) \lesssim (X+1)^{-\frac{11}{10}}} }\longrightarrow {\boxed{ |\mathcal{N}(\phi)|\lesssim (X+1)^{-\frac{51}{20}}}}\xrightarrow{\text{Second Iteration}}{\boxed {\mathcal{A}(X) \lesssim (X+1)^{-\beta}} }
\end{aligned}
\end{equation*}}

\textit{First Iteration.} First and foremost, we emphasize that in this step of iteration, no need for $ \alpha $ to be extremely close to one. Combining \eqref{energy} and \eqref{four}, there exist positive numbers $c_{0}$ and $c_{1}$ such that
\begin{equation}\label{differential}
\mathcal{A}'(X)+c_{0}\min \left\{(X+1)^{-\frac{3+2\alpha}{6+8\alpha}}\mathcal{A}^{\frac{6+4\alpha}{3+4\alpha}}(X),\mathcal{A}^{\frac{3+2\alpha}{1+2\alpha}}(X)\right\}\leq c_{1} (X+1)^{-(2-3\delta)}.
\end{equation}
We first demonstrate that the above differential inequality implies
\begin{equation}\label{first}
\mathcal{A}(X)\lesssim (X+1)^{-\frac{11}{10}}.
\end{equation}
Let
\begin{equation*}
\mathcal{J}=\left\{X>0: \mathcal{A}(X)>C(X+1)^{-\frac{11}{10}}\right\}.
\end{equation*}
If the open set $\mathcal{J}$ is non-empty, we can express $\mathcal{J}$ as
\begin{equation*}
\mathcal{J}=\bigcup_{k=1}^{\infty}(a_{k},b_{k}),\quad (a_{k},b_{k})\cap (a_{j},b_{j})=\emptyset \quad \text{if}\quad k\neq j.
\end{equation*}
Choose a sub-interval of  $\mathcal{J} $ denoted by $(a,b)$. It is evident that $\mathcal{A}(a)=C(a+1)^{-\frac{11}{10}}$, if $ C $ is chosen large. Define
\begin{equation*}
\mathcal{B}(X)=\mathcal{A}(X)-C(X+1)^{-\frac{11}{10}},
\end{equation*}
then $ B(a) = 0 $ and $B(X)>0$ for $ a<X<b $. Using the differential inequality \eqref{differential}, the following inequality holds for $a<X<b$:
\begin{equation*}
\begin{aligned}
\mathcal{B}'(X)&=\mathcal{A}'(X)+\frac{11}{10}C(X+1)^{-\frac{21}{10}}\\
&\leq c_{1}(X+1)^{-(2-3\delta)}+\frac{11}{10}C(X+1)^{-\frac{21}{10}}-c_{0}\min \left\{(X+1)^{-\frac{3+2\alpha}{6+8\alpha}}\mathcal{A}^{\frac{6+4\alpha}{3+4\alpha}}(X),\mathcal{A}^{\frac{3+2\alpha}{1+2\alpha}}(X)\right\}\\
&=c_{1}(X+1)^{-(2-3\delta)}+\frac{11}{10}C(X+1)^{-\frac{21}{10}}-c_{0}(X+1)^{-\frac{3+2\alpha}{6+8\alpha}}\mathcal{A}^{\frac{6+4\alpha}{3+4\alpha}}(X)\\
&\leq c_{1}(X+1)^{-(2-3\delta)}+\frac{11}{10}C(X+1)^{-\frac{21}{10}}-c_{0}C^{\frac{6+4\alpha}{3+4\alpha}}(X+1)^{-\left[\frac{3+2\alpha}{6+8\alpha}+\frac{11}{10}\frac{6+4\alpha}{3+4\alpha}\right]}.
\end{aligned}
\end{equation*}
Therefore, $\mathcal{B}'(a)<0$ whenever $C$ is large. Since $\mathcal{B}(a)=0$, it follows that $\mathcal{B}(a)<0$ in a neighborhood of $X=a$, contradicting $\mathcal{B}(X)>0$ for $a<X<b$. Thus, there exists $C>0$ such that $\mathcal{J}=\emptyset$, i.e., \eqref{first} holds.

\textit{Second Iteration.} To obtain a faster decay rate, we revisit the estimation of the third term in \eqref{test}. Let us recall
\begin{equation*}
\int_{0}^{\infty}\phi^{2}\frac{2u^{2}A+uu_{X}}{2u^{3}}\geq \int_{0}^{\infty}\phi^{2}\frac{\mathcal{R}_{1}+\mathcal{R}_{2}+\mathcal{R}_{3}}{u^{3}}d\psi=:\mathcal{N}(\phi).
\end{equation*}
For the first term, by applying Theorem \ref{iyer} and the improved $L^{2}$ decay rate \eqref{first}, we have
\begin{equation*}
\begin{aligned}
\int_{0}^{\infty}\phi^{2}\frac{|\mathcal{R}_{1}|}{u^{3}}d\psi&\leq \int_{0}^{\infty} |A\phi^{3}|\frac{1}{u^{3}}d\psi\\
&\lesssim (X+1)^{-1}\left\|\frac{\phi}{u^{2}}\right\|_{L^{\infty}_{\psi}}\int_{0}^{\infty}\phi^{2}\frac{1}{u}d\psi\\
&\lesssim (X+1)^{-1}(X+1)^{-(\frac{1}{2}-\delta)}(X+1)^{-\frac{11}{10}}\\
&\lesssim (X+1)^{-\frac{51}{20}}.
\end{aligned}
\end{equation*}
The terms involving $\mathcal{R}_{2}$ and $\mathcal{R}_{3}$ can be estimated similarly. Thus, we conclude that
\begin{equation*}
|\mathcal{N}(\phi)|\lesssim (X+1)^{-\frac{51}{20}}.
\end{equation*}
Therefore, similar to \eqref{differential}, there exist two positive constants $d_{0}$ and $d_{1}$ such that
\begin{equation}\label{differentiall}
\mathcal{A}'(X)+d_{0}\min \left\{(X+1)^{-\frac{3+2\alpha}{6+8\alpha}}\mathcal{A}^{\frac{6+4\alpha}{3+4\alpha}}(X),\mathcal{A}^{\frac{3+2\alpha}{1+2\alpha}}(X)\right\}\leq d_{1} (X+1)^{-\frac{51}{20}}.
\end{equation}
We assert that the above differential inequality \eqref{differentiall} implies that,
\begin{equation}\label{decayy}
\mathcal{A}(X)\lesssim_{\beta} (X+1)^{-\beta},\quad \beta=\alpha+\frac{1}{2}=\frac{29}{20}.
\end{equation}
We only need to repeat the procedure in the first iteration. Initially, fix $\beta>0$ and define
\begin{equation*}
\mathcal{I} = \left\{X>0 : \mathcal{A}(X)>D(X+1)^{-\beta}\right\},
\end{equation*}
for some $ D>1 $ to be chosen large.

By continuity, we can express
\begin{equation*}
\mathcal{I}=\bigcup_{k}(c_{k},d_{k}),\quad (c_{k},d_{k}) \cap (c_{j},d_{j}) = \emptyset\ \text{if}\ k \neq j.
\end{equation*}
Select one of the sub-intervals of $\mathcal{I}$, denoted by $(c,d)$, hence  $\mathcal{A}(c)=D(c+1)^{-\beta}$. Define
\begin{equation*}
\mathcal{F}(X)=\mathcal{A}(X)-D(X+1)^{-\beta}.
\end{equation*}
It is noted that for any $ X\in \mathcal{I} $:
\begin{equation*}
(X+1)^{-\frac{3+2\alpha}{6+8\alpha}}\mathcal{A}^{\frac{6+4\alpha}{3+4\alpha}}(X)\leq \mathcal{A}^{\frac{3+2\alpha}{1+2\alpha}}(X).
\end{equation*}
Hence, utilizing \eqref{differentiall}, when $ c<X<d $:
\begin{equation*}
\begin{aligned}
\mathcal{F}'(X)&=\mathcal{A}'(X)+\beta D(X+1)^{-(\beta+1)}\\
&\leq d_{1}(X+1)^{-\frac{51}{20}}+\beta D(X+1)^{-(\beta+1)}-d_{0}\min \left\{(X+1)^{-\frac{3+2\alpha}{6+8\alpha}}\mathcal{A}^{\frac{6+4\alpha}{3+4\alpha}}(X),\mathcal{A}^{\frac{3+2\alpha}{1+2\alpha}}(X)\right\}\\
&=d_{1}(X+1)^{-\frac{51}{20}}+\beta D(X+1)^{-(\beta+1)}-d_{0}(X+1)^{-\frac{3+2\alpha}{6+8\alpha}}\mathcal{A}^{\frac{6+4\alpha}{3+4\alpha}}(X)\\
&\leq d_{1}(X+1)^{-\frac{51}{20}}+\beta D(X+1)^{-(\beta+1)}-d_{0}D^{\frac{6+4\alpha}{3+4\alpha}}(X+1)^{-\left[\frac{3+2\alpha}{6+8\alpha}+\beta\frac{6+4\alpha}{3+4\alpha}\right]}.
\end{aligned}
\end{equation*}
Hence, when $ D>1 $ is large enough, we have $ \mathcal{F}'(c)<0 $, contradicting  $ \mathcal{F}(X)>0 $ when $ c<X<d $. Therefore, there exists $ D>0 $ such that $ \mathcal{I}=\emptyset $, i.e. \eqref{decayy} holds. Hence we complete the proof of Lemma \ref{L2decay}.
\end{proof}
Next we establish the $ L^{2} $ decay rate of $ \phi_{\psi} $.
\begin{lemma}\label{Middle}
For any $ X_{0}>0 $, we have
\begin{equation}\label{middle}
\int_{0}^{X_{0}}\int_{0}^{\infty}(X+1)^{\beta+\frac{1}{2}}|\phi_{\psi}|^{2}d\psi dX\lesssim (X_{0}+1)^{\frac{1}{2}}.
\end{equation}
\end{lemma}
\begin{proof}Testing \eqref{eqphi} by $ \displaystyle \frac{\phi}{u}(X+1)^{\beta+\frac{1}{2}} $ and integrating by parts, we obtain:
\begin{equation}\label{middleenergy}
\begin{aligned}
&\frac{1}{2}\p_{X}\left[(X+1)^{\beta+\frac{1}{2}}\int_{0}^{\infty}|\phi|^{2}\frac{1}{u}d\psi\right]-\frac{1}{2}\left(\beta+\frac{1}{2}\right)(X+1)^{\beta-\frac{1}{2}}\int_{0}^{\infty}|\phi|^{2}\frac{1}{u}d\psi\\
&\quad\quad+\frac{1}{2}(X+1)^{\beta+\frac{1}{2}}\int_{0}^{\infty}\phi^{2}\frac{u_{X}}{u^{2}}d\psi+(X+1)^{\beta+\frac{1}{2}}\int_{0}^{\infty}|\phi_{\psi}|^{2}d\psi\\
&\quad\quad+(X+1)^{\beta+\frac{1}{2}}\int_{0}^{\infty}A\phi^{2}\frac{1}{u}d\psi=(X+1)^{\beta+\frac{1}{2}}\int_{0}^{\infty}\left(\phi_{X}-u\phi_{\psi\psi}+A\phi\right)\frac{\phi}{u}=0.
\end{aligned}
\end{equation}
Applying Lemma \ref{L2decay}, we have:
\begin{equation*}
(X+1)^{\beta+\frac{1}{2}}\int_{0}^{\infty}|\phi|^{2}\frac{1}{u}d\psi\lesssim \sqrt{X+1}.
\end{equation*}
Additionally, recalling that $ u=\bar{u}+\rho $, we have:
\begin{equation}\label{hhh}
\int_{0}^{\infty}\phi^{2}\frac{u_{X}}{u^{2}}d\psi \leq \int_{0}^{\infty}\phi^{2}\frac{|\bar{u}_{X}|}{u^{2}}d\psi +\int_{0}^{\infty}\phi^{2}\frac{|\rho_{X}|}{u^{2}}d\psi=\eqref{hhh}_{1}+\eqref{hhh}_{2}.
\end{equation}
Applying Lemma \ref{Blasius} and Lemma \ref{L2decay}, we have:
\begin{equation*}
\eqref{hhh}_{1}\leq \left\|\frac{\bar{u}_{X}}{u}\right\|_{L^{\infty}_{\psi}}\int_{0}^{\infty}\phi^{2}\frac{1}{u}d\psi \lesssim (X+1)^{-(\beta+1)}.
\end{equation*}

For $\eqref{hhh}_{2}$, by utilizing \eqref{rho}, Theorem \ref{iyer}, Lemma \ref{L2decay}, and Hardy's inequality, we get:
\begin{equation*}
\begin{aligned}
\eqref{hhh}_{2}&\leq \left(\int_{0}^{\infty}\frac{\phi^{4}}{u^{2}}d\psi\right)^{\frac{1}{2}}\left(\int_{0}^{\infty}\frac{|\rho_{X}|^{2}}{u^{2}}d\psi\right)^{\frac{1}{2}}\\
&\lesssim  \left\|\frac{\phi}{u^{\frac{1}{2}}}\right\|_{L^{\infty}_{\psi}}\left(\int_{0}^{\infty}\frac{\phi^{2}}{u}d\psi\right)^{\frac{1}{2}}\left(\int_{0}^{\infty}\left[\frac{|\phi_{X}|^{2}}{u^{4}}+\frac{|\phi|^{2}|\bar{u}_{X}|^{2}}{u^{6}}\right]d\psi\right)^{\frac{1}{2}}\\
&\lesssim (X+1)^{-(\frac{7}{4}+\frac{\beta}{2}-2\delta)}+(X+1)^{-(\frac{\beta}{2}-\delta)}\|\phi_{X\psi}\|_{L^{2}_{\psi}}.
\end{aligned}
\end{equation*}

Therefore, we conclude that:
\begin{equation*}
(X+1)^{\beta+\frac{1}{2}}\int_{0}^{\infty}\phi^{2}\frac{|u_{X}|}{u^{2}}d\psi\lesssim (X+1)^{-\frac{1}{2}}+(X+1)^{\frac{1}{2}+\frac{\beta}{2}+\delta}\|\phi_{X\psi}\|_{L^{2}_{\psi}}.
\end{equation*}

Hence, integrating \eqref{middleenergy} with respect to $ X $, and applying Cauchy-Schwarz inequality, we arrive at \eqref{middle}.
\end{proof}
\begin{lemma}\label{H1decay}For any $ X_{0}>0 $ :
\begin{equation}\label{H1}
(X_{0}+1)^{\beta+\frac{3}{2}}\int_{0}^{\infty}|\phi_{\psi}(X_{0},\psi)|^{2}d\psi+\int_{0}^{X_{0}}\int_{0}^{\infty}|\phi_{X}|^{2}\frac{1}{u}d\psi (X+1)^{\beta+\frac{3}{2}}dX\lesssim (X_{0}+1)^{\frac{1}{2}}.
\end{equation}
\end{lemma}
\begin{proof}
Testing \eqref{eqphi} by $ \displaystyle \frac{\phi_{X}}{u}(X+1)^{\beta+\frac{3}{2}} $, and integrating with respect to $\psi$, we obtain
\begin{equation}\label{H1energy}
(X+1)^{\beta+\frac{3}{2}}\int_{0}^{\infty}|\phi_{X}|^{2}\frac{1}{u}-(X+1)^{\beta+\frac{3}{2}}\int_{0}^{\infty}\phi_{\psi\psi}\phi_{X}+(X+1)^{\beta+\frac{3}{2}}\int_{0}^{\infty}A\phi\phi_{X}\frac{1}{u}=0.
\end{equation}
The second term can be computed as follows:
\begin{equation}\label{H1Energyy}
\begin{aligned}
&-(X+1)^{\beta+\frac{3}{2}}\int_{0}^{\infty}\phi_{\psi\psi}\phi_{X}d\psi =(X+1)^{\beta+\frac{3}{2}}\int_{0}^{\infty}\phi_{\psi}\phi_{X\psi}d\psi\\
&\quad\quad =\frac{1}{2}\p_{X}\left[(X+1)^{\beta+\frac{3}{2}}\int_{0}^{\infty}|\phi_{\psi}|^{2}d\psi\right]-\frac{1}{2}\left(\beta+\frac{3}{2}\right)(X+1)^{\beta+\frac{1}{2}}\int_{0}^{\infty}|\phi_{\psi}|^{2}d\psi.
\end{aligned}
\end{equation}
Substitute \eqref{H1Energyy} into \eqref{H1energy}:
\begin{equation}\label{upper}
\begin{aligned}
&(X+1)^{\beta+\frac{3}{2}}\int_{0}^{\infty}|\phi_{X}|^{2}\frac{1}{u}d\psi +\frac{1}{2}\p_{X}\left[(X+1)^{\beta+\frac{3}{2}}\int_{0}^{\infty}|\phi_{\psi}|^{2}d\psi\right]\\
&\quad\quad =\frac{1}{2}\left(\beta+\frac{3}{2}\right)(X+1)^{\beta+\frac{1}{2}}\int_{0}^{\infty}|\phi_{\psi}|^{2}d\psi-(X+1)^{\beta+\frac{3}{2}}\int_{0}^{\infty}A\phi\phi_{X}\frac{1}{u}d\psi\\
&\quad\quad \leq  \frac{1}{2}\left(\beta+\frac{3}{2}\right)(X+1)^{\beta+\frac{1}{2}}\int_{0}^{\infty}|\phi_{\psi}|^{2}d\psi+ (X+1)^{\beta+\frac{3}{2}}\int_{0}^{\infty}\frac{ |\phi_{X}|^{2}+|A\phi|^{2}}{2u}d\psi\\
\end{aligned}
\end{equation}
Applying Lemma \ref{L2decay}, and recalling the fact that $ |A(X,\psi)|\lesssim (X+1)^{-1} $.
\begin{equation}\label{H1energY}
\begin{aligned}
&\p_{X}\left[(X+1)^{\beta+\frac{3}{2}}\int_{0}^{\infty}|\phi_{\psi}|^{2}d\psi\right]+(X+1)^{\beta+\frac{3}{2}}\int_{0}^{\infty}|\phi_{X}|^{2}\frac{1}{u}d\psi\\
&\quad\quad \lesssim (X+1)^{\beta+\frac{1}{2}}\int_{0}^{\infty}|\phi_{\psi}|^{2}d\psi+(X+1)^{-\frac{1}{2}}.
\end{aligned}
\end{equation}
Integrating \eqref{H1energY} with respect to $ X $, and applying \eqref{middle},  we obtain \eqref{H1}.
\end{proof}

\section{Proof of Main Results}\label{proofmain}
In this section, we present our main result, Theorem \ref{main}, which asserts:
\begin{equation*}
\|u-\bar{u}\|_{L^{\infty}_{\psi}} \lesssim (X+1)^{-1},\ \text{and}\ \|u-\bar{u}\|_{L^{\infty}_{y}} \lesssim (x+1)^{-1}.
\end{equation*}
To delineate the proof, we break it down into three key parts:
\begin{equation*}
\begin{aligned}
&\text{Step 1: Sharp $ H^{1} $ decay rates by using enhanced $ H^{1} $ decay rates stated in Section \ref{enhanced}.}\\
&\text{Step 2: Sharp $ L^{\infty} $ convergence rate in von Mises coordinates by Sobolev's embedding.}\\
&\text{Step 3: Sharp $ L^{\infty} $ convergence rate in physical coordinates through twisted subtraction.}
\end{aligned}
\end{equation*}
\subsection{Sharp $ H^{1} $ Decay Rates of $\phi$}

In this subsection, to derive precise $H^{1}$ decay rates, we improve the estimates from Lemma \ref{weight} and Lemma \ref{L2decay} --- Lemma \ref{H1decay}.

Firstly, we present the refined version of Lemma \ref{weight}.
\begin{proposition}\label{weightt}
Under the assumptions of Theorem \ref{main}, for any $X \in \mathbb{R}^{+}$, we have
\begin{equation}
\int_{0}^{\infty}\frac{1}{u}|\phi(X,\psi)|\psi d\psi \lesssim 1.
\end{equation}
\end{proposition}
\begin{proof}
The proof is identical to that of Lemma \ref{weight}. The only difference is that we can apply Lemma \ref{L2decay} and Lemma \ref{H1decay} to refine the estimates of \eqref{K111} and \eqref{K131}.
\end{proof}

Next, we establish the sharp $ L^{2} $ decay.
\begin{proposition}\label{l2decay}The following estimate holds:
\begin{equation}
\int_{0}^{\infty}\frac{\phi^{2}}{u}d\psi \lesssim (X+1)^{-\frac{3}{2}}.
\end{equation}
\end{proposition}
\begin{proof} In virtue of \eqref{four} with $ \alpha= 1$, we obtain
\begin{equation}\label{sharpinter}
\mathcal{A}(X)=\int_{0}^{\infty}\phi^{2}\frac{1}{u}d\psi\lesssim \max\left\{(X+1)^{\frac{1}{4}}\|\phi_{\psi}\|_{L^{2}_{\psi}}^{\frac{7}{5}}, \|\phi_{\psi}\|_{L^{2}_{\psi}}^{\frac{6}{5}}\right\}.
\end{equation}
By the energy balance \eqref{test}, we obtain the following differential inequality:
\begin{equation}\label{difff}
\mathcal{A}'(X)+\min \left\{(X+1)^{-\frac{5}{14}}\mathcal{A}^{\frac{10}{7}}(X),\mathcal{A}^{\frac{5}{3}}(X)\right\}\lesssim (X+1)^{-\frac{51}{20}}.
\end{equation}
which implies that for $ X > 0 $, 
\begin{equation}
\mathcal{A}(X)\lesssim (X+1)^{-\frac{3}{2}},
\end{equation}
by using the arguments analogous to those in the proof of \eqref{first} and \eqref{decayy}.
\end{proof}
We now estimate the decay rate of $ \| \phi_{\psi} \|_{L^{2}_{\psi}} $.

\begin{proposition}\label{Middlee}
For any $ X_{0} > 0 $, the following inequality holds:
\begin{equation}
\int_{0}^{X_{0}} \int_{0}^{\infty} (X+1)^{2} | \phi_{\psi} |^{2} d\psi dX \lesssim (X_{0}+1)^{\frac{1}{2}}.
\end{equation}
\end{proposition}

\begin{proof}
Testing \eqref{eqphi} with
\begin{equation*}
\phi \frac{1}{u} (X+1)^{2},
\end{equation*}
we obtain the following energy identity:
\begin{equation}\label{middlebalance}
\begin{aligned}
& \frac{\partial_{X}}{2} \left( \int_{0}^{\infty} \phi^{2} \frac{1}{u} d\psi (X+1)^{2} \right) - \int_{0}^{\infty} \phi^{2} \frac{1}{u} d\psi (X+1) + \frac{1}{2} \int_{0}^{\infty} \phi^{2} \frac{u_{X}}{u^{2}} d\psi (X+1)^{2} \\
& \quad\quad\quad\quad+ \int_{0}^{\infty} | \phi_{\psi} |^{2} d\psi (X+1)^{2} + \int_{0}^{\infty} A \phi^{2} \frac{1}{u} d\psi (X+1)^{2} = 0.
\end{aligned}
\end{equation}

We need to estimate
\begin{equation}\label{danger}
\int_{0}^{\infty} \phi^{2} \frac{u_{X}}{u^{2}} d\psi = \int_{0}^{\infty} \phi^{2} \frac{\bar{u}_{X}}{u^{2}} d\psi + \int_{0}^{\infty} \phi^{2} \frac{\rho_{X}}{u^{2}} d\psi.
\end{equation}

Using the properties of the Blasius profile and Proposition \ref{l2decay}, we have
\begin{equation*}
\begin{aligned}
| \eqref{danger}_{1} | & = \left|\int_{0}^{\infty} \phi^{2} \frac{\bar{u}_{X}}{u^{2}} d\psi\right| \lesssim (X+1)^{-1} \int_{0}^{\infty} \phi^{2} \frac{1}{u} d\psi \lesssim (X+1)^{-\frac{5}{2}}.
\end{aligned}
\end{equation*}

Applying Theorem \ref{iyer}, Lemma \ref{sobolev}, Lemma \ref{L2decay} and Lemma \ref{H1decay}, we have:
\begin{equation*}
\begin{aligned}
|\eqref{danger}_{2}|&\leq \left(\int_{0}^{\infty}\frac{\phi^{4}}{u^{2}}d\psi\right)^{\frac{1}{2}}\left(\int_{0}^{\infty}\frac{|\rho_{X}|^{2}}{u^{2}}d\psi\right)^{\frac{1}{2}}\\
&\lesssim  \left\|\frac{\phi}{u^{\frac{1}{2}}}\right\|_{L^{\infty}_{\psi}}\left(\int_{0}^{\infty}\frac{\phi^{2}}{u}d\psi\right)^{\frac{1}{2}}\left(\int_{0}^{\infty}\left[\frac{|\phi_{X}|^{2}}{u^{4}}+\frac{|\phi|^{2}|\bar{u}_{X}|^{2}}{u^{6}}\right]d\psi\right)^{\frac{1}{2}}\\
&\lesssim (X+1)^{-\frac{5}{2}}+(X+1)^{-1}\|\phi_{X\psi}\|_{L^{2}_{\psi}}.
\end{aligned}
\end{equation*}
Finally, recalling the energy balance \eqref{middlebalance} and employing Theorem \ref{iyer}, Proposition \ref{l2decay}, we obtain that
\begin{equation*}
\int_{0}^{X_{0}}\int_{0}^{\infty}(X+1)^{2}|\phi_{\psi}|^{2}d\psi dX\lesssim (X_{0}+1)^{\frac{1}{2}}.
\end{equation*}
Hence we finish the proof of Proposition \ref{Middlee}.
\end{proof}
\begin{proposition}\label{h1decay}
For any $ X_{0}>0 $,
\begin{equation}
(X_{0}+1)^{3}\int_{0}^{\infty}|\phi_{\psi}(X_{0},\psi)|^{2}d\psi+\int_{0}^{X_{0}}\int_{0}^{\infty}|\phi_{X}|^{2}\frac{1}{u}d\psi (X+1)^{3}dX\lesssim (X_{0}+1)^{\frac{1}{2}}.
\end{equation}
\end{proposition}
\begin{proof}Testing \eqref{eqphi} by
\begin{equation*}
\phi_{X}\frac{1}{u}(X+1)^{3}.
\end{equation*}
Then we only need to apply Proposition \ref{l2decay}, Proposition \ref{Middlee}, and mimic the proof of Lemma \ref{H1decay}; hence, we omit the details.
\end{proof}
\subsection{Sharp $ L^{\infty} $ Convergence Rates between $u$ and $\bar{u}$(Theorem \ref{main})}\label{pfmain1}
In this subsection, we demonstrate our main result, Theorem \ref{main}, which directly follows from Theorem \ref{vonmisedecay} and Theorem \ref{physicaldecay}.
We start by determining the decay rate of $ |u-\bar{u}| $ in von Mises Coordinates.
\begin{theorem}\label{vonmisedecay} For any $ (X,\psi) \in [0,+\infty) \times [0,+\infty) $, we have
\begin{equation}\label{infty}
|u(X,\psi)-\bar{u}(X,\psi)|\lesssim
\begin{cases}
(X+1)^{-1},\quad \psi\geq\sqrt{X+1};\\
\psi^{\frac{1}{4}}(X+1)^{-\frac{9}{8}},\quad \psi\leq\sqrt{X+1}.
\end{cases}
\end{equation}
\end{theorem}
\begin{proof}By the definition of $ \phi $,
\begin{equation*}
|u(X,\psi)-\bar{u}(X,\psi)|\leq\frac{|\phi(X,\psi)|}{\bar{u}(X,\psi)}.
\end{equation*}
If $ \psi\geq\sqrt{X+1} $, by applying Agmon's inequality(Lemma \ref{sobolev}), Proposition \ref{l2decay}, and Proposition \ref{h1decay}, we get:
\begin{equation*}
\begin{aligned}
|u(X,\psi)-\bar{u}(X,\psi)|&\lesssim |\phi(X,\psi)|\lesssim \|\phi\|_{L^{2}_{\psi}}^{\frac{1}{2}}\|\phi_{\psi}\|_{L^{2}_{\psi}}^{\frac{1}{2}}\\
&\lesssim (X+1)^{-\frac{1}{2}\times\frac{3}{4}}(X+1)^{-\frac{1}{2}\times \frac{5}{4}}\\
&\lesssim (X+1)^{-1}.
\end{aligned}
\end{equation*}
To obtain the desired pointwise estimate in the region $ \psi\leq\sqrt{X+1} $, the estimate of $ \|\phi_{\psi}\|_{L^{2}_{\psi}} $ is not sufficient, as this quantity only provides the estimate of $ \|\phi\|_{\dot{C}^{\frac{1}{2}}_{\psi}} $. Therefore, we require the following additional estimate, which can be considered as an \emph{$ L^{4} $ energy estimate}.
\begin{theorem}\label{l4decay}
The following estimate holds:
\begin{equation}
\int_{0}^{\infty}|\phi_{\psi}|^{4}d\psi \lesssim (X+1)^{-\frac{11}{2}}.
\end{equation}
\end{theorem}
Now we conclude the proof of Theorem \ref{vonmisedecay}. For $ \psi \leq \sqrt{X+1} $, utilizing Theorem \ref{l4decay} and Sobolev embedding $ \dot{W}^{1,4}_{\psi}\rightarrow \dot{C}^{\frac{3}{4}}_{\psi} $, we obtain that
\begin{equation*}
\begin{aligned}
|u(X,\psi)-\bar{u}(X,\psi)|&\lesssim (X+1)^{\frac{1}{4}}\frac{|\phi(X,\psi)|}{\psi^{\frac{1}{2}}}\\
&\leq (X+1)^{\frac{1}{4}}\psi^{\frac{1}{4}}\|\phi\|_{\dot{C}^{\frac{3}{4}}_{\psi}}\\
&\lesssim (X+1)^{\frac{1}{4}}\psi^{\frac{1}{4}}\|\phi_{\psi}\|_{L^{4}_{\psi}}\\
&\lesssim \psi^{\frac{1}{4}}(X+1)^{-\frac{9}{8}}.
\end{aligned}
\end{equation*}
This concludes the proof of \eqref{infty}.
\end{proof}
We now finish the proof of Theorem \ref{l4decay}.
\begin{proof}[Proof of Theorem \ref{l4decay}]
The proof is divided into two steps.

\textit{Step 1.}
First we demonstrate that for any $ X_{0}>0 $:
\begin{equation}\label{l4decay1}
\int_{0}^{X_{0}}\int_{0}^{\infty}|\phi_{\psi}|^{4}d\psi (X+1)^{5}dX \lesssim (X_{0}+1)^{\frac{1}{2}}.
\end{equation}
By testing \eqref{eqphi} with
\begin{equation*}
\phi\phi_{\psi}^{2}\frac{1}{u}(X+1)^{5},
\end{equation*}
and integrating by parts, we derive the following $ L^{4}$-energy inequality:
\begin{equation}
\begin{aligned}
& \frac{1}{3}\int_{0}^{\infty}|\phi_{\psi}|^{4}d\psi (X+1)^{5}+\int_{0}^{\infty}A|\phi|^{2}|\phi_{\psi}|^{2}\frac{1}{u}d\psi (X+1)^{5}\\
&\quad\quad \leq \left|\int_{0}^{\infty} \phi_{X}\phi \phi_{\psi}^{2}\frac{1}{u}d\psi\right| (X+1)^{5}\\
&\quad\quad \leq \frac{1}{6} \int_{0}^{\infty} |\phi_{\psi}|^{4}d\psi(X+1)^{5}+\frac{3}{2}\int_{0}^{\infty}|\phi_{X}|^{2}|\phi|^{2}\frac{1}{u^{2}}d\psi (X+1)^{5}.
\end{aligned}
\end{equation}
Notice that,
\begin{equation*}
\begin{aligned}
\int_{0}^{\infty}|\phi_{X}|^{2}|\phi|^{2}\frac{1}{u^{2}}d\psi\leq \left\|\frac{\phi}{u}\right\|_{L^{\infty}_{\psi}}^{2}\int_{0}^{\infty}|\phi_{X}|^{2}d\psi\lesssim (X+1)^{-2}\int_{0}^{\infty}|\phi_{X}|^{2}d\psi.
\end{aligned}
\end{equation*}
Combining this with Proposition \ref{h1decay}, we get \eqref{l4decay1}.

\textit{Step 2.} Now we conclude the proof of Theorem \ref{l4decay}. By testing \eqref{eqphi} with
\begin{equation*}
\phi_{X}\phi_{\psi}^{2}\frac{1}{u}(X+1)^{6},
\end{equation*}
and integrating by parts, we obtain the following energy inequality:
\begin{equation}\label{l4decay2}
\begin{aligned}
&\frac{\partial_{X}}{12}\left(\int_{0}^{\infty}|\phi_{\psi}|^{4}d\psi (X+1)^{6}\right)+\int_{0}^{\infty}|\phi_{X}|^{2}|\phi_{\psi}|^{2}\frac{1}{u}d\psi (X+1)^{6}\\
&\quad=\frac{1}{2}\int_{0}^{\infty}|\phi_{\psi}|^{4}d\psi (X+1)^{5}-\int_{0}^{\infty}A\phi\phi_{X}\phi_{\psi}^{2}\frac{1}{u}d\psi (X+1)^{6}\\
&\quad\leq \frac{1}{2}\int|\phi_{\psi}|^{4}(X+1)^{5}+\frac{1}{2}\int|\phi_{X}|^{2}|\phi_{\psi}|^{2}\frac{1}{u}(X+1)^{6}+\frac{1}{2}\int|A|^{2}|\phi|^{2}|\phi_{\psi}|^{2}\frac{1}{u}(X+1)^{6}.
\end{aligned}
\end{equation}
Utilizing Proposition \ref{l2decay} and Proposition \ref{h1decay}, we have
\begin{equation*}
\int|A|^{2}|\phi|^{2}|\phi_{\psi}|^{2}\frac{1}{u}\leq \|A\|_{L^{\infty}_{\psi}}^{2}\left\|\frac{\phi}{u^{\frac{1}{2}}}\right\|_{L^{\infty}_{\psi}}^{2}\int |\phi_{\psi}|^{2}\lesssim (X+1)^{-\frac{13}{2}}.
\end{equation*}
Combining \eqref{l4decay1} and \eqref{l4decay2}, we obtain
\begin{equation*}
\int_{0}^{\infty}|\phi_{\psi}|^{4}d\psi(X+1)^{6}\lesssim (X+1)^{\frac{1}{2}},
\end{equation*}
which establishes Theorem \ref{l4decay}.
\end{proof}
\begin{remark}
The typical method for improving the regularity of $ \phi $ involves differentiating \eqref{eqphi} with respect to $ X $ and estimating $ \|\phi_{X}\|_{L^{2}_{\psi}} $, as well as $ \|\phi_{\psi\psi}\|_{L^{2}_{\psi}} $. However, given that we are working with a \emph{low regularity framework}, differentiating \eqref{eqphi} could introduce additional obstacles. Therefore, the benefit of Theorem \ref{l4decay} lies in avoiding such additional complications.
\end{remark}
Now, we establish the convergence rate in physical coordinates.

\begin{theorem}\label{physicaldecay}
In the physical coordinates, the following estimate holds:
\begin{equation*}
\|u-\bar{u}\|_{L^{\infty}_{y}} \lesssim (x+1)^{-1}.
\end{equation*}
\end{theorem}

\begin{proof}
For $x, y \geq 0$, by the mean value theorem, we obtain:
\begin{equation}\label{twist}
\begin{aligned}
|u(x,y) - \bar{u}(x,y)| &\leq |u(x,y) - \bar{u}(x,y^{*})| + |\bar{u}(x,y^{*}) - \bar{u}(x,y)| \\
&\leq |u(x,y) - \bar{u}(x,y^{*})| + \bar{u}_{y}(x,\hat{y}) |y - y^{*}|,
\end{aligned}
\end{equation}
where
\begin{equation*}
y^{*} = \int_{0}^{\psi} \frac{1}{\bar{u}(x,\psi')} d\psi', \quad \psi = \psi(x,y;u) = \int_{0}^{y} u(x,s) ds, \quad \hat{y} \in [\min\{y,y^{*}\},\max\{y,y^{*}\}].
\end{equation*}
By using \eqref{infty}, we get:
\begin{equation*}
|u(x,y)-\bar{u}(x,y^{*})|=|u(X,\psi)-\bar{u}(X,\psi)|_{X=x}\lesssim (x+1)^{-1}.
\end{equation*}
For the second term, by utilizing Lemma \ref{comparison} and the property that $f''(\cdot)$ is decreasing, we can find a constant $c > 0$ such that:
\begin{equation*}
\bar{u}_{y}(x,\hat{y})=\frac{1}{\sqrt{x+1}}f''\left(\frac{\hat{y}}{\sqrt{x+1}}\right)\leq \frac{1}{\sqrt{x+1}}f''\left(\frac{cy^{*}}{\sqrt{x+1}}\right).
\end{equation*}
By applying Lemma \ref{Blasius}, we obtain:
\begin{equation}\label{mean}
\bar{u}_{y}(x,\hat{y}) \leq \frac{1}{\sqrt{x+1}}f''\left(\frac{cy^{*}}{\sqrt{x+1}}\right) \leq \frac{1}{\sqrt{x+1}}\exp\left[-\frac{c^{2}(y^{*})^{2}}{x+1}\right].
\end{equation}
It is worth noting that:
\begin{equation*}
\begin{aligned}
|y-y^{*}|&=\left|\int_{0}^{\psi}\frac{1}{u(X,\psi')}d\psi'-\int_{0}^{\psi}\frac{1}{\bar{u}(X,\psi')}d\psi'\right|\\
&\leq \left|\int_{0}^{\sqrt{X+1}}\frac{1}{u(X,\psi')}d\psi'-\int_{0}^{\sqrt{X+1}}\frac{1}{\bar{u}(X,\psi')}d\psi'\right|\\
&\quad\quad+\left|\int_{\sqrt{X+1}}^{\psi}\frac{1}{u(X,\psi')}d\psi'-\int_{\sqrt{X+1}}^{\psi}\frac{1}{\bar{u}(X,\psi')}d\psi'\right|\\
&=: B_{1}+B_{2}.
\end{aligned}
\end{equation*}
Let us start by estimating the term $ B_{1} $. Using \eqref{infty}, we get:
\begin{equation*}
\begin{aligned}
\left| \int_{0}^{\sqrt{X+1}} \left( \frac{1}{u(X,\psi')} - \frac{1}{\bar{u}(X,\psi')} \right) d\psi' \right| & \leq \int_{0}^{\sqrt{X+1}} \frac{|u(X,\psi') - \bar{u}(X,\psi')|}{u(X,\psi')\bar{u}(X,\psi')} d\psi' \\
& \lesssim (X+1)^{\frac{1}{2}} \int_{0}^{\sqrt{X+1}} \frac{|u(X,\psi') - \bar{u}(X,\psi')|}{\psi'} d\psi' \\
& \lesssim (X+1)^{-\frac{5}{8}} \int_{0}^{\sqrt{X+1}} \psi^{-\frac{3}{4}}d\psi \\
& \lesssim (X+1)^{-\frac{1}{2}}.
\end{aligned}
\end{equation*}

Next, let us consider the term $B_{2}$. For $\psi \geq \sqrt{X+1}$, we have $u(X,\psi), \bar{u}(X,\psi) \gtrsim 1$. Hence:
\begin{equation*}
\begin{aligned}
B_{2} & \leq \int_{\sqrt{X+1}}^{\psi} \frac{|u(X,\psi') - \bar{u}(X,\psi')|}{u(X,\psi')\bar{u}(X,\psi')} d\psi' \lesssim (X+1)^{-1} \int_{0}^{\psi} d\psi' \leq \psi (X+1)^{-1}.
\end{aligned}
\end{equation*}

In conclusion, using \eqref{twist} and \eqref{mean}, we obtain:
\begin{equation*}
\begin{aligned}
|u(x,y) - \bar{u}(x,y)| & \lesssim (x+1)^{-1} + e^{-\frac{c^{2}(y^{*})^{2}}{x+1}} \left[ (x+1)^{-1} + \frac{\psi(x,y;u)}{(x+1)^{\frac{3}{2}}} \right] \\
& \lesssim (x+1)^{-1} + e^{-\frac{c^{2}(y^{*})^{2}}{x+1}} \left[ (x+1)^{-1} + (x+1)^{-1} f\left(\frac{y^{*}}{\sqrt{x+1}}\right) \right] \\
& \lesssim (x+1)^{-1}.
\end{aligned}
\end{equation*}

Thus, we have completed the proof of Theorem \ref{main}.
\end{proof}
\subsection{Sharp $ L^{\infty} $ Convergence Rates between $u$ and $\bar{u}$(Theorem \ref{main2})}\label{pfmain2}
In this subsection, we finish the proof of Theorem \ref{main2}. By the discussion in subsection \ref{pfmain1}, we only need to show that
\begin{equation}\label{mainestimate}
|u(X,\psi)-\bar{u}(X,\psi)|\lesssim M(X,\psi).
\end{equation}
Where
\begin{equation}\label{mxpsi}
M(X,\psi)=
\begin{cases}
(X+1)^{-1},\quad \psi\geq\sqrt{X+1};\\
\psi^{\frac{1}{4}}(X+1)^{-\frac{9}{8}},\quad \psi\leq\sqrt{X+1}.
\end{cases}
\end{equation}
\eqref{mainestimate} is verified by the following series of lemmas.
\begin{lemma}\label{mainlm1}
In von Mises coordinate, define
\begin{equation*}
w(X,\psi)=u^{2}(X,\psi),\quad \bar{w}(X,\psi)=\bar{u}^{2}(X,\psi),
\end{equation*}
then
\begin{equation}
|w(0,\psi)-\bar{w}(0,\psi)|\lesssim \frac{\psi}{1+\psi^{4}}.
\end{equation}
\end{lemma}
\begin{proof}
Define
\begin{equation*}
y_{1}(\psi)=y(X,\psi;u)|_{X=0},\quad y_{2}(\psi)=y(X,\psi;\bar{u})|_{X=0}=f^{-1}(\psi).
\end{equation*}
By Oleinik's conditions, we have
\begin{equation}\label{inversebound}
\begin{cases}
C_{1}\sqrt{\psi}\leq y_{i}(\psi)\leq C_{2}\sqrt{\psi},\ \psi\leq 1,\\
C_{3}\psi\leq y_{i}(\psi)\leq C_{4}\psi,\ \psi\geq 1,
\end{cases}
\quad i=1,2.
\end{equation}
According to the definition of von Mises transformation,
\begin{equation*}
\begin{aligned}
w(0,\psi)-\bar{w}(0,\psi)&=u_{0}^{2}\left(y_{1}(\psi)\right)-\bar{u}_{0}^{2}\left(y_{2}(\psi)\right)\\
&=\underbrace{u_{0}^{2}\left(y_{1}(\psi)\right)-\bar{u}_{0}^{2}\left(y_{1}(\psi)\right)}_{D_{1}}+\underbrace{\bar{u}_{0}^{2}\left(y_{1}(\psi)\right)-\bar{u}_{0}^{2}\left(y_{2}(\psi)\right)}_{D_{2}}.
\end{aligned}
\end{equation*}
When $\psi\leq 1$,
\begin{equation*}
\begin{aligned}
|w(0,\psi)-\bar{w}(0,\psi)|&\leq u_{0}^{2}\left(y_{1}(\psi)\right)+\bar{u}_{0}^{2}\left(y_{2}(\psi)\right)\\
&\lesssim \left(y_{1}(\psi)\right)^{2}+\left(y_{2}(\psi)\right)^{2}\lesssim \psi.
\end{aligned}
\end{equation*}
When $\psi\geq 1$, the term $D_{1}$ can be estimated directly by \eqref{additional}:
\begin{equation*}
|D_{1}|\lesssim \frac{1}{y_{1}^{3}(\psi)}\lesssim \frac{1}{\psi^{3}}.
\end{equation*}
For the $D_{2}$ term, by mean value theorem, and the fast decay of $\p_{y}\bar{u}_{0}(y)$:
\begin{equation*}
|D_{2}|\leq 2\left|\bar{u}_{0}(\hat{y})\p_{y}\bar{u}_{0}(\hat{y})[y_{1}(\psi)-y_{2}(\psi)]\right|\lesssim \frac{1}{\psi^{3}},
\end{equation*}
where $\hat{y}\in [\min\{y_{1}(\psi),y_{2}(\psi)\},\ \max\{y_{1}(\psi),y_{2}(\psi)\}]$. Hence we finish the proof of Lemma \ref{mainlm1}.
\end{proof}
Recall that in von Mises coordinate, $w=u^{2}$ satisfies the following quasilinear parabolic equation:
\begin{equation}\label{quasiparabolic}
\begin{cases}
w_{X}-\sqrt{w}w_{\psi\psi}=0;\\
w(X,0)=0,\quad w(X,\psi)|_{\psi\uparrow\infty}=1.
\end{cases}
\end{equation}
Now define
\begin{equation}\label{initialdesign}
g_{-}(\psi)=\bar{w}(0,\psi)-\kappa B(\psi),\ g_{+}(\psi)=\bar{w}(0,\psi)+\kappa B(\psi),\ B(\psi)=\frac{\psi}{1+\psi^{4}}.
\end{equation}
Where $0<\kappa\ll1$. 
\begin{remark}The function $B(\psi)$ is designed to satisfy the parabolic compatibility conditions of \eqref{quasiparabolic}.
\end{remark}
Let $w_{-}(X,\psi), w_{+}(X,\psi)$ be the solutions of \eqref{quasiparabolic} with initial data $g_{-}(\psi), g_{+}(\psi)$ respectively. Through the discussion in previous sections, we obtain that
\begin{lemma}\label{mainlm2}There exists $0<\kappa_{0}\ll 1$, such that for $h\in\{\sqrt{w_{-}},\ \sqrt{w_{+}}\}$,
\begin{equation}
|h(X,\psi)-\bar{u}(X,\psi)|\lesssim M(X,\psi),
\end{equation}
whenever $\kappa \leq \kappa_{0}$.
\end{lemma}
\begin{proof}By the discussion in previous sections, the conclusion of Theorem \ref{vonmisedecay} still holds provided
\begin{equation*}
\|\phi_{0}(1+\psi)\|_{L^{1}_{\psi\in\R^{+}}}+\|\p_{\psi}^{2}\phi_{0}(1+\psi)\|_{L^{2}_{\psi\in\R^{+}}}\leq\kappa\ll 1.
\end{equation*}
Hence Lemma \ref{mainlm2} follows from \eqref{initialdesign}.
\end{proof}
To finish the proof of Theorem \ref{main2}, we introduce a scaling technique to obtain the two-sided bound of $u(X,\psi)$. For positive $\lambda \ll 1, \Lambda \gg 1$, we set for $\psi\geq 0$,
\begin{equation*}
g_{-}^{r}(\psi):=g_{-}(\lambda\psi),\ g_{+}^{r}(\psi):=g_{+}(\Lambda\psi).
\end{equation*}
Therefore, $g_{-}^{r}, g_{+}^{r}$ are the rescaled functions corresponding to $g_{-},g_{+}$ respectively. The following lemma is crucial to derive the two-sided bound of $u(X,\psi)$.
\begin{lemma}\label{mainlm3}
There exist positive $\lambda_{0}\ll 1$ and $\Lambda_{0}\gg 1$, such that
\begin{equation}
g^{r}_{-}(\psi)\leq w(0,\psi)\leq g^{r}_{+}(\psi),
\end{equation}
whenever $\lambda\leq \lambda_{0}, \Lambda\geq \Lambda_{0}$.
\end{lemma}
\begin{proof}The proof of Lemma \ref{mainlm3} is based on a case-by-case discussion. 

\textit{Step 1.} First we show that $g_{-}^{r}(\psi)\leq w(0,\psi)$ for $\lambda\ll1$. By Oleinik's conditions, there exists $a>0,\psi_{0}>0$, such that
\begin{equation}
\begin{cases}
\p_{\psi}w(0,\psi)\geq a,\quad 0\leq\psi\leq\psi_{0},\\
w(0,\psi)\geq a,\quad \psi\geq \psi_{0}.
\end{cases}
\end{equation}
(1a)\ When $\psi\leq \psi_{0}$, 
\begin{equation*}
g_{-}^{r}(\psi)\leq \bar{w}(0,\lambda\psi)\leq C\lambda \psi.
\end{equation*}
(1b)\ When $\psi_{0}\leq \psi\leq \lambda^{-\frac{1}{2}}$,
\begin{equation*}
g_{-}^{r}(\psi)\leq \bar{w}(0,\lambda\psi)\leq C\lambda\psi\leq C\sqrt{\lambda}.
\end{equation*}
Hence,
\begin{equation*}
g_{-}^{r}(\psi)\leq w(0,\psi),\quad 0\leq\psi\leq  \lambda^{-\frac{1}{2}},
\end{equation*}
provided $\lambda\ll 1$.\\
(1c)\ When $\psi\geq \lambda^{-\frac{1}{2}}$, applying Lemma \ref{mainlm1},
\begin{equation*}
\begin{aligned}
w(0,\psi)-g_{-}^{r}(\psi)&\geq \bar{w}(0,\psi)-C\frac{\psi}{1+\psi^{5}}-\bar{w}(0,\lambda\psi)+\kappa B(\lambda\psi)\\
&\geq \kappa B(\lambda\psi)-C\frac{\psi}{1+\psi^{5}}\\
&=\left(\kappa\lambda\frac{1+\psi^{5}}{1+(\lambda\psi)^{4}}-C\right)\frac{\psi}{1+\psi^{5}}\\
&\geq\left(\kappa\lambda\frac{1+\lambda^{-\frac{5}{2}}}{1+\lambda^{2}}-C\right)\frac{\psi}{1+\psi^{5}}\\
&\geq 0,
\end{aligned}
\end{equation*}
provided
\begin{equation*}
\kappa\lambda^{-\frac{3}{2}}\geq 2C.
\end{equation*}

\textit{Step 2.} Next we show that $w(0,\psi)\leq g_{+}^{r}(\psi)$ for $\Lambda\gg 1$.\\
(2a)\ When $\psi\leq \Lambda^{-1}$, 
\begin{equation*}
w(0,\psi)\lesssim \psi,\ g_{+}^{r}(\psi)\geq \bar{w}(0,\Lambda\psi)\gtrsim \Lambda\psi.
\end{equation*}
(2b)\ When $\Lambda^{-1}\leq\psi\leq \Lambda^{-\frac{1}{2}}$,
\begin{equation*}
w(0,\psi)\lesssim \psi, g_{+}^{r}(\psi)\geq \bar{w}(0,\Lambda\psi)\geq \bar{w}(0,1).
\end{equation*}
Hence,
\begin{equation*}
w(0,\psi)\leq g_{+}^{r}(\psi),\quad 0\leq\psi\leq \Lambda^{-\frac{1}{2}},
\end{equation*}
provided $\Lambda\gg 1$.\\
(2c)\ When $\psi\geq \Lambda^{-\frac{1}{2}}$, by the Property of Blasius profile,
\begin{equation*}
\begin{aligned}
g_{+}^{r}(\psi)&=\bar{w}(0,\Lambda\psi)+\kappa B(\Lambda \psi)\\
&\geq 1-e^{-c(\Lambda\psi)^{2}}+\kappa\Lambda\frac{\psi}{1+(\Lambda\psi)^{4}}\\
&\geq 1\geq w(0,\psi),
\end{aligned}
\end{equation*}
whenever $\Lambda\gg 1$.
Hence we conclude Lemma \ref{mainlm3}.
\end{proof}
Now let $w_{-}^{r}(X,\psi)$ and $w_{+}^{r}(X,\psi)$ be the solutions of \eqref{quasiparabolic} with initial data $g_{-}^{r}(\psi), g_{+}^{r}(\psi)$ respectively. By the scaling invariance of \eqref{quasiparabolic}, we have
\begin{equation}\label{scaling}
w_{-}^{r}(X,\psi)=w_{-}(\lambda^{2}X,\lambda\psi),\ w_{+}^{r}(X,\psi)=w_{+}(\Lambda^{2} X,\Lambda\psi).
\end{equation}
Furthermore, by virtue of Lemma \ref{mainlm3} and the standard parabolic comparison principle, we obtain that
\begin{equation}\label{twoside}
\sqrt{w_{-}^{r}}(X,\psi)\leq u(X,\psi)\leq \sqrt{w_{+}^{r}}(X,\psi).
\end{equation}
Finally, applying Lemma \ref{mainlm2}, \eqref{scaling}, \eqref{twoside}, and Lemma \ref{mainlm4} below, we obtain that
\begin{equation*}
\begin{aligned}
|u(X,\psi)-\bar{u}(X,\psi)|&\leq |\sqrt{w_{-}^{r}}(X,\psi)-\bar{u}(X,\psi)|+|\sqrt{w_{+}^{r}}(X,\psi)-\bar{u}(X,\psi)|\\
&\leq |\sqrt{w_{-}^{r}}(X,\psi)-\bar{u}(\lambda^{2}X,\lambda\psi)|+|\sqrt{w_{+}^{r}}(X,\psi)-\bar{u}(\Lambda^{2}X,\Lambda\psi)|\\
&\quad +|\bar{u}(\lambda^{2}X,\lambda\psi)-\bar{u}(X,\psi)|+|\bar{u}(\Lambda^{2}X,\Lambda\psi)-\bar{u}(X,\psi)|\\
&\lesssim M(\lambda^{2} X,\lambda \psi)+M(\Lambda^{2} X,\Lambda \psi)\\
&\quad +|\bar{u}(\lambda^{2}X,\lambda\psi)-\bar{u}(X,\psi)|+|\bar{u}(\Lambda^{2}X,\Lambda\psi)-\bar{u}(X,\psi)|\\
&\lesssim_{\lambda,\Lambda}M(X,\psi),
\end{aligned}
\end{equation*}
which shows the vadility of \eqref{mainestimate}, hence we finish the proof of Theorem \ref{main2}.
\begin{lemma}\label{mainlm4}
For any $\mu>0$, we have
\begin{equation}
|\bar{u}(\mu^{2}X,\mu\psi)-\bar{u}(X,\psi)|\lesssim_{\mu} M(X,\psi).
\end{equation}
\end{lemma}
\begin{proof}Just note that in the von Mises coordinate,
\begin{equation*}
\bar{u}(X,\psi)=f'\circ f^{-1}(\zeta),\quad \zeta=\frac{\psi}{\sqrt{X+1}}.
\end{equation*}
By virtue of the properties of Blasius solution(Lemma \ref{Blasius}), we finish the proof.
\end{proof}
\section*{Acknowledgement}
	Hao Jia is supported by NSF grant DMS-1945179. Zhen Lei is in part supported by NSFC (No. 11725102), Sino-German Center (No. M-0548), the National Key R\&D Program of China (No. 2018AAA0100303), National Support Program for Young Top-Notch Talents, Shanghai Science and Technology Program (No. 21JC1400600 and No. 19JC1420101) and New Cornerstone Science Foundation through the XPLORER PRIZE. Cheng Yuan is supported by NSFC (No. 123B2008).

\appendix
\section{Appendix: Proof of Theorem \ref{iyer}}\label{pfiyer}
In this section, we present the proof of Theorem \ref{iyer}. It is recommended that readers have familiarity with fundamental inequalities stated in Lemma \ref{sobolev}, weighted embedding inequalities, and energy estimation techniques as discussed in \cite{MR4097332}. 
In reference to Iyer's work \cite{MR4097332}, we define the following weighted energies:
\begin{equation}
\begin{aligned}
\|\phi\|_{E_{0}}^{2}&=\sup_{X\geq 0}\left\|\phi\frac{1}{\sqrt{u}}(1+\psi^{(\frac{1}{2}-\mu)})\right\|_{L^{2}_{\psi\in\R^{+}}}^{2}+\left\|\phi_{\psi}(1+\psi^{(\frac{1}{2}-\mu)})\right\|_{L^{2}_{X\in\R^{+}\psi\in\R^{+}}};\\
\|\phi\|_{E_{1}}^{2}&=\sup_{X\geq 0}\left\|\phi_{\psi}(1+\psi^{(\frac{1}{2}-\mu)})(X+1)^{(\frac{1}{2}-\theta)}\right\|_{L^{2}_{\psi\in\R^{+}}}^{2}\\
&\quad\quad+\left\|\phi_{X}\frac{1}{\sqrt{u}}(1+\psi^{(\frac{1}{2}-\mu)})(X+1)^{(\frac{1}{2}-\theta)}\right\|_{L^{2}_{X\in\R^{+},\psi\in\R^{+}}}^{2};\\
\|\phi\|_{E_{2}}^{2}&=\sup_{X\geq 0}\left\|\phi_{X}\frac{1}{\sqrt{u}}(1+\psi^{(\frac{1}{2}-\mu)})(X+1)^{(1-\theta)}\right\|_{L^{2}_{\psi\in\R^{+}}}^{2}\\
&\quad\quad+\left\|\phi_{X\psi}(1+\psi^{(\frac{1}{2}-\mu)})(X+1)^{(1-\theta)}\right\|_{L^{2}_{X\in\R^{+},\psi\in\R^{+}}}^{2}.
\end{aligned}
\end{equation}
We also define the following norms for initial data:
\begin{equation}
\begin{aligned}
\|\phi_{0}\|_{E_{0,0}}^{2}&=\int_{0}^{\infty}|\phi_{0}(\psi)|^{2}\frac{1}{u}(1+\psi^{(1-2\mu)})d\psi;\\
\|\phi_{0}\|_{E_{1,0}}^{2}&=\int_{0}^{\infty}|\phi_{0}'(\psi)|^{2}(1+\psi^{(1-2\mu)})d\psi;\\
\|\phi_{0}\|_{E_{2,0}}^{2}&=\int_{0}^{\infty}|\phi_{X}(0,\psi)|^{2}\frac{1}{u}(1+\psi^{(1-2\mu)})d\psi;\\
\|\phi_{0}\|_{E_{\text{initial}}}^{2}&=\sum_{j=0}^{2}\|\phi_{0}\|_{E_{j,0}}^{2}.
\end{aligned}
\end{equation}
Here we fix $\theta=\frac{1}{1000}, \mu=\frac{1}{500}$. We next set the functional space
\begin{equation*}
E_{\leq 2}=\left\{\phi: \|\phi\|_{E_{\leq 2}}:=\sum_{j=0}^{2}\|\phi\|_{E_{j}}<\infty\right\}.
\end{equation*}
Then, Theorem \ref{iyer} is a direct consequence of the following Lemma \ref{initial} --- Lemma \ref{lm63}.

First, we show that under the assumption \eqref{smallnorm}, the initial data (in von Mises coordinate) $\phi_{0}(\psi)$ lies in the space $E_{\text{initial}}$.

\begin{lemma}\label{initial}
Assume that the initial data in physical coordinate $u_{0}(y)$ satisfies
\begin{equation}\label{norm}
\sum_{j\leq 2}\left\|(1+y)\partial_{y}^{j}(u_{0}-\bar{u}_{0})(y)\right\|_{L^{2}(\R^{+})}\leq \kappa \ll 1.
\end{equation}
Then, under the von Mises transform, $\phi_{0}(\psi)$ satisfies 
\begin{equation}\label{smallinitial}
\|\phi_{0}\|_{E_{\text{initial}}}\lesssim \kappa \ll 1.
\end{equation}
\end{lemma}
\begin{comment}
Before proving the above conclusion, we require the following lemma.

\begin{lemma}\label{A2}
Under the assumptions of Theorem \ref{initial}, we have the following estimates:
\begin{align}
&|v_{0}(y)-\bar{v}_{0}(y)|\lesssim
\begin{cases}\label{a}
\kappa y^{2};\quad y\leq 1,\\
\kappa;\quad y\geq 1;
\end{cases}\\
&|u_{x}(0,y)-\bar{u}_{x}(0,y)|\lesssim
\begin{cases}\label{b}
\kappa y;\quad y\leq 1,\\
\kappa y^{-3};\quad y\geq 1;
\end{cases}\\
&|u_{xy}(0,y)-\bar{u}_{xy}(0,y)|\lesssim
\begin{cases}\label{c}
\kappa ;\quad y\leq 1,\\
\kappa y^{-3};\quad y\geq 1;
\end{cases}\\
&|u_{xyy}(0,y)-\bar{u}_{xyy}(0,y)|\lesssim
\begin{cases}\label{d}
\kappa ;\quad y\leq 1,\\
\kappa y^{-3};\quad y\geq 1.
\end{cases}
\end{align}
\end{lemma}
We acknowledge Lemma \ref{A2} temporarily and proceed to complete the proof of Theorem \ref{initial}.
\end{comment}
\begin{proof} For convenience, let us define
\begin{equation*}
y_{1}(\psi)=y(X,\psi;u)|_{X=0},\quad y_{2}(\psi)=y(X,\psi;\bar{u})|_{X=0}=f^{-1}(\psi),
\end{equation*}
where $ f $ is the solution of the Blasius equation \eqref{blasiusode}.

Firstly, we estimate the difference between $ y_{1}(\psi) $ and $ y_{2}(\psi) $. For $ \psi\leq1 $, by using \eqref{norm} and Sobolev embedding, we have:
\begin{equation*}
\int_{0}^{y_{1}(\psi)}(\bar{u}_{0}(y)-\kappa y)dy\leq\int_{0}^{y_{1}(\psi)}u_{0}(y)dy\leq\int_{0}^{y_{1}(\psi)}(\bar{u}_{0}(y)+\kappa y)dy,
\end{equation*}
we have
\begin{equation*}
f\big(y_{1}(\psi)\big)-\frac{\kappa}{2}\big(y_{1}(\psi)\big)^{2}\leq\psi\leq f\big(y_{1}(\psi)\big)+\frac{\kappa}{2}\big(y_{1}(\psi)\big)^{2}.
\end{equation*}
Thus,
\begin{equation*}
f^{-1}\left[\psi-\frac{\kappa}{2}\big(y_{1}(\psi)\big)^{2}\right]\leq y_{1}(\psi) \leq f^{-1}\left[\psi+\frac{\kappa}{2}\big(y_{1}(\psi)\big)^{2}\right].
\end{equation*}
Consequently,
\begin{equation}\label{minus1}
\begin{aligned}
&|y_{1}(\psi)-y_{2}(\psi)|\\
&\quad\leq \max \left\{f^{-1}\left[\psi+\frac{\kappa}{2}\big(y_{1}(\psi)\big)^{2}\right]-f^{-1}(\psi),\ f^{-1}(\psi)-f^{-1}\left[\psi-\frac{\kappa}{2}\big(y_{1}(\psi)\big)^{2}\right]\right\}\\
&\quad\lesssim \kappa \psi^{\frac{1}{2}}.
\end{aligned}
\end{equation}
For $\psi\geq 1$, similar to Lemma \ref{comparison}, we have:
\begin{equation*}
f\big(y_{1}(\psi)\big)-\kappa \leq \psi \leq f\big(y_{1}(\psi)\big)+\kappa.
\end{equation*}
Therefore,
\begin{equation}\label{minus2}
|y_{1}(\psi)-y_{2}(\psi)|\leq \max\left\{f^{-1}(\psi+\kappa)-f^{-1}(\psi),\ f^{-1}(\psi)-f^{-1}(\psi-\kappa)\right\}\lesssim \kappa,\quad \psi\geq 1.
\end{equation}
To conclude Lemma \ref{initial}, we need to demonstrate that
\begin{equation}\label{wait}
\int_{0}^{\infty}\phi_{0}^{2}\frac{\psi+1}{u_{0}(\psi)}d\psi,\ \int_{0}^{\infty}|\p_{\psi}\phi_{0}|^{2}\frac{\psi+1}{u_{0}(\psi)}d\psi,\ \int_{0}^{\infty}\left|(\p_{X}\phi)(0,\psi)\right|^{2}\frac{\psi+1}{u_{0}(\psi)}d\psi \lesssim \kappa^{2}.
\end{equation}
By abusing notations, here
\begin{equation*}
u_{0}(\psi):=u_{0}\big(y_{1}(\psi)\big).
\end{equation*}
Firstly, let us estimate the first term. Note that
\begin{equation*}
\int_{0}^{\infty}\phi_{0}^{2}\frac{\psi+1}{u_{0}(\psi)}d\psi\lesssim \int_{0}^{1}\phi_{0}^{2}(\psi)\frac{1}{\psi^{\frac{1}{2}}}d\psi +\int_{1}^{\infty}\phi_{0}^{2}(\psi)\psi d\psi=: A+B.
\end{equation*}

\textit{Estimate of $A$.} By the definition of $ \phi $,
\begin{align*}
&\int_{0}^{1}\phi_{0}^{2}(\psi)\frac{1}{\psi^{\frac{1}{2}}}d\psi=\int_{0}^{1}\left|u_{0}^{2}\big(y_{1}(\psi)\big)-\bar{u}_{0}^{2}\big(y_{2}(\psi)\big)\right|^{2}\frac{1}{\psi^{\frac{1}{2}}}d\psi\\
&\quad = \int_{0}^{1}\left|u_{0}^{2}\big(y_{1}(\psi)\big)-\bar{u}_{0}^{2}\big(y_{1}(\psi)\big)\right|^{2}\frac{1}{\psi^{\frac{1}{2}}}d\psi+\int_{0}^{1}\left|\bar{u}_{0}^{2}\big(y_{1}(\psi)\big)-\bar{u}_{0}^{2}\big(y_{2}(\psi)\big)\right|^{2}\frac{1}{\psi^{\frac{1}{2}}}d\psi.
\end{align*}
By a change of variable,
\begin{align*}
\int_{0}^{1}\left|u_{0}^{2}\big(y_{1}(\psi)\big)-\bar{u}_{0}^{2}\big(y_{1}(\psi)\big)\right|^{2}\frac{1}{\psi^{\frac{1}{2}}}d\psi&=\int_{0}^{y_{1}(1)}\left|u_{0}^{2}(y)-\bar{u}_{0}^{2}(y)\right|^{2}\frac{u_{0}(y)dy}{\displaystyle\left(\int_{0}^{y}u_{0}(s)ds\right)^{\frac{1}{2}}}\\
&\lesssim \|u_{0}-\bar{u}_{0}\|_{L^{2}_{y}}^{2}\lesssim \kappa^{2}.
\end{align*}
Utilizing the  mean-value theorem and the estimate of $|y_{1}(\psi)-y_{2}(\psi)|$,
\begin{align*}
\int_{0}^{1}\left|\bar{u}_{0}^{2}\big(y_{1}(\psi)\big)-\bar{u}_{0}^{2}\big(y_{2}(\psi)\big)\right|^{2}\frac{1}{\psi^{\frac{1}{2}}}d\psi&\lesssim \left\|\bar{u}_{0}^{2}\big(y_{1}(\psi)\big)-\bar{u}_{0}^{2}\big(y_{2}(\psi)\big)\right\|_{L^{\infty}_{\psi}}^{2}\\
&\lesssim \|y_{1}(\psi)-y_{2}(\psi)\|_{L^{\infty}_{\psi}}^{2}\lesssim \kappa^{2}.
\end{align*}

\textit{Estimate of $ B $.}
Similarly to the estimates of $ A $,
\begin{equation*}
B=\int_{1}^{\infty}\left|u_{0}^{2}\big(y_{1}(\psi)\big)-\bar{u}_{0}^{2}\big(y_{1}(\psi)\big)\right|^{2}\psi d\psi+\int_{1}^{\infty}\left|\bar{u}_{0}^{2}\big(y_{1}(\psi)\big)-\bar{u}_{0}^{2}\big(y_{2}(\psi)\big)\right|^{2}\psi d\psi.
\end{equation*}
By changing the variable again,
\begin{align*}
\int_{1}^{\infty}\left|u_{0}^{2}\big(y_{1}(\psi)\big)-\bar{u}_{0}^{2}\big(y_{1}(\psi)\big)\right|^{2}\psi d\psi&=\int_{y_{1}(1)}^{\infty}\left|u_{0}^{2}(y)-\bar{u}_{0}^{2}(y)\right|^{2}\left(\int_{0}^{y}u_{0}(s)ds\right)u_{0}(y)dy\\
&\lesssim \int_{y_{1}(1)}^{\infty}|u_{0}^{2}(y)-\bar{u}_{0}^{2}(y)|^{2}y dy\lesssim \kappa^{2}.
\end{align*}
By applying the Fundamental Theorem of Calculus, the estimate of $|y_{1}(\psi)-y_{2}(\psi)|$, and the rapid decay of $ \bar{u}_{0}'$:
\begin{align*}
&\int_{1}^{\infty}\left|\bar{u}_{0}^{2}\big(y_{1}(\psi)\big)-\bar{u}_{0}^{2}\big(y_{2}(\psi)\big)\right|^{2}\psi d\psi= \int_{1}^{\infty}\left|\int_{0}^{1}\frac{d}{dt}\bar{u}_{0}^{2}\left[y_{2}(\psi)+t\big(y_{1}(\psi)-y_{2}(\psi)\big)\right]dt \right|^{2}\psi d\psi \\
&\quad\lesssim \int_{1}^{\infty}\bigg(\int_{0}^{1}\left|\bar{u}_{0}'\left[y_{2}(\psi)+t\Big(y_{1}(\psi)-y_{2}(\psi)\Big)\right]\right|\times |y_{1}(\psi)-y_{2}(\psi)|dt\bigg)^{2}\psi d\psi\\
&\quad\lesssim \int_{1}^{\infty}\kappa^{2}\psi^{-2}d\psi \leq \kappa^{2}.
\end{align*}
Consequently,
\begin{equation}\label{A4}
\int_{0}^{\infty}\phi_{0}^{2}\frac{\psi+1}{u_{0}(\psi)}d\psi\lesssim \kappa^{2}.
\end{equation}
For other terms of \eqref{wait}, using Prandtl's equation \eqref{prandtl}, it is straightforward to compute that
\begin{equation}\label{A5}
\begin{cases}
\p_{\psi}\phi_{0}(\psi)&=2u_{0}'\big(y_{1}(\psi)\big)-2\bar{u}_{0}'\big(y_{2}(\psi)\big);\\
(\p_{X}\phi)(X,\psi)|_{X=0}&=2u_{0}''\big(y_{1}(\psi)\big)-2\bar{u}_{0}''\big(y_{2}(\psi)\big);\\
\end{cases}
\end{equation}
In virtue of \eqref{A5}, and following the same steps as in the estimate of \eqref{A4}, we complete the proof of Lemma \ref{initial}.
\end{proof}

 \begin{lemma}\label{lm61}
 There exists $\kappa\in(0,1)$ sufficiently small such that under the assumptions in Theorem \ref{iyer}, the solution of \eqref{eqphi} exists globally in space $ E_{\leq 2} $.
 \end{lemma}
 \begin{proof}In the proof, we sometimes adopt the following simplified notation
 \begin{equation*}
 \int \cdots = \int_{0}^{\infty}\cdots d\psi.
 \end{equation*}
 We organize the proof of Lemma \ref{lm61} into four steps.

\textit{Step 1. Estimation of $ E_{0} $.}
By testing equation \eqref{eqphi} with
\begin{equation*}
\phi\frac{1}{u}(1+\psi^{(1-2\mu)}),
\end{equation*}
and integrating by parts, we derive the following energy inequality:
\begin{equation}\label{E01}
\begin{aligned}
&\frac{\p_{x}}{2}\left(\int |\phi|^{2}\frac{1}{u}(1+\psi^{(1-2\mu)})d\psi\right)+\int |\phi_{\psi}|^{2}(1+\psi^{(1-2\mu)})d\psi\\
&\quad\quad +\int |\phi|^{2}(1+\psi^{(1-2\mu)})\left[\frac{A}{u}+\frac{u_{X}}{2u^{2}}\right]d\psi\leq 0.
\end{aligned}
\end{equation}
As demonstrated in the proof of Lemma \ref{weight},
\begin{equation}\label{E02}
\int  |\phi|^{2}(1+\psi^{(1-2\mu)})\left[\frac{A}{u}+\frac{u_{X}}{2u^{2}}\right]d\psi\geq \int |\phi|^{2}(1+\psi^{(1-2\mu)})\frac{1}{u^{3}}\left(\mathcal{R}_{1}+\mathcal{R}_{2}+\mathcal{R}_{3}\right).
\end{equation}
Recall the definition of $\mathcal{R}_{1},\mathcal{R}_{2},\mathcal{R}_{3}$ from \eqref{R}, the most critical term in \eqref{E02} is
\begin{equation}\label{E03}
\int |\phi|^{2}(1+\psi^{(1-2\mu)})\frac{1}{u^{2}}\rho_{X} d\psi,
\end{equation}
as \eqref{E03} involves the highest derivations of $ \phi $. Utilizing \eqref{rho} and Lemma \ref{Blasius}, we have
\begin{equation}\label{E04}
|\eqref{E03}|\lesssim \int |\phi|^{2}|\phi_{X}|\frac{1}{u^{3}}(1+\psi^{(1-2\mu)})+(X+1)^{-1}\int |\phi|^{3}\frac{1}{u^{3}}(1+\psi^{(1-2\mu)}).
\end{equation}
We concentrate on the first term of \eqref{E04} since it involves a first-order derivation. On the one hand , for $ \psi\geq \sqrt{X+1} $, we have
\begin{equation}\label{E05}
\begin{aligned}
\int_{\psi \geq \sqrt{X+1}} |\phi|^{2}|\phi_{X}|\frac{1}{u^{3}}(1+\psi^{(1-2\mu)})&\lesssim \int |\phi|^{2}|\phi_{X}|(1+\psi^{(1-2\mu)})\\
&\leq \left\|\phi (1+\psi^{(\frac{1}{2}-\mu)})\right\|_{L^{2}_{\psi}}\left\|\phi_{X}(1+\psi^{(\frac{1}{2}-\mu)})\right\|_{L^{2}_{\psi}}\|\phi\|_{L^{\infty}_{\psi}}\\
&\lesssim (X+1)^{-(\frac{5}{4}-\frac{3}{2}\theta)}\|\phi\|_{E_{0}}^{\frac{3}{2}}\|\phi\|_{E_{1}}^{\frac{1}{2}}\|\phi\|_{E_{2}}.
\end{aligned}
\end{equation}
On the other hand, for $ \psi \leq \sqrt{X+1} $, we have
\begin{equation}\label{E06}
\begin{aligned}
&\int _{\psi\leq \sqrt{X+1}}|\phi|^{2}|\phi_{X}|\frac{1}{u^{3}}(1+\psi^{(1-2\mu)})\\
&\leq \left(\int |\phi|^{2}\frac{1}{u}(1+\psi^{(1-2\mu)})\right)^{\frac{1}{2}} \left(\int |\phi_{X}|^{2}\frac{1}{u}(1+\psi^{(1-2\mu)})\right)^{\frac{1}{2}}\left\|\frac{\phi}{u^{2}}\right\|_{L^{\infty}_{\psi\leq\sqrt{X+1}}}\\
&\lesssim  \left(\int |\phi|^{2}\frac{1}{u}(1+\psi^{(1-2\mu)})\right)^{\frac{1}{2}} \left(\int |\phi_{X}|^{2}\frac{1}{u}(1+\psi^{(1-2\mu)})\right)^{\frac{1}{2}}\|\phi_{\psi}\|_{L^{\infty}_{\psi}}(X+1)^{\frac{1}{2}}.
\end{aligned}
\end{equation}
We estimate the quantity $ \|\phi_{\psi}\|_{L^{\infty}_{\psi}} $ via
\begin{equation}\label{E066}
\begin{aligned}
\|\phi_{\psi}\|_{L^{\infty}_{\psi}}^{2}&\lesssim \|\phi_{\psi}\|_{L^{2}_{\psi}}\|\phi_{\psi\psi}\|_{L^{2}_{\psi}}\\
&\leq \|\phi_{\psi}\|_{L^{2}_{\psi}}\bigg(\left\|\frac{\phi_{X}}{u}\right\|_{L^{2}_{\psi}}+\left\|\frac{A\phi}{u}\right\|_{L^{2}_{\psi}}\bigg)\\
&\lesssim \|\phi\|_{E_{1}}\left[(X+1)^{-(\frac{3}{2}-2\theta)}\|\phi\|_{E_{\leq 2}}+(X+1)^{\theta}\left\|\phi_{X\psi}\right\|_{L^{2}_{\psi}}\right].
\end{aligned}
\end{equation}
By combining \eqref{E01}, \eqref{E05}, \eqref{E06}, \eqref{E066}, and applying the H\"older's inequality in time, we conclude that 
\begin{equation}\label{E07}
\|\phi\|_{E_{0}}^{2}\leq \|\phi_{0}\|_{E_{0,0}}^{2}+\mathcal{C}_{1}\left(\|\phi\|_{E_{0}}^{\frac{3}{2}}\|\phi\|_{E_{1}}^{\frac{1}{2}}\|\phi\|_{E_{2}}+\|\phi\|_{E_{1}}^{2}\|\phi\|_{E_{2}}\right).
\end{equation}

\textit{Step 2. Estimation of $ E_{1} $.} First, by testing \eqref{eqphi} with
\begin{equation*}
\phi_{X}\frac{1}{u}(1+\psi^{(1-2\mu)})(1+X)^{(1-2\theta)},
\end{equation*}
we derive the following energy inequality:
\begin{equation}\label{E11}
\begin{aligned}
&\int |\phi_{X}|^{2}\frac{1}{u}(1+\psi^{(1-2\mu)})d\psi (1+X)^{(1-2\theta)}-\int \phi_{\psi\psi}\phi_{X}(1+\psi^{(1-2\mu)})d\psi(1+X)^{(1-2\theta)}\\
&\quad \quad +\int A\phi \phi_{X}\frac{1}{u}(1+\psi^{(1-2\mu)})d\psi (1+X)^{(1-2\theta)}=0.
\end{aligned}
\end{equation}
The third term of \eqref{E11} can be bounded as follows:
\begin{equation}\label{E12}
\begin{aligned}
&\int A\phi \phi_{X}\frac{1}{u}(1+\psi^{(1-2\mu)})d\psi (1+X)^{(1-2\theta)}\\
&\quad \leq \frac{1}{2} \int |\phi_{X}|^{2}\frac{1}{u}(1+\psi^{(1-2\mu)})d\psi (1+X)^{(1-2\theta)}\\
&\quad\quad+\frac{1}{2}\int |A|^{2}|\phi|^{2}\frac{1}{u}(1+\psi^{(1-2\mu)})d\psi (1+X)^{(1-2\theta)}\\
&\quad \leq \frac{1}{2} \int |\phi_{X}|^{2}\frac{1}{u}(1+\psi^{(1-2\mu)})d\psi (1+X)^{(1-2\theta)}+(X+1)^{-(1+2\theta)}\|\phi\|_{E_{0}}^{2}.
\end{aligned}
\end{equation}
Regarding the second term of \eqref{E11}, integrating by parts, we find
\begin{equation}\label{E13}
\begin{aligned}
\eqref{E11}_{2}&=\int \phi_{\psi}\phi_{X \psi}(1+\psi^{(1-2\mu)})d\psi (1+X)^{(1-2\theta)}\\
&\quad +(1-2\mu)\int \phi_{X}\phi_{\psi}\psi^{-2\mu}d\psi(1+X)^{(1-2\theta)}\\
&=\frac{\partial_{X}}{2}\left(\int |\phi_{\psi}|^{2}(1+\psi^{(1-2\mu)})d\psi (1+X)^{(1-2\theta)}\right)\\
&\quad -\frac{1-2\theta}{2}\int |\phi_{\psi}|^{2}(1+\psi^{(1-2\mu)})d\psi (X+1)^{-2\theta}\\
&\quad+(1-2\mu)\int \phi_{X}\phi_{\psi}\psi^{-2\mu}d\psi (1+X)^{(1-2\theta)}\\
\end{aligned}
\end{equation}
Applying Hardy's inequality, we get
\begin{equation}\label{E14}
\begin{aligned}
\int \phi_{X}\phi_{\psi}\psi^{-2\mu}d\psi (1+X)^{(1-2\theta)}&\leq \left\|\phi_{\psi}\psi^{(\frac{1}{2}-\mu)}\right\|_{L^{2}_{\psi}}\left\|\phi_{X}\psi^{-(\frac{1}{2}+\mu)}(X+1)^{(1-2\theta)}\right\|_{L^{2}_{\psi}}\\
&\lesssim_{\mu} \left\|\phi_{\psi}\psi^{(\frac{1}{2}-\mu)}\right\|_{L^{2}_{\psi}}\left\|\phi_{X\psi}\psi^{(\frac{1}{2}-\mu)}(X+1)^{(1-2\theta)}\right\|_{L^{2}_{\psi}}\\
&\leq \mathcal{C}_{2}(\iota,\mu)\left\|\phi_{\psi}\psi^{(\frac{1}{2}-\mu)}\right\|_{L^{2}_{\psi}}^{2}\\
&\quad +\iota \left\|\phi_{X\psi}\psi^{(\frac{1}{2}-\mu)}(X+1)^{(1-2\theta)}\right\|_{L^{2}_{\psi}}^{2},
\end{aligned}
\end{equation}
where $\iota\in(0,1)$ is a small number which will be chosen later. Combining \eqref{E11}, \eqref{E12}, \eqref{E13}, \eqref{E14}, we obtain the energy estimation for $ E_{1} $:
\begin{equation}\label{E15}
\|\phi\|_{E_{1}}^{2}\leq \|\phi_{0}\|_{E_{1,0}}^{2}+ \mathcal{C}_{2}(\iota,\mu)\|\phi\|_{E_{0}}^{2}+\iota \|\phi\|_{E_{2}}^{2}.
\end{equation}

\textit{Step 3.  Estimation of $ E_{2} $.}
We now proceed to estimate the higher-order energy $ E_{2} $. It is noted that $ \phi_{X} $ satisfies the following equation:
\begin{equation}\label{E21}
\partial_{X}\phi_{X} - u\partial_{\psi}^{2}\phi_{X} + A\phi_{X} - u_{X}\phi_{\psi\psi} + A_{X}\phi = 0.
\end{equation}
By testing \eqref{E21} with
\begin{equation*}
\phi_{X}\frac{1}{u}(1+\psi^{(1-2\mu)}) (X+1)^{(2-2\theta)},
\end{equation*}
we obtain the following energy balance:
\begin{equation}\label{E22}
\begin{aligned}
0 & = \int \Big(\partial_{X}\phi_{X} - u\partial_{\psi\psi}\phi_{X} + A\phi_{X} - u_{X}\phi_{\psi\psi} + A_{X}\phi\Big)\\
&\quad\quad\quad\times\Big(\phi_{X}\frac{1}{u}(1+\psi^{(1-2\mu)}(X+1)^{(2-2\theta)}\Big) d\psi \\\
&=:I_{1} + I_{2} + I_{3} + I_{4} + I_{5}.
\end{aligned}
\end{equation}
Where
\begin{equation}
\begin{cases}
I_{1}= \displaystyle\int \phi_{XX}\phi_{X}\frac{1}{u}(1+\psi^{(1-2\mu)})d\psi (X+1)^{(2-2\theta)}\\
I_{2}=- \displaystyle\int \phi_{X\psi\psi}\phi_{X}(1+\psi^{(1-2\mu)})d\psi (X+1)^{(2-2\theta)}\\
I_{3}=\displaystyle\int A|\phi_{X}|^{2}\frac{1}{u}(1+\psi^{(1-2\mu)})d\psi (X+1)^{(2-2\theta)}\\
I_{4}=\displaystyle- \int u_{X}\phi_{\psi\psi}\phi_{X}\frac{1}{u}(1+\psi^{(1-2\mu)})d\psi (X+1)^{(2-2\theta)}\\
I_{5}=\displaystyle\int A_{X}\phi\phi_{X}\frac{1}{u}(1+\psi^{(1-2\mu)})d\psi (X+1)^{(2-2\theta)}
\end{cases}
\end{equation}

\textit{Estimate of $I_{1}$.}
By direct computation:
\begin{equation}\label{E23}
\begin{aligned}
I_{1} &= \int \phi_{XX}\phi_{X}\frac{1}{u}(1+\psi^{(1-2\mu)})d\psi (X+1)^{(2-2\theta)}\\
&= \frac{\partial_{X}}{2}\left(\int |\phi_{X}|^{2}\frac{1}{u}(1+\psi^{(1-2\mu)})d\psi(X+1)^{(2-2\theta)}\right)\\
&\quad + \int |\phi_{X}|^{2}\frac{u_{X}}{2u^{2}}(1+\psi^{(1-2\mu)})d\psi(X+1)^{(2-2\theta)}\\
&\quad - (1-\theta)\int |\phi_{X}|^{2}\frac{1}{u}(1+\psi^{(1-2\mu)})d\psi(X+1)^{(1-2\theta)}.
\end{aligned}
\end{equation}

\textit{Estimate of $I_{2}$.}
For the second term $ I_{2} $, we perform integration by parts twice:
\begin{equation}\label{E24}
\begin{aligned}
I_{2} &= -\int \phi_{X\psi\psi}\phi_{X}(1+\psi^{(1-2\mu)})d\psi (X+1)^{(2-2\theta)}\\
&= \int |\phi_{X\psi}|^{2}(1+\psi^{(1-2\mu)})d\psi (X+1)^{(2-2\theta)} \\
&\quad + (1-2\mu)\int \phi_{X\psi}\phi_{X}\psi^{-2\mu}d\psi (X+1)^{(2-2\theta)}\\
&\geq \int |\phi_{X\psi}|^{2}(1+\psi^{(1-2\mu)})d\psi (X+1)^{(2-2\theta)}.
\end{aligned}
\end{equation}

\textit{Estimate of $I_{3}$.}
As demonstrated in the previous proof, the estimation of $ I_{3} $ should be combined with the second term of \eqref{E23}:
\begin{equation}\label{E25}
\begin{aligned}
I_{3} + \eqref{E23}_{2} &= \int |\phi_{X}|^{2}\left[\frac{A}{u} + \frac{u_{X}}{2u^{2}}\right](1 + \psi^{(1-2\mu)})d\psi (X+1)^{(2-2\theta)}\\
&\geq \int |\phi_{X}|^{2}\frac{1}{u^{3}}\left(\mathcal{R}_{1} + \mathcal{R}_{2} + \mathcal{R}_{3}\right)(1 + \psi^{(1-2\mu)})d\psi (X+1)^{(2-2\theta)}.
\end{aligned}
\end{equation}

As explained in the first step, we focus solely on the most critical term:
\begin{equation}\label{E26}
\int |\phi_{X}|^{3}\frac{1}{u^{3}}(1 + \psi^{(1-2\mu)})d\psi (X+1)^{(2-2\theta)}.
\end{equation}

In the interval $ \psi \geq \sqrt{X+1} $:
\begin{equation}\label{E27}
\begin{aligned}
&\int_{\psi \geq \sqrt{X+1}} |\phi_{X}|^{3}\frac{1}{u^{3}}(1 + \psi^{(1-2\mu)})d\psi (X+1)^{(2-2\theta)} \\
&\quad\quad\lesssim \int |\phi_{X}|^{3}(1 + \psi^{(1-2\mu)})d\psi (X+1)^{(2-2\theta)}\\
&\quad\quad\leq \|\phi_{X}\|_{L^{\infty}_{\psi}}\|\phi\|_{E_{2}}^{2} \lesssim \|\phi_{X}\|_{L^{2}_{\psi}}^{\frac{1}{2}}\|\phi_{X\psi}\|_{L^{2}_{\psi}}^{\frac{1}{2}}\|\phi\|_{E_{2}}^{2}\\
&\quad\quad\lesssim \|\phi\|_{E_{2}}^{\frac{5}{2}}\|\phi_{X\psi}(X+1)^{(1-\theta)}\|_{L^{2}_{\psi}}^{\frac{1}{2}}(X+1)^{\theta-1}.
\end{aligned}
\end{equation}
In the interval $ \psi\leq \sqrt{X+1} $, by applying the Cauchy-Schwarz inequality and Hardy's inequality, we can derive the following inequality:
\begin{equation}\label{E28}
\begin{aligned}
&\int_{\psi\leq \sqrt{X+1}}|\phi_{X}|^{3}\frac{1}{u^{3}}(1+\psi^{(1-2\mu)})d\psi (X+1)^{(2-2\theta)}\\
&\quad \lesssim \int_{0}^{\sqrt{X+1}}|\phi_{X}|^{3}\left\{\frac{1}{\psi^{\frac{3}{2}}}+\frac{1}{\psi^{\frac{3}{5}}}\right\}d\psi (X+1)^{(\frac{14}{5}-2\theta)}\\
&\quad \leq (X+1)^{(\frac{14}{5}-2\theta)}\left(\int_{0}^{\sqrt{X+1}}|\phi_{X}|^{2}\left\{\frac{1}{\psi^{2}}+\frac{1}{\psi^{\frac{11}{10}}}\right\}d\psi\right)^{\frac{1}{2}}\\
&\quad\quad\quad\quad\times \left(\int_{0}^{\sqrt{X+1}}|\phi_{X}|^{4}\left\{\frac{1}{\psi}+\frac{1}{\psi^{\frac{1}{10}}}\right\}d\psi\right)^{\frac{1}{2}}\\
&\quad \lesssim (X+1)^{(\frac{14}{5}-2\theta)}\|\phi_{X}(1+\psi^{\frac{1}{2}-\mu})\|_{L^{2}_{\psi}}\|\phi_{X\psi}(1+\psi^{\frac{1}{2}-\mu})\|_{L^{2}_{\psi}}^{2}.
\end{aligned}
\end{equation}
The last inequality is due to
\begin{equation*}
\begin{aligned}
\int |\phi_{X}|^{4}\left\{\frac{1}{\psi}+\frac{1}{\psi^{\frac{1}{10}}}\right\}&\leq \left\|\frac{\phi_{X}}{\psi^{\frac{1}{2}}}\right\|_{L^{\infty}_{\psi}}^{2}\int |\phi_{X}|^{2}\left(1+\psi^{\frac{9}{10}}\right)\\
&\leq \|\phi_{X\psi}\|_{L^{2}_{\psi}}^{2}\|\phi_{X}(1+\psi^{\frac{1}{2}-\mu})\|_{L^{2}_{\psi}}^{2}.
\end{aligned}
\end{equation*}

\textit{Estimate of $I_{4}$.}
Next, we examine the fourth term $ I_{4} $:
\begin{equation}\label{E29}
\begin{aligned}
I_{4} &= -\int \bar{u}_{X}\phi_{\psi\psi}\phi_{X}\frac{1}{u}(1+\psi^{(1-2\mu)})d\psi (X+1)^{(2-2\theta)}\\
&\quad -\int \rho_{X}\phi_{\psi\psi}\phi_{X}\frac{1}{u}(1+\psi^{(1-2\mu)})d\psi (X+1)^{(2-2\theta)}\\
&= I_{41} + I_{42}.
\end{aligned}
\end{equation}
Given that $ \bar{u}_{X} \leq 0 $, we can express $I_{41}$ as:
\begin{equation}\label{E210}
\begin{aligned}
I_{41} &= -\int \frac{\bar{u}_{X}}{u}\left(\phi_{X}+A\phi\right)\phi_{X}\frac{1}{u}(1+\psi^{(1-2\mu)})d\psi (X+1)^{(2-2\theta)}\\
&\geq -\int \frac{\bar{u}_{X}}{u}\frac{A\phi\phi_{X}}{u}(1+\psi^{(1-2\mu)})d\psi (X+1)^{(2-2\theta)}.
\end{aligned}
\end{equation}
By applying the Cauchy-Schwarz inequality and the properties of the Blasius profile, we obtain:
\begin{equation}\label{E211}
\begin{aligned}
&\left|\int \frac{\bar{u}_{X}}{u}\frac{A\phi\phi_{X}}{u}(1+\psi^{(1-2\mu)})d\psi (X+1)^{(2-2\theta)}\right|\\
&\lesssim (X+1)^{-2\theta}\int |\phi||\phi_{X}|\frac{1}{u}(1+\psi^{(1-2\mu)})d\psi\\
&\leq (X+1)^{-(1+\theta)}\left(\int |\phi|^{2}\frac{1}{u}(1+\psi^{(1-2\mu)})d\psi\right)^{\frac{1}{2}}\\
&\quad\quad\quad\times\left(\int |\phi_{X}|^{2}\frac{1}{u}(1+\psi^{(1-2\mu)})d\psi(X+1)^{(2-2\theta)}\right)^{\frac{1}{2}}\\
&\leq (X+1)^{-(1+\theta)}\|\phi\|_{E_{0}}\|\phi\|_{E_{2}}.
\end{aligned}
\end{equation}
The estimation of $I_{42}$ follows a similar approach to that of \eqref{E26}; therefore, we omit it.

\textit{Estimate of $I_{5}$.}
Finally, we proceed to estimate the last term $I_{5}$. We express this term as follows:
\begin{equation}\label{E212}
\begin{aligned}
|I_{5}| &= \left|\int A_{X}\phi\phi_{X}\frac{1}{u}(1+\psi^{(1-2\mu)})d\psi(X+1)^{(2-2\theta)}\right|\\
&\leq \left\| u^{\frac{3}{2}}A_{X}(1+\psi^{(\frac{1}{2}-\mu)})(X+1)^{\frac{7}{4}}\right\|_{L^{2}_{\psi}} \left\| \phi\phi_{X}\frac{1}{u^{\frac{5}{2}}}(1+\psi^{(\frac{1}{2}-\mu)})(X+1)^{\frac{1}{4}-2\theta}\right\|_{L^{2}_{\psi}}\\
&\lesssim \left(1+\|\phi\|_{E_{0}}+\|\phi\|_{E_{1}}+\|\phi\|_{E_{2}}\right) \left\| \phi\phi_{X}\frac{1}{u^{\frac{5}{2}}}(1+\psi^{(\frac{1}{2}-\mu)})(X+1)^{\frac{1}{4}-2\theta}\right\|_{L^{2}_{\psi\geq\sqrt{X+1}}}\\
&\quad\quad + \left(1+\|\phi\|_{E_{0}}+\|\phi\|_{E_{1}}+\|\phi\|_{E_{2}}\right) \left\| \phi\phi_{X}\frac{1}{u^{\frac{5}{2}}}(1+\psi^{(\frac{1}{2}-\mu)})(X+1)^{\frac{1}{4}-2\theta}\right\|_{L^{2}_{\psi\leq\sqrt{X+1}}}\\
&= I_{51} + I_{52}.
\end{aligned}
\end{equation}
Where we have utilized the weighted embedding inequality proved in  \cite{MR4097332}:
\begin{equation}\label{techweight}
\left\| u^{\frac{3}{2}}A_{X}(1+\psi^{(\frac{1}{2}-\mu)})(X+1)^{\frac{7}{4}}\right\|_{L^{2}_{\psi}}\lesssim \left(1+\|\phi\|_{E_{0}}+\|\phi\|_{E_{1}}+\|\phi\|_{E_{2}}\right).
\end{equation}
Given that $I_{52}$ is more intricate than $I_{51}$ due to the involvement of a singular weight, our focus shifts towards analyzing $I_{52}$. By applying Hardy's inequality and the interpolation inequality, we arrive at the following expression:

\begin{equation}\label{E213}
\begin{aligned}
&\left\| \phi\phi_{X}\frac{1}{u^{\frac{5}{2}}}(1+\psi^{(\frac{1}{2}-\mu)})(X+1)^{\frac{1}{4}-2\theta}\right\|_{L^{2}_{\psi\leq\sqrt{X+1}}}\\
&\quad\lesssim \left\| \phi\phi_{X}\frac{1}{\psi^{\frac{5}{4}}}(X+1)^{\frac{7}{8}-2\theta}\right\|_{L^{2}_{\psi\leq\sqrt{X+1}}}+\left\| \phi\phi_{X}\frac{1}{\psi^{\frac{3}{4}}}(X+1)^{\frac{7}{8}-2\theta}\right\|_{L^{2}_{\psi\leq\sqrt{X+1}}}\\
&\quad\leq \left(\|\phi\|_{\dot{C}^{\frac{1}{4}}_{\psi}}\left\|\frac{\phi_{X}}{\psi}\right\|_{L^{2}_{\psi}}+\|\phi\|_{\dot{C}^{\frac{3}{16}}}\left\|\frac{\phi_{X}}{\psi^{\frac{9}{16}}}\right\|_{L^{2}_{\psi}}\right)(X+1)^{\frac{7}{8}-2\theta}\\
&\quad\lesssim\|\phi(1+\psi^{\frac{1}{2}-\mu})\|_{L^{2}_{\psi}}^{\frac{1}{4}}\|\phi_{\psi}\|_{L^{2}_{\psi}}^{\frac{3}{4}}\|\phi_{X\psi}(1+\psi^{\frac{1}{2}-\mu})\|_{L^{2}_{\psi}}(X+1)^{\frac{7}{8}-2\theta}\\
&\quad \leq \|\phi\|_{E_{0}}^{\frac{1}{4}}\|\phi\|_{E_{1}}^{\frac{3}{4}}\|\phi_{X\psi}(1+\psi^{\frac{1}{2}-\mu})\|_{L^{2}_{\psi}}(X+1)^{(\frac{1}{2}-\frac{5}{4}\theta)}.
\end{aligned}
\end{equation}
Combining equations $\eqref{E21}-\eqref{E213}$ and applying Hölder's inequality in time, we deduce the following energy estimate for $E_{2}$:
\begin{equation}\label{E214}
\|\phi\|_{E_{2}}^{2} \leq \|\phi\|_{E_{2,0}}^{2} + \mathcal{C}_{3}\left(\|\phi\|_{E_{1}}^{2} + \|\phi\|_{E_{2}}^{3} + \|\phi\|_{E_{0}}\|\phi\|_{E_{2}} + \|\phi\|_{E_{0}}^{\frac{1}{4}}\|\phi\|_{E_{1}}^{\frac{3}{4}}\|\phi\|_{E_{2}}\right).
\end{equation}

\textit{Step 4: Bootstrap Arguments.}
The bootstrap hypothesis is
\begin{equation}\label{boot}
\begin{cases}
\|\phi\|_{E_{0}}\leq (F_{0}+1)\|\phi_{0}\|_{E_{\text{initial}}},\\
\|\phi\|_{E_{1}}\leq (F_{1}+1)\|\phi_{0}\|_{E_{\text{initial}}},\\
\|\phi\|_{E_{2}}\leq (F_{2}+1)\|\phi_{0}\|_{E_{\text{initial}}},
\end{cases}
\end{equation}
where $F_{i}=F_{i}(\theta,\mu,\iota)$.

Under bootstrap hypothesis \eqref{boot}, by suitably choosing the small parameter $\iota$ relative to universal constants, combining \eqref{E07}, \eqref{E15}, and \eqref{E214}, and applying Lemma \ref{initial} we obtain
\begin{equation}\label{boot1}
\begin{cases}
\|\phi\|_{E_{0}}\leq F_{0}\|\phi_{0}\|_{E_{\text{initial}}},\\
\|\phi\|_{E_{1}}\leq F_{1}\|\phi_{0}\|_{E_{\text{initial}}},\\
\|\phi\|_{E_{2}}\leq F_{2}\|\phi_{0}\|_{E_{\text{initial}}},
\end{cases}
\end{equation}
which is stronger than \eqref{boot}. Hence, along with a standard bootstrap argument, we establish the global existence of $\phi$ in the space $E_{\leq 2}$, with the bound in the space $E_{\leq 2}$:
\begin{equation}
\|\phi\|_{E_{\leq 2}}\lesssim \|\phi_{0}\|_{E_{\text{initial}}}.
\end{equation}
Hence, we finish the proof of Lemma \ref{lm61}.
\end{proof}
Now we are proceeding to derive the decay rates stated in Theorem \ref{iyer}. Notice that the global existence of $\phi$ in the function space $E_{\leq 2}$ alone does not suffice to establish the decay rates prescribed in Theorem \ref{iyer}. Before doing it, we state a quantitative estimate of $\|\phi_{\psi}\|_{L^{\infty}_{\psi}}$, which is finer than \eqref{E066}.
\begin{comment}
\begin{proof}[Step 4: Decay Rates]
The global existence of $\phi$ in the function space $E_{\leq 2}$ alone does not suffice to establish the decay rates prescribed in Theorem \eqref{iyer}. We now elucidate the methodology for deriving these essential decay rates.

It is crucial to discern that the quantity $\|F \psi^{\frac{1}{2}}\|_{L^{2}_{\psi}}$ captures lower-frequency components relative to $\|F\|_{L^{2}_{\psi}}$. This distinction enables the establishment of the following interpolation inequality:
\begin{equation}\label{E215}
\|F\|_{L^{2}_{\psi}}^{2} \lesssim \|F\psi^{\frac{1}{2}}\|_{L^{2}_{\psi}}^{\frac{4}{3}} \|F_{\psi}\|_{L^{2}_{\psi}}^{\frac{2}{3}}.
\end{equation}
By employing \eqref{E215} within the energy balance identity, we can deduce the decay rates specified in Theorem \ref{iyer}. The methodology closely mirrors that of Section \ref{main}; therefore, we shall abstain from delving into intricate details at this point.
\end{proof}
\end{comment}
\begin{lemma}\label{lm62}Assume $\phi$ is the global solution of \eqref{eqphi} in space $E_{\leq 2}$, then
\begin{equation}\label{decay1}
\begin{aligned}
\left\|\phi_{\psi}\right\|_{L^{\infty}_{\psi\leq\sqrt{X+1}}}^{2}&\lesssim (X+1)^{\frac{1}{2}}\left\|\frac{\phi_{X}}{\sqrt{u}}\right\|_{L^{2}_{\psi}}^{2}+\left\|\frac{\phi_{X}}{\sqrt{u}}\right\|_{L^{2}_{\psi}}\left\|\phi_{\psi}\right\|_{L^{2}_{\psi}}\\
&\quad +(X+1)^{-\frac{1}{2}}\left[\|\phi_{\psi}\|_{L^{2}_{\psi}}^{2}+\left\|\frac{\phi_{X}}{\sqrt{u}}\right\|_{L^{2}_{\psi}}\left\|\frac{\phi}{\sqrt{u}}\right\|_{L^{2}_{\psi}}\right]\\
&\quad +(X+1)^{-1}\left\|\frac{\phi}{\sqrt{u}}\right\|_{L^{2}_{\psi}}\|\phi_{\psi}\|_{L^{2}_{\psi}}+(X+1)^{-\frac{3}{2}}\left\|\frac{\phi}{\sqrt{u}}\right\|_{L^{2}_{\psi}}^{2}.
\end{aligned}
\end{equation}
\end{lemma}
\begin{proof}Assume $ \chi(\psi):\mathbb{R}^{+}\to [0,1] $ is the smooth cut-off function with
\begin{equation*}
\chi(\psi)=
\begin{cases}
1,\quad \psi\leq 1;\\
0,\quad \psi\geq 2.
\end{cases}
\end{equation*}
We also define
\begin{equation*}
\chi_{X}(\psi)=\chi\left(\frac{\psi}{\sqrt{X+1}}\right).
\end{equation*}
We then have
\begin{equation*}
\begin{aligned}
\phi_{\psi}^{2}(\psi)\chi_{X}^{2}(\psi)&=2\int_{\infty}^{\psi}\phi_{\psi}\phi_{\psi\psi}\chi_{X}^{2}d\psi+ \int_{\infty}^{\psi}\phi_{\psi}^{2}\p_{\psi}\left(\chi_{X}^{2}\right)d\psi=N_{1}+N_{2}.
\end{aligned}
\end{equation*}
By the definition of $\chi_{X}(\psi)$,
\begin{equation*}
|N_{2}|\lesssim (X+1)^{-\frac{1}{2}}\|\phi_{\psi}\|_{L^{2}_{\psi}}^{2}.
\end{equation*}
For the first term, applying Cauchy-Schwarz inequality and Lemma \ref{Lm3},
\begin{equation*}
\begin{aligned}
|N_{1}|&\leq 2\left\|\phi_{\psi}\psi^{-\frac{1}{4}}\chi_{X}\right\|_{L^{2}_{\psi}}\left\|\phi_{\psi\psi}\psi^{\frac{1}{4}}\chi_{X}\right\|_{L^{2}_{\psi}}\\
&\lesssim (X+1)^{\frac{1}{8}}\left\|\phi_{\psi}\psi^{-\frac{1}{4}}\chi_{X}\right\|_{L^{2}_{\psi}}\left\|\phi_{\psi\psi}\sqrt{u}\right\|_{L^{2}_{\psi}}\\
&\leq (X+1)^{\frac{1}{8}}\left\|\phi_{\psi}\psi^{-\frac{1}{4}}\chi_{X}\right\|_{L^{2}_{\psi}}\left(\left\|\frac{\phi_{X}}{\sqrt{u}}\right\|_{L^{2}_{\psi}}+\left\|\frac{A\phi}{\sqrt{u}}\right\|_{L^{2}_{\psi}}\right)\\
&\lesssim (X+1)^{\frac{1}{8}}\left\|\phi_{\psi}\psi^{-\frac{1}{4}}\chi_{X}\right\|_{L^{2}_{\psi}}\left\|\frac{\phi_{X}}{\sqrt{u}}\right\|_{L^{2}_{\psi}}+(X+1)^{-\frac{7}{8}}\left\|\phi_{\psi}\psi^{-\frac{1}{4}}\chi_{X}\right\|_{L^{2}_{\psi}}\left\|\frac{\phi}{\sqrt{u}}\right\|_{L^{2}_{\psi}}.
\end{aligned}
\end{equation*}
Applying Hardy's inequality (Lemma \ref{sobolev}), we have
\begin{equation}\label{decay2}
\left\|\phi_{\psi}\psi^{-\frac{1}{4}}\chi_{X}\right\|_{L^{2}_{\psi}}\lesssim \left\|\phi_{\psi\psi}\psi^{\frac{3}{4}}\chi_{X}\right\|_{L^{2}_{\psi}}+\left\|\phi_{\psi}\psi^{\frac{3}{4}}\p_{\psi}\chi_{X}\right\|_{L^{2}_{\psi}}.
\end{equation}
Clearly,
\begin{equation*}
\eqref{decay2}_{2}\lesssim (X+1)^{-\frac{1}{2}}\left\|\phi_{\psi}\psi^{\frac{3}{4}}\right\|_{L^{2}_{\psi\leq\sqrt{X+1}}}\lesssim (X+1)^{-\frac{1}{8}}\|\phi_{\psi}\|_{L^{2}_{\psi}}.
\end{equation*}
By the equation \eqref{eqphi},
\begin{equation*}
\eqref{decay2}_{1}\lesssim (X+1)^{\frac{3}{8}}\left\|\phi_{X}+A\phi \right\|_{L^{2}_{\psi}}\lesssim (X+1)^{\frac{3}{8}}\|\phi_{X}\|_{L^{2}_{\psi}}+(X+1)^{-\frac{5}{8}}\|\phi\|_{L^{2}_{\psi}}.
\end{equation*}
Therefore, we finish the proof of Lemma \ref{lm62}.
\end{proof}
\begin{lemma}\label{lm63}
Under the assumptions in Theorem \ref{iyer}, the global solution $\phi$ in space $E_{\leq 2}$ satisfies the decay rates \eqref{iyerdecay}.
\end{lemma}
\begin{proof}
\textit{Step 1.} First we derive the decay rates of $\|\phi\|_{L^{2}_{\psi}}$ and $\|\phi_{\psi}\|_{L^{2}_{\psi}}$. We are going to prove that
\begin{equation*}
\left\|\frac{\phi}{\sqrt{u}}\right\|_{L^{2}_{\psi}}^{2}\lesssim (X+1)^{-\left(\frac{1}{2}-\mu\right)},\quad \|\phi_{\psi}\|_{L^{2}_{\psi}}^{2}\lesssim (X+1)^{-\left(\frac{3}{2}-\mu\right)}.
\end{equation*}

Testing the equation \eqref{eqphi} by $\phi\frac{1}{u}$, we have:
\begin{equation}\label{decay3}
\frac{\p_{x}}{2}\left(\int_{0}^{\infty} \phi^{2}\frac{1}{u}d\psi\right)+\int_{0}^{\infty}|\phi_{\psi}|^{2}d\psi\leq \sum_{j=1}^{3}\int_{0}^{\infty}\phi^{2}\frac{1}{u^{3}}|\mathcal{R}_{j}|d\psi.
\end{equation}
Recall the definition of $\mathcal{R}_{j}$ from \eqref{R},
\begin{equation}\label{decay4}
\sum_{j=1}^{3}|\mathcal{R}_{j}|\lesssim (X+1)^{-1}|\phi|+|\phi_{X}|.
\end{equation}
In the interval $\psi\geq\sqrt{X+1}$, applying Lemma \ref{lm61}, we have
\begin{equation*}
\begin{aligned}
\sum_{j=1}^{3}\int_{\psi\geq \sqrt{X+1}}\phi^{2}\frac{1}{u^{3}}|\mathcal{R}_{j}|d\psi&\lesssim (X+1)^{-1}\int |\phi|^{3}+\int |\phi|^{2}|\phi_{X}|\\
&\lesssim (X+1)^{-1}\|\phi\|_{L^{2}_{\psi}}^{2}\|\phi\|_{L^{\infty}_{\psi}}+\|\phi_{X}\|_{L^{2}_{\psi}}\|\phi\|_{L^{2}_{\psi}}\|\phi\|_{L^{\infty}_{\psi}}\\
&\lesssim (X+1)^{-\left(\frac{5}{4}-\frac{3\theta}{2}\right)}.
\end{aligned}
\end{equation*}
In the interval $\psi\leq\sqrt{X+1}$, applying Lemma \ref{lm61} and Sobolev embedding $\dot{H}^{1}_{\psi}\to \dot{C}^{\frac{1}{2}}_{\psi}$, we have 
\begin{equation*}
\begin{aligned}
(X+1)^{-1}\int_{\psi\leq\sqrt{X+1}}|\phi|^{3}\frac{1}{u^{3}}d\psi&\lesssim (X+1)^{-\frac{1}{4}}\int_{\psi\leq\sqrt{X+1}}\frac{|\phi|^{3}}{\psi^{\frac{3}{2}}}d\psi\\
&\lesssim (X+1)^{\frac{1}{4}}\|\phi_{\psi}\|_{L^{2}_{\psi}}^{3}\\
&\lesssim (X+1)^{-\left(\frac{5}{4}-3\theta\right)},
\end{aligned}
\end{equation*}
and
\begin{equation*}
\begin{aligned}
\int_{\psi\leq\sqrt{X+1}}|\phi|^{2}|\phi_{X}|\frac{1}{u^{3}}d\psi&\lesssim \left\|\frac{\phi_{X}}{\sqrt{u}}\right\|_{L^{2}_{\psi}}\left(\int_{\psi\leq\sqrt{X+1}}|\phi|^{4}\frac{1}{u^{5}}d\psi\right)^{\frac{1}{2}}\\
&\lesssim (X+1)^{-(1-\theta)}(X+1)^{\frac{5}{8}}\left(\int_{\psi\leq\sqrt{X+1}}|\phi|^{4}\frac{1}{\psi^{\frac{5}{2}}}d\psi\right)^{\frac{1}{2}}\\
&\lesssim (X+1)^{-\left(\frac{1}{4}-\theta\right)}\|\phi_{\psi}\|_{L^{2}_{\psi}}^{2}\\
&\lesssim (X+1)^{-\left(\frac{5}{4}-3\theta\right)}.
\end{aligned}
\end{equation*}
Similar to \eqref{onee} and \eqref{one}, we can derive the following interpolation inequality:
\begin{equation}\label{decay5}
\begin{aligned}
\int_{0}^{\infty}\frac{|\phi|^{2}}{\psi^{\frac{1}{2}}}d\psi&\lesssim \left(\int_{0}^{\infty}|\phi_{\psi}|^{2}d\psi\right)^{\frac{3-4\mu}{6-4\mu}}\left(\int_{0}^{\infty}|\phi|^{2}\psi^{1-2\mu}d\psi\right)^{\frac{3}{6-4\mu}};\\
\int_{0}^{\infty}|\phi|^{2}d\psi&\lesssim \left(\int_{0}^{\infty}|\phi_{\psi}|^{2}d\psi\right)^{\frac{1-2\mu}{3-2\mu}}\left(\int_{0}^{\infty}|\phi|^{2}\psi^{1-2\mu}d\psi\right)^{\frac{2}{3-2\mu}}.
\end{aligned}
\end{equation}
Applying Lemma \ref{lm61},we obtain
\begin{equation}\label{decay6}
\mathcal{A}(X):=\int_{0}^{\infty}|\phi|^{2}\frac{1}{u}d\psi\lesssim \max\left\{(X+1)^{\frac{1}{4}}\|\phi_{\psi}\|_{L^{2}_{\psi}}^{\frac{3-4\mu}{3-2\mu}},\ \|\phi_{\psi}\|_{L^{2}_{\psi}}^{\frac{2-4\mu}{3-2\mu}}\right\}.
\end{equation}
Combining with \eqref{decay3}, we obtain the following differential inequality,
\begin{equation}\label{decay7}
\mathcal{A}'(X)+K_{0}\min\left\{(X+1)^{-\frac{3-2\mu}{6-8\mu}}\mathcal{A}(X)^{\frac{6-4\mu}{3-4\mu}},\ \mathcal{A}(X)^{\frac{3-2\mu}{1-2\mu}}\right\}\leq K_{1}(X+1)^{-\left(\frac{5}{4}-3\theta\right)},
\end{equation}
for some $K_{0},K_{1}>0$.

Similar to the proof of Lemma \ref{L2decay}, \eqref{decay7} implies that
\begin{equation}\label{decay8}
\mathcal{A}(X)\lesssim (X+1)^{-\frac{5}{16}}.
\end{equation}
Proceeding in the same way as Lemma \ref{Middle} and Lemma \ref{H1decay}, we obtain
\begin{equation}\label{decay9}
\int_{0}^{\infty}|\phi_{\psi}(X,\psi)|^{2}d\psi\lesssim (X+1)^{-\frac{21}{16}}.
\end{equation}

Although the decay rates \eqref{decay8} and \eqref{decay9} do not attain our expectations, they improve the decay rate of the right-hand side of the energy inequality \eqref{decay3}. Indeed, applying \eqref{decay8}, \eqref{decay9}, and Lemma \ref{lm61}, we have
\begin{equation*}
\begin{aligned}
\int_{\psi\geq\sqrt{X+1}}|\phi|^{3}\frac{1}{u^{3}}&\lesssim \|\phi\|_{L^{2}_{\psi}}^{2}\|\phi\|_{L^{\infty}_{\psi}}\lesssim (X+1)^{-\frac{23}{32}};\\
\int_{\psi\leq\sqrt{X+1}}|\phi|^{3}\frac{1}{u^{3}}&\lesssim (X+1)^{\frac{5}{4}}\|\phi_{\psi}\|_{L^{2}_{\psi}}^{3}\lesssim (X+1)^{-\frac{23}{32}};\\
\int_{\psi\geq\sqrt{X+1}}|\phi|^{2}|\phi_{X}|\frac{1}{u^{3}}&\lesssim \|\phi_{X}\|_{L^{2}_{\psi}}\|\phi\|_{L^{2}_{\psi}}\|\phi\|_{L^{\infty}_{\psi}}\lesssim (X+1)^{-\left(\frac{25}{16}-\theta\right)};\\
\int_{\psi\leq\sqrt{X+1}}|\phi|^{2}|\phi_{X}|\frac{1}{u^{3}}&\lesssim (X+1)^{-\left(\frac{1}{4}-\theta\right)}\|\phi_{\psi}\|_{L^{2}_{\psi}}^{2}\lesssim (X+1)^{-\left(\frac{25}{16}-\theta\right)}.
\end{aligned}
\end{equation*}
Recall \eqref{decay3}, we obtain the following differential inequality,
\begin{equation}\label{decay10}
\mathcal{A}'(X)+K_{2}\min\left\{(X+1)^{-\frac{3-2\mu}{6-8\mu}}\mathcal{A}(X)^{\frac{6-4\mu}{3-4\mu}},\ \mathcal{A}(X)^{\frac{3-2\mu}{1-2\mu}}\right\}\leq K_{3}(X+1)^{-\left(\frac{25}{16}-\theta\right)}.
\end{equation}
Similar to the previous discussions, \eqref{decay10} implies that
\begin{equation}\label{decay11}
\mathcal{A}(X):=\int_{0}^{\infty}|\phi|^{2}\frac{1}{u}d\psi\lesssim (X+1)^{-\left(\frac{1}{2}-\mu\right)}.
\end{equation}
Therefore, 
\begin{equation}\label{decay111}
\|\phi_{\psi}\|_{L^{2}_{\psi}}^{2}\lesssim (X+1)^{-\left(\frac{3}{2}-\mu\right)}.
\end{equation}

\textit{Step 2.} Next, we derive the decay rates of $\|\phi_{X}\|_{L^{2}_{\psi}}$ and $\|\phi_{X\psi}\|_{L^{2}_{\psi}}$. We will show that:
\begin{equation*}
\left\|\frac{\phi_{X}}{\sqrt{u}}\right\|_{L^{2}_{\psi}}^{2}\lesssim (X+1)^{-\left(\frac{5}{2}-2\theta-\mu\right)},\quad \left\|\phi_{X\psi}(X+1)^{\frac{5}{4}-(\theta+\mu)}\right\|_{L^{2}_{X,\psi}}\lesssim 1.
\end{equation*}
Recall that $\phi_{X}$ satisfies the following equation:
\begin{equation}\label{decay12}
\p_{X}\phi_{X}-u\p_{\psi}^{2}\phi_{X}+A\phi_{X}-u_{X}\phi_{\psi\psi}+A_{X}\phi=0.
\end{equation}
We also define
\begin{equation*}
\mathcal{Q}(X):=\int_{0}^{\infty}|\phi_{X}|^{2}\frac{1}{u}d\psi(X+1)^{2-2\theta}.
\end{equation*}
Testing \eqref{decay12} by 
\begin{equation*}
\phi_{X}\frac{1}{u}(X+1)^{2-2\theta},
\end{equation*}
we obtain the following energy identity:
\begin{equation}\label{decay13}
\begin{aligned}
0 & = \int \Big(\partial_{X}\phi_{X} - u\partial_{\psi\psi}\phi_{X} + A\phi_{X} - u_{X}\phi_{\psi\psi} + A_{X}\phi\Big)\phi_{X}\frac{1}{u}d\psi (X+1)^{(2-2\theta)}\\
& = \int \phi_{XX}\phi_{X}\frac{1}{u}d\psi (X+1)^{(2-2\theta)} - \int \phi_{X\psi\psi}\phi_{X}d\psi (X+1)^{(2-2\theta)}\\
& \quad + \int A|\phi_{X}|^{2}\frac{1}{u}d\psi (X+1)^{(2-2\theta)} - \int u_{X}\phi_{\psi\psi}\phi_{X}\frac{1}{u}d\psi (X+1)^{(2-2\theta)}\\
& \quad + \int A_{X}\phi\phi_{X}\frac{1}{u}d\psi (X+1)^{(2-2\theta)} =: I_{1} + I_{2} + I_{3} + I_{4} + I_{5}.
\end{aligned}
\end{equation}

\textit{Estimate of $I_{1}$.}
For the term $I_{1}$, we have
\begin{equation}\label{decay14}
\begin{aligned}
I_{1}&=\frac{\p_{X}}{2}\left(\int_{0}^{\infty}|\phi_{X}|^{2}\frac{1}{u}d\psi (X+1)^{2-2\theta}\right)\\
&+\int_{0}^{\infty}|\phi_{X}|^{2}\frac{u_{X}}{2u^{2}}d\psi (X+1)^{2-2\theta}-(1-\theta)\int_{0}^{\infty}|\phi_{X}|^{2}\frac{1}{u}d\psi(X+1)^{1-2\theta}.
\end{aligned}
\end{equation}

\textit{Estimate of $I_{2}$.}
Integrating by parts, we have
\begin{equation}\label{decay15}
I_{2}=\int_{0}^{\infty}|\phi_{X\psi}|^{2}d\psi(X+1)^{2-2\theta}.
\end{equation}

\textit{Estimate of $I_{3}$.}
Similar to the proof of Lemma \ref{lm61}, the term $I_{3}$ should be estimated together with the second term of \eqref{decay14},
\begin{equation}\label{decay16}
\begin{aligned}
&|I_{3}+\eqref{decay14}_{2}|\\
&\leq \sum_{j=1}^{3}\int_{0}^{\infty} |\phi_{X}|^{2}\frac{1}{u^{3}}|\mathcal{R}_{j}|d\psi (X+1)^{2-2\theta}\\
&\lesssim_{\eqref{decay4}} \int_{0}^{\infty}|\phi_{X}|^{2}|\phi|\frac{1}{u^{3}}d\psi(X+1)^{1-2\theta}+\int_{0}^{\infty}|\phi_{X}|^{3}\frac{1}{u^{3}}d\psi(X+1)^{2-2\theta}\\
&= \int_{\psi\geq\sqrt{X+1}}|\phi_{X}|^{2}|\phi|\frac{1}{u^{3}}d\psi(X+1)^{1-2\theta}+ \int_{\psi\leq\sqrt{X+1}}|\phi_{X}|^{2}|\phi|\frac{1}{u^{3}}d\psi(X+1)^{1-2\theta}\\
&\quad +\int_{\psi\geq\sqrt{X+1}}|\phi_{X}|^{3}\frac{1}{u^{3}}d\psi(X+1)^{2-2\theta}+\int_{\psi\leq\sqrt{X+1}}|\phi_{X}|^{3}\frac{1}{u^{3}}d\psi(X+1)^{2-2\theta}\\
&=:M_{1}+M_{2}+M_{3}+M_{4}.
\end{aligned}
\end{equation}
By \eqref{decay11} and \eqref{decay111}, we have
\begin{equation*}
M_{1}\leq (X+1)^{1-2\theta}\|\phi\|_{L^{\infty}_{\psi}}\|\phi_{X}\|_{L^{2}_{\psi}}^{2}\lesssim (X+1)^{-\left(\frac{3}{2}-\frac{\mu}{2}\right)}\mathcal{Q}(X).
\end{equation*}
For the term $M_{2}$, applying Lemma \ref{lm62}, we obtain
\begin{equation*}
\begin{aligned}
M_{2}&\leq (X+1)^{-1}\left\|\frac{\phi}{u^{2}}\right\|_{L^{\infty}_{\psi\leq\sqrt{X+1}}}\mathcal{Q}(X)\\
&\lesssim (X+1)^{-\frac{1}{2}}\|\phi_{\psi}\|_{L^{\infty}_{\psi\leq\sqrt{X+1}}}\mathcal{Q}(X)\\
&\lesssim (X+1)^{-\left(\frac{5}{4}-\theta\right)}\mathcal{Q}(X).
\end{aligned}
\end{equation*}
To estimate $M_{3}$, by Young's inequality, we have
\begin{equation*}
\begin{aligned}
M_{3}&\lesssim (X+1)^{2-2\theta}\|\phi_{X}\|_{L^{2}_{\psi}}^{\frac{5}{2}}\|\phi_{X\psi}\|_{L^{2}_{\psi}}^{\frac{1}{2}}\\
&\leq \frac{1}{2}(X+1)^{2-2\theta}\|\phi_{X\psi}\|_{L^{2}_{\psi}}^{2}+C_{1} (X+1)^{2-2\theta}\|\phi_{X}\|_{L^{2}_{\psi}}^{\frac{10}{3}}\\
&\leq \frac{1}{2}(X+1)^{2-2\theta}\|\phi_{X\psi}\|_{L^{2}_{\psi}}^{2}+C_{2}(X+1)^{-\frac{4}{3}(1-\theta)}\mathcal{Q}(X).
\end{aligned}
\end{equation*}
for some $C_{1},C_{2}>0$.
For the term $M_{4}$, by Cauchy-Schwarz inequality and Lemma \ref{lm61},
\begin{equation*}
\begin{aligned}
M_{4}&\lesssim (X+1)^{\frac{11}{4}-2\theta}\int_{\psi\leq\sqrt{X+1}}\frac{|\phi_{X}|^{3}}{\psi^{\frac{3}{2}}}d\psi\\
&\leq (X+1)^{\frac{11}{4}-2\theta}\left(\int \frac{|\phi_{X}|^{2}}{\psi^{2}}\right)^{\frac{1}{2}}\left(\int \frac{|\phi_{X}|^{4}}{\psi}\right)^{\frac{1}{2}}\\
&\lesssim (X+1)^{\frac{11}{4}-2\theta}\|\phi_{X}\|_{L^{2}_{\psi}}\|\phi_{X\psi}\|_{L^{2}_{\psi}}^{2}\\
&\lesssim \|\phi_{0}\|_{E_{\text{initial}}}(X+1)^{2-2\theta}\|\phi_{X\psi}\|_{L^{2}_{\psi}}^{2}.
\end{aligned}
\end{equation*}
Therefore, $M_{4}$ is absorbed by $I_{2}$ term provided $\kappa>0$ is sufficiently small.

\textit{Estimate of $I_{4}$.}
Recall that
\begin{equation}\label{decay17}
\begin{aligned}
I_{4}&=-\int_{0}^{\infty}u_{X}\phi_{\psi\psi}\phi_{X}\frac{1}{u}d\psi (X+1)^{2-2\theta}\\
&=-\int \bar{u}_{X}\phi_{\psi\psi}\phi_{X}\frac{1}{u}(X+1)^{2-2\theta}-\int \rho_{X}\phi_{\psi\psi}\phi_{X}\frac{1}{u}(X+1)^{2-2\theta}\\
&=:I_{41}+I_{42}.
\end{aligned}
\end{equation}
By the equation \eqref{eqphi}, the estimation of $I_{42}$ is similar to $I_{3}$. For the term $I_{41}$, by the properties of Blasius profile,
\begin{equation*}
I_{41}=-\int \bar{u}_{X}\phi_{X}\left(\frac{\phi_{X}+A\phi}{u^{2}}\right)d\psi (X+1)^{2-2\theta}\geq -\int A\phi\phi_{X}\frac{\bar{u}_{X}}{u^{2}}(X+1)^{2-2\theta}.
\end{equation*}
Applying Cauchy-Schwarz inequality and \eqref{decay11}, we have
\begin{align*}
\left|\int A\phi\phi_{X}\frac{\bar{u}_{X}}{u^{2}}(X+1)^{2-2\theta}\right|&\lesssim (X+1)^{-2\theta}\left\|\frac{\phi}{\sqrt{u}}\right\|_{L^{2}_{\psi}}\left\|\frac{\phi_{X}}{\sqrt{u}}\right\|_{L^{2}_{\psi}}\\
&\leq (X+1)^{-(1+2\theta)}\left\|\frac{\phi}{\sqrt{u}}\right\|_{L^{2}_{\psi}}^{2}+\frac{\mathcal{Q}(X)}{X+1}\\
&\lesssim (X+1)^{-\frac{3}{2}}+\frac{\mathcal{Q}(X)}{X+1}.
\end{align*}

\textit{Estimate of $I_{5}$.}Utilizing the weighted embedding inequality \eqref{techweight}, we have
\begin{equation}\label{decay18}
\begin{aligned}
|I_{5}|&\leq \left\|u^{\frac{3}{2}}A_{X}\right\|_{L^{2}_{\psi}}\left\|\frac{\phi\phi_{X}}{u^{\frac{5}{2}}}\right\|_{L^{2}_{\psi}}(X+1)^{2-2\theta}\lesssim \left\|\frac{\phi\phi_{X}}{u^{\frac{5}{2}}}\right\|_{L^{2}_{\psi}}(X+1)^{\frac{1}{4}-2\theta}\\
&\leq \left\|\frac{\phi}{u^{2}}\right\|_{L^{\infty}_{\psi}}\left\|\frac{\phi_{X}}{\sqrt{u}}\right\|_{L^{2}_{\psi}}(X+1)^{\frac{1}{4}-2\theta}=\left\|\frac{\phi}{u^{2}}\right\|_{L^{\infty}_{\psi}}\sqrt{\mathcal{Q}(X)}(X+1)^{-\left(\frac{3}{4}+\theta\right)}\\
&\leq \left\|\frac{\phi}{u^{2}}\right\|_{L^{\infty}_{\psi}}^{2}(X+1)^{-\left(\frac{1}{2}+2\theta\right)}+\frac{\mathcal{Q}(X)}{X+1}.
\end{aligned}
\end{equation}
Clearly,
\begin{equation*}
 \left\|\frac{\phi}{u^{2}}\right\|_{L^{\infty}_{\psi\geq\sqrt{X+1}}}\lesssim \|\phi\|_{L^{\infty}}\lesssim (X+1)^{-\left(\frac{1}{2}-\frac{\mu}{2}\right)}.
 \end{equation*}
 By applying Lemma \ref{lm62},
 \begin{equation*}
 \begin{aligned}
 \left\|\frac{\phi}{u^{2}}\right\|_{L^{\infty}_{\psi\leq\sqrt{X+1}}}^{2}&\lesssim (X+1)\|\phi_{\psi}\|_{L^{\infty}_{\psi\leq\sqrt{X+1}}}\\
 &\lesssim (X+1)^{\mu-1}+(X+1)^{\left(\frac{1}{4}+\frac{\mu}{2}\right)}\left\|\frac{\phi_{X}}{\sqrt{u}}\right\|_{L^{2}_{\psi}}+(X+1)^{\frac{3}{2}}\left\|\frac{\phi_{X}}{\sqrt{u}}\right\|_{L^{2}_{\psi}}^{2}\\
 &=(X+1)^{\mu-1}+(X+1)^{\frac{\mu}{2}+\theta-\frac{3}{4}}\sqrt{\mathcal{Q}(X)}+(X+1)^{2\theta-\frac{1}{2}}\mathcal{Q}(X).
 \end{aligned}
 \end{equation*}
 Therefore,
 \begin{equation*}
 \begin{aligned}
 |I_{5}|&\lesssim (X+1)^{-\left(\frac{3}{2}+2\theta-\mu\right)}+(X+1)^{-\left(\frac{5}{4}+\theta-\frac{\mu}{2}\right)}\sqrt{\mathcal{Q}(X)}+\frac{\mathcal{Q}(X)}{X+1}\\
 &\lesssim (X+1)^{-\frac{3}{2}}+\frac{\mathcal{Q}(X)}{X+1}.
 \end{aligned}
 \end{equation*}
 Combining \eqref{decay13} --- \eqref{decay18}, and the interpolation inequality \eqref{decay6}, we obtain the following differential inequality for $\mathcal{Q}(X)$:
 \begin{equation}\label{decay19}
 \mathcal{Q}'(X)+C_{3}\min\left\{(X+1)^{-\frac{3-2\mu}{6-8\mu}}\mathcal{Q}(X)^{\frac{6-4\mu}{3-4\mu}},\ \mathcal{Q}(X)^{\frac{3-2\mu}{1-2\mu}}\right\}\leq C_{4}\left[(X+1)^{-\frac{3}{2}}+\frac{\mathcal{Q}(X)}{X+1}\right],
 \end{equation}
 which implies that
 \begin{equation}\label{decay20}
 \mathcal{Q}(X)\lesssim (X+1)^{\mu-\frac{1}{2}}.
 \end{equation}
 Finally, testing \eqref{decay12} by
 \begin{equation*}
\phi_{X}\frac{1}{u}(X+1)^{\frac{5}{2}-2(\theta+\mu)},
\end{equation*}
and calculate the same as \eqref{decay13}-\eqref{decay18}, we obtain
\begin{equation}\label{decay21}
\int_{0}^{X_{0}}\int_{0}^{\infty}|\phi_{X\psi}|^{2}(X+1)^{\frac{5}{2}-2(\theta+\mu)}d\psi dX\lesssim 1,\quad \forall X_{0}>0.
\end{equation}
Hence we finish the proof of Lemma \ref{lm63}.
\end{proof}
	\frenchspacing
	\bibliographystyle{plain}
	\bibliography{Prandtl_Equation}	
\end{document}